%% file: ScatteringNullPaper.tex
\documentclass[12pt]{article}
\usepackage{slashed}
\usepackage{amsmath,amsthm,amssymb,mathrsfs}
\usepackage{graphicx}
\usepackage[margin=1in]{geometry}  
\usepackage{mathtools}
\mathtoolsset{showonlyrefs=true} 

\title{Scattering from infinity for semilinear wave equations\\ satisfying the null condition or the weak null condition}
\author{Hans Lindblad and Volker Schlue}

\numberwithin{equation}{section}
\usepackage[raggedright,small,bf]{titlesec}

\newcommand{\ud}{\mathrm{d}}
\newcommand{\dm}[1]{\ud \mu_{#1}}

\newcommand{\Lb}{\underline{L}}
\newcommand{\nablas}{\nabla\!\!\!\! /}
\newcommand{\pl}{\overline{\partial}}

\newcommand{\gammac}{\mathring{\gamma}}
\newcommand{\beq}{\begin{equation}}\newcommand{\eq}{\end{equation}}
\newcommand{\beqs}{\begin{equation*}}\newcommand{\eqs}{\end{equation*}}
\def\pa{\partial}\def\pab{\bar\pa}\def\opa{\overline{\pa }}

 \def\nabb{\text{$\nabla \mkern-13mu /$\,}}
 \def\slashm{\text{$m \mkern-13mu /$\,}}
 
\theoremstyle{plain}
\newtheorem{prop}{Proposition}[section]
\newtheorem{lemma}[prop]{Lemma}
\newtheorem{remark}{Remark}[section]
\newtheorem{cor}[prop]{Corollary}

\newtheorem{theorem}{Theorem}[section]

\def\Lb{\underline{L\!}}

\newcommand*{\bigtwo}[1]{\vcenter{\hbox{\scalebox{1.4}{\ensuremath#1}}}}

\allowdisplaybreaks[1]

\begin{document}

\maketitle

\begin{abstract}
  We show global existence backwards from scattering data at infinity for semilinear wave equations satisfying the null condition or the weak null condition.  Semilinear terms satisfying the weak null condition appear in many equations in physics.
  The scattering data is given in terms of the radiation field, although in the case of the weak null condition there is an additional  logarithmic term in the asymptotic behaviour that has to be taken into account.
  Our results are sharp in the sense that the solution has the same spatial decay as the radiation field does along null infinity, which is assumed to decay at a rate that is consistent with the forward problem.
  The proof uses a higher order asymptotic expansion together with a new fractional Morawetz estimates with strong weights at infinity.
\end{abstract}

  \section{Introduction}

  In this paper we prove global existence results from scattering data at null infinity for semilinear wave equations satisfying the null condition. The data at infinity is given in terms a radiation field which is assumed to decay at a rate consistent with the forward problem. We then prove the existence of a global solution with the prescribed asymptotics that decays spatially at the same rate.
  The proof relies firstly on a fractional Morawetz estimate for the backward problem, which allows us to exploit the decay of the radiation field, and secondly on the construction of suitable approximate solutions from the scattering data.

  The results in this paper apply to semilinear systems of wave equations satisfying the weak null condition. These kind of semilinear terms appear in many equations from physics, in particular in Einstein's equations in harmonic gauge but also for Maxwell Klein-Gordon on the Minkowski background, where there are only semilinear terms.

  To illustrate the type of results proven in this paper we present first the statement for the classical wave equation:
\begin{equation}
  \Box\psi=0\qquad \text{: on }\mathbb{R}^{3+1}\label{eq:intro:wave}
\end{equation}
Given the asymptotic form of the solution,
\begin{equation}\label{eq:radiationfielddata}
\psi(t,x) \sim \frac{F_0(r-t,\omega)}{r},\qquad r=|x|,\quad\omega=\frac{x}{|x|},
\end{equation}
we wish to prescribe data in terms of its radiation field $F_0$.
For the construction of a backward solution we will assume that $F_0$ decays along null infinity,
in the sense that for some $1/2<\gamma<1$, and some $N\geq 4$:
\begin{equation}\label{eq:decayassumption:intro}
\|F_0\|_{N,\gamma-1/2}^2:=\sum_{|\alpha|+k\leq N} \int_{\mathbb{R}} \int_{\mathbb{S}^2} \big|(\langle\, q\rangle\pa_q)^k\pa_\omega^\alpha F_0(q,\omega)\big|^2\langle q\rangle^{2\gamma -1}\,  \ud S(\omega) dq <\infty
\end{equation}
where throughout this paper $\langle q\rangle = \sqrt{1+|q|^2}$.

In this setting we then have:
\begin{theorem} \label{thm:intro:hom} Given $F_0$ such that \eqref{eq:decayassumption:intro} holds for some $1/2<\gamma<1$, and $N=3+k$, for some $k\in\mathbb{N}$, the wave equation \eqref{eq:intro:wave} has a unique solution  with radiation field $F_0$ as in \eqref{eq:radiationfielddata}    which for any $1\leq s<\gamma+1/2$ satisfies
 \begin{equation}\label{eq:thm:intro:bound}
\|\psi(t,\cdot)\|_{k,s-1}:= {\sum}_{|I|\leq k} \| \langle t-r\rangle^{s-1} (Z^I \psi)(t,\cdot)\|_{\mathrm{L}^2(\mathbb{R}^3)} \lesssim \|F_0\|_{3+k,\gamma-1/2}\,.
\end{equation}
Here $Z^I$ is any combination of $|I|\leq k$ of the commuting vector fields
$Z$, 
i.e. translations $\partial_t$, $\pa_i$, rotations $x^i\pa_j-x^j\pa_i$, boosts $x^i\pa_t+t\pa_i$, and the scaling vector field $t\pa_t+x^i\pa_i$.
\end{theorem}

Note that the theorem states in particular that the ``initial data'' (the induced data at $t=0$) falls off at spatial infinity, if the radiation field decays along null infinity. In fact, the solution exhibits  almost the same spatial decay as the radiation field,
as can be seen from the following pointwise decay estimates which follow readily from the Klainerman-Sobolev inequality with weights:

\begin{cor}\label{cor:intro:decay} Let $\psi$, $\gamma$ and $F_0$ be as in as in Theorem 1.1. Then for any $\gamma^\prime<\gamma$ we have
\begin{equation}
|Z^I\psi(t,x)|\lesssim \frac{\|F_0\|_{5+|I|,\gamma-1/2}}{\langle\, t+r\rangle\langle\, t-r\rangle^{\gamma^\prime}}\,.
\end{equation}
\end{cor}

Following a brief discussion of the relation of the above theorem
to the forward problem in Section~\ref{sec:scattering:intro},
we continue our discussion of the results for semilinear wave equations satisfying the (weak) null condition in Section~\ref{sec:intro:scattering}.

A  motivation for the type of wave equations considered in this paper is the scattering problem for gravitational waves in the context of Einstein's equations. We review the asymptotic form of Einstein's equations in harmonic coordinates briefly in Section~\ref{sec:intro:einstein}.

Further inspiration for this paper is  the work of first author with Avy Soffer on the scattering problem for Klein-Gordon \cite{LS05,LS06}, and the work of H\"ormander \cite{H97}.

\subsection{Relation to the forward problem}
\label{sec:scattering:intro}

As already mentioned, to capture the decay of the  the radiation field $F_0$, we assume that \eqref{eq:decayassumption:intro} holds for some $1/2<\gamma<1$, and $N\geq 4$,  which implies by a Sobolev inequality that for $N'=N-3$,
\begin{equation}\label{eq:decayassumptioninfty}
\|F_0\|_{N',\infty,\gamma}:=\sum_{|\alpha|+k\leq N'} \sup_{q\in\mathbb{R}}\sup_{\omega\in\mathbb{S}^2} \big|(\langle\, q\rangle\pa_q)^k\pa_\omega^\alpha F_0(q,\omega)\big|\langle q\rangle^{\gamma}\,   <\infty\,.
\end{equation}
So at a pointwise level the assumption means that the radiation field decays like $|F_0|\lesssim 1/\langle q\rangle^\gamma$, and $| F_0'| \lesssim 1/\langle q\rangle^{\gamma+1}$, etc.~for some $\gamma>1/2$.

Recall that solutions to the Cauchy problem for \eqref{eq:intro:wave} with smooth compactly supported data at $t=0$ satisfy
\begin{equation}
  \label{eq:ho:decay}
  |\pa^\alpha Z^I \psi| \leq \frac{C_{\alpha,I}}{\langle t+r\rangle \langle t-r\rangle}\,;
\end{equation}
see for instance Theorem~6.2.1 in~\cite{H97}. Thus there are solutions to the forward problem which at least in terms of the rate satisfy the assumptions with $\gamma=1$. However, energy methods combined with the Klainerman-Sobolev inequality only give \emph{forward} asymptotics with $\gamma=1/2$; see for instance  Remark following Proposition~6.5.1 in \cite{H97}. Our Theorem~\ref{thm:intro:hom} shows that energy methods \emph{can} be used in the \emph{backwards} problem to cover the range $1/2<\gamma<1$.

While for the forward problem  it is difficult to obtain interior decay, the main difficulty for the backwards problem is exterior decay.  This is evident already from the weighted energy estimates of Section~6 in \cite{LR10}, which in the simplest case for the homogeneous wave equation say that the weighted energy
\begin{equation}
  \int |\partial \phi|^2(t,x) w(|x|-t) \ud x
\end{equation}
is decreasing in $t$ for any weight $w(q)\geq 0$, which in increasing towards the exterior, $w'(q)\geq 0$. Thus these estimates can be used to obtain exterior decay for the forward problem. Interior decay typically relies  on the fundamental solution, see for instance Section~5 of \cite{L17},  which loses regularity.

Conversely for the backwards problem any positive weight which is \emph{increasing towards the interior}, $w'(q)\leq 0$, gives an energy which is decreasing towards the past. In particular one may use a weight  $w(q)=(1+|q|)^{\gamma}$ for any $\gamma> 0$ when $q<0$.
We have included the analogue of these weighted estimates for the backwards problem in Section~\ref{sec:weighted}, and use them to prove global existence backwards for the wave equation with null condition in Section~\ref{sec:classical:null}. However, this alone does not give exterior decay.

The exterior decay for the backwards problem is obtained in this paper with the fractional Morawetz estimate  in Section~\ref{sec:intro:morawetz}.
Since in the scattering problem the radiation field is \emph{given},
we may subtract the known leading order term in \eqref{eq:radiationfielddata} and essentially assume fast interior decay. In other words, the weighted estimate of Theorem~\ref{thm:morawetz:intro} is applied to a remainder that results after subtracting a suitable higher order expansion.
However,  the use of approximate solutions comes at the  cost of a loss of derivatives.

We do not attempt to formulate a scattering theory in weighted Sobolev spaces, but in this language our results correspond to  the \emph{existence and uniqueness} of scattering states, but \emph{not} their asymptotic completeness.

\subsection{Scattering with spatial decay for semilinear model problems}\label{sec:intro:scattering}

In Section~\ref{sec:intro:einstein} we will see that  a suitable simplified semilinear model of Einstein's equations in wave coordinates with a similar weak null structure is
\beq
  \Box \psi = Q(\pa \psi,\pa \varphi),\qquad
  \Box \varphi = \bigl( \partial_t\psi \bigr)^2\,. \label{eq:weaknullcondintro}
\eq

We separate this model further and solve the scattering problem in order of increasing complexity for the following equations: \emph{homogeneous wave equation}, \emph{scalar wave equation with classical null condition}, \emph{systems of wave equations with null condition}, and finally a simple model with \emph{weak null condition}.

\subsubsection{Homogeneous wave equation}
\label{sec:hom:intro}

The classical homogeneous wave equation  \eqref{eq:intro:wave} is treated in  \textbf{Section~\ref{sec:scattering}}.

Here we are given scattering data in the form of a radiation field $F_0(q,\omega)$, and the solution $\psi$ is constructed by taking the data for
$\psi$ when $t=T$ to be the restriction of the following \emph{approximate solution} at $t=T$, and then taking $T\to\infty$.
The approximate solution $\psi_0$ is defined in terms of the radiation field $F_0$, and is supported in the wave zone, away from the origin,
\begin{equation}\label{eq:linearexpansionpsi0}
\psi_{0}(t,x):=\frac{F_0(r-t,\omega)}{r}\chi\big(\tfrac{\langle\, t-r\rangle}{r}\big),\qquad r=|x|\,,\ \omega =\frac{x}{|x|}\,,
\end{equation}
where $\chi$ is a smooth decreasing function with $\text{supp}\chi\subset [0,1/4]$.

In fact, for \eqref{eq:intro:wave} as well as for the equations \eqref{eq:nullcondition}, and \eqref{eq:nullcondition:system:intro} below, we use a \emph{second order approximate solution},
\begin{equation}\label{eq:intro:second:approx}
  \psi_{01}:=\psi_0+\psi_1\qquad \psi_1:=\frac{F_1(r-t,\omega)}{r^2}\chi\big(\tfrac{\langle\, t-r\rangle}{r}\big)\,,
\end{equation}
where $F_1$ is determined by solving an ODE in $q$, i.e.~along null infinity, whose precise form depends on the equation. For example, for \eqref{eq:intro:wave} this reads
\begin{equation}\label{eq:intro:ODE:hom}
  2\partial_q F_1(q,\omega)=\triangle_\omega F_0(q,\omega)\,,\qquad F_1(0,\omega)=0\,.
\end{equation}

The main results are \textbf{Theorem~\ref{thm:intro:hom}} above, and \textbf{Proposition~\ref{prop:scattering:homogeneous}} below, which includes precise estimates for the remainder, which decays faster, namely
\begin{equation}
  \label{eq:intro:remainder:pointwise}
  |Z^I(\psi-\psi_{01})|\lesssim \frac{\|F_0\|_{5+|I|,\gamma-1/2}}{\langle t+r\rangle\langle t\rangle^\gamma}\,.
\end{equation}

Our scattering constructions also allow for the inclusion of a \emph{``mass term'' in the exterior}, namely of a term $M/r$ in the solution for $t<r$.

\begin{remark}
  Note that the solutions arising from this data have the property that $\overline{F_1}=0$, which coincides with the Newman Penrose constant \cite{AAG18a,NP68}.
\end{remark}

\begin{remark}
   A representation of the solution to the standard wave equation as a formal power series in $r^{-1}$ --- of which we only use here the first two terms as an approximation --- was first suggested in the analytic setting by Friedlander \cite{F62} where also the recurrence relations for the higher order radiation fields $F_i$, $i\in\mathbb{N}$, can be found.
   Asymptotic expansions of forward solutions to the wave equation have also been studied in \cite{BVW15}.
\end{remark}

\subsubsection{Classical null condition model}
\label{sec:null:intro}

The first nonlinear model is a scalar  wave equation satisfying the classical null condition
\begin{equation}\label{eq:nullcondition}
 \Box \phi = Q(\pa \phi,\pa \phi)\,.
\end{equation}

We present two scattering constructions for this model:

\smallskip
First, in \textbf{Section~\ref{sec:classical:null}} we show the existence of a global solution to \eqref{eq:nullcondition} which scatters to a given linear solution. More precisely, here  the  asymptotic data is given
in the form  $\phi\sim \phi_0$ as $t\to\infty$ of  a given solution to the homogeneous equation $\Box \phi_0=0$ with small (generalised) energy.  The main result here is \textbf{Proposition~\ref{prop:wave:null}}.  The approach  here is based on the weighted space-time energy estimates of Section~\ref{sec:weighted}, and does \emph{not} use the fractional Morawetz estimate. Moreover a strategy is prepared for the boundary terms that will used in the proof of the  more refined results in Section~\ref{sec:classical:null:revisited}.

  \smallskip
Then second, in \textbf{Section~\ref{sec:classical:null:revisited}} we give scattering data purely in terms of a radiation field. In fact, we do this for a \emph{system} of wave equations satisfying the null condition:
  \begin{equation}\label{eq:nullcondition:system:intro}
    \Box\psi=Q(\pa \psi,\pa \phi)\qquad \Box \phi = \widetilde{Q}(\pa \psi,\pa \phi)\,.
  \end{equation}
  The main result here is \textbf{Proposition~\ref{prop:nullcond:revisited}}, which is proven under the following smallness assumption on the scattering data for $\psi\sim F_0(r-t,\omega)/r$ and $\phi\sim G_0(r-t,\omega)/r$:
  \begin{equation}
    \label{eq:intro:smallness}
    \|F_0\|_{N,\gamma-1/2}+\|G_0\|_{N,\gamma-1/2}<\varepsilon
  \end{equation}
  The proof relies on the new fractional Morawetz estimates of Section~\ref{sec:conformal:energy}, and the scattering result can be summarised as follows:

\begin{theorem}\label{thm:intro:nullcond}
  Given $F_0$, $G_0$, satisfying \eqref{eq:intro:smallness} for some $1/2<\gamma<1$, and $N\geq 6$, and for $\varepsilon>0$ sufficiently small, the system of wave equations \eqref{eq:nullcondition:system:intro} satisfying the null condition has a global solution $(\psi,\phi)$  with radiation fields $F_0$, $G_0$, respectively, and the property that for all $1\leq s<1/2+\gamma$, and $k\leq N-6$,
\begin{equation}
  \|\psi(t,\cdot)\|_{k+1,s-1}+\|\phi(t,\cdot)\|_{k+1,s-1}\lesssim \|F_0\|_{k+6,\gamma-1/2}+\|G_0\|_{k+6,\gamma-1/2}\,.
\end{equation}
\end{theorem}

Here the solutions are also constructed with the help of an approximate solution of the form \eqref{eq:intro:second:approx}, where in contrast to \eqref{eq:intro:ODE:hom}  $F_1$, and $G_1$ now satisfy the nonlinear ODEs:
\begin{subequations}
\begin{align}
2\pa_qF_1(q,\omega)\!&= \triangle_\omega F_0(q,\omega)
-Q^0(F_0,F_0^\prime,\pa_\omega F_0,G_0,G_0^\prime,\pa_\omega G_0),\quad F_1(0,\omega)=0\\
2\pa_q G_1(q,\omega)\!&= \triangle_\omega G_0(q,\omega)-\widetilde{Q}^0(F_0,F_0^\prime,\pa_\omega F_0,G_0,G_0^\prime,\pa_\omega G_0)\,,\quad G_1(0,\omega)=0
\end{align}
\end{subequations}
where $Q^0$, $\widetilde{Q}^0$  are  bilinear forms in the radiation fields $F_0$, and $G_0$, and its first derivatives.

\subsubsection{Weak null condition model}
\label{sec:weak:intro}

The weak null condition model we consider is
\begin{subequations}
\begin{align}
  \Box \psi &= Q(\pa \psi,\pa \psi) \label{eq:linearequation}\\
  \Box \varphi &= \bigl( \partial_t\psi \bigr)^2 \label{eq:weaknullcond}\,.
\end{align}
\end{subequations}

In \textbf{Section~\ref{sec:weak:null:simple}} we approach the scattering problem for this system as follows: While  $\psi$  has a radiation field $F_0$, we prescribe a radiation field for $\varphi$ only after subtracting a suitable approximate solution which picks up the source term. In fact, we  first define $\Psi_0$ to be the solution of
\begin{equation}\label{eq:intro:psi:source}
\Box \Psi_0=(\psi_0^\prime)^2,\qquad\text{where}\qquad
\psi_{0}^\prime=-\frac{F_0^\prime(r-t,\omega)}{r}\chi\big(\tfrac{\langle\, t-r\rangle}{r}\big),
\end{equation}
with vanishing data at $t=-\infty$.

The main result for this system is given in \textbf{Proposition~\ref{prop:weak:null:system}},
which establishes the existence of a solution of the form $(\psi=v+\psi_{01}+\psi_e,\,\varphi=w+\Psi_{0}+\varphi_{01})$ (where $\psi_{01}$ and $\varphi_{01}$ are approximate solutions associated to radiation fields $F_i$, and $G_i$, $i=0,1$) and decay results for the remainder terms $v$, and $w$. In Section~\ref{sec:weak:null:simple} we thus estimate also the equation \eqref{eq:weaknullcond} for $\varphi$ in the same norms that we use for $\psi$.

Finally, from \cite{L90a,L17} we have a formula for the solution to \eqref{eq:intro:psi:source},
\begin{gather}
  \Psi_0=\Phi^2[n]\,,\qquad n(q,\omega)= F_0^\prime(q,\omega)^2\,,\\
\Phi^2[n](t,r\omega)=\int_{r-t}^{\infty} \frac{1}{4\pi}\int_{\bold{S}^2}{\frac{
 n({q},{\sigma})\chi\big(\tfrac{\langle\,{q}\,\rangle}{\rho}\big)^2  dS({\sigma})d {q}}{t-r+{q}+r\big(1-\langle\, \omega,{\sigma}\rangle\big)}\,
}\,,
\end{gather}
discussed in Section~\ref{sec:sources},
and an estimate for the solution of the form
\beq
| \Phi^2[n](t,r\omega)|\lesssim \frac{1}{2r}\ln{\Big(\frac{\langle \, t+r\rangle}{\langle\, t-r\rangle}\Big)}\,.
\eq

\begin{remark}
These asymptotics --- which in addition to the asymptotic behaviour \eqref{eq:radiationfielddata} in the wave zone display a logarithmic divergence at the lightcone and slower decay in the interior --- are the starting point of a scattering construction based directly on the asymptotics in the interior which we present separately in a forthcoming paper.
\end{remark}

\subsection{The fractional conformal Morawetz energy estimate from infinity}\label{sec:intro:morawetz}
The main technical tool of the paper  used to obtain the right spatial decay is a
new energy estimate from infinity with strong weights,  presented in \textbf{Section~\ref{sec:conformal:energy}}:
\begin{theorem}\label{thm:morawetz:intro}
Let $t_1\leq t_2$,  then for $s\geq 1$
\begin{equation}\label{eq:energyestimateintro}
\|\phi(t_1,\cdot)\|_{1,+,s-1}\lesssim \|\phi(t_2,\cdot)\|_{1,+,s-1}+\int_{t_1}^{t_2} \|\langle t+r\rangle^s \Box \phi(t,\cdot)\|_{L^2_x} \ud t
\end{equation}
provided that  $\lim_{r\to\infty}\sup_{\omega\in\mathbb{S}^2} r^{s+1/2}|\phi(t,r\omega)|=0$.
Here
\begin{multline}\label{eq:norm:sintro}
    \|\phi(t,\cdot)\|_{1,+,s-1}^2 := \\ \!\int_{\mathbb{R}^3}\! \langle\, t\!+r\rangle^{2s} \Big( \big|(\pa_t\!+\pa_r)(r \phi)\big|^2\! + |\nabb (r \phi)|^2\Big)
 +  \langle\, t\!-\!r\rangle^{2s}\Big(\big|(\pa_t\!-\!\pa_r)(r \phi)\big|^2\! + \phi^2\!+\frac{(r\phi)^2}{\langle \, t\!-\!r\rangle^2} \Big)\, \frac{dx}{r^2}
\end{multline}
is a norm such that
\begin{equation}\label{eq:normestimateintro}
\|\phi(t,\cdot)\|_{1,s-1}:={\sum}_{|I|\leq 1} \| \langle t-r\rangle^{s-1} (Z^I \phi)(t,\cdot)\|_{L^2}\lesssim \|\phi(t,\cdot)\|_{1,+,s-1}\,.
\end{equation}
\end{theorem}

If $s=1$ this reduces to the classical conformal Morawetz  estimate which holds in both the forward  direction $t_1>t_2$ (with the boundaries of integration switched) and the backward direction $t_1<t_2$. We remark that in the forward direction $t_1>t_2$ it was proven in Lindblad-Sterbenz \cite{LS06b}
that this estimate holds if $1/2\leq s\leq 1$.
The case $s>1$ will here be used in the backwards direction $t_1<t_2$,
and applied to the remainder $v$ of a solution $\psi=\psi_0+v$, where $\psi_0$ is a suitable approximate solution defined in terms of the scattering data.
The extra spatial decay we see in the backwards solution stems from the weight in $\langle t-r\rangle$ in the norm $\|\cdot\|_{1,s-1}$.

\subsection{Asymptotic form of Einstein's equation in wave coordinates}
\label{sec:intro:einstein}
Our main motivation for studying semilinear equations satisfying the weak null condition is that these types of semilinear terms show up in Einstein's equations in harmonic coordinates,  see \cite{LR05,LR10}. Moreover, for Einstein's equations the detailed asymptotic behaviour has been studied in \cite{L17}. The techniques in \cite{L17} would work for the forward problems for the equations that we are studying here and give asymptotics with the same spatial decay that we are prescribing at infinity. 

 In fact,
in the context of the global existence problem for Einstein's equations in harmonic coordinates  with asymptotically flat data close to Minkowski space it was shown in \cite{L17} that the components of the metric $h=g-m$ in suitable coordinates $(t,x^\ast)$, where $m$ denotes the Minkowski metric, satisfy asymptotically the semilinear system
\begin{equation}\label{eq:EinsteinWaveSemilinearModel}
\Box^\ast \,  h_{\mu\nu} = F_{\mu\nu} (m) (\pa h, \pa h)
\end{equation}
where $\Box^\ast$ is a constant coefficient wave operator in $(t,x^\ast)$ coordinates, and  $F_{\mu\nu}(m)(\pa h,\pa h)$ is a  sum of classical null forms which we denote by
$Q_{\mu\nu}(\pa h,\pa h)$ and $P(\pa h, \pa h)$ where
\beq
 P(D,E)= D_{\alpha}^{\,\, \, \alpha}
E_{\beta}^{\,\, \, \beta}/4- D^{\alpha\beta}
E_{\alpha\beta}/2.
\eq

Indeed Einstein's equations in wave coordinates do not satisfy the {\it null condition} \cite{C00},
which was identified by Klainerman as a condition on the type of nonlinearity that guarantees small data global existence  \cite{K82,K86,C86}.
Nonetheless Lindblad and Rodnianski  showed in \cite{LR03} that they satisfy a {\it weak null condition} and in
\cite{LR05,LR10} used it to prove global existence.

In fact   in \cite{LR03} it was  observed that \eqref{eq:EinsteinWaveSemilinearModel}  has a weak null structure in a  null frame $(L,\Lb\,; E_1,E_2)$ and reduces  to
 \beq\label{eq:einsteinfirstapproximationintro}
{\Box}h_{\mu\nu}\sim L_\mu L_\nu  P_{S}(\pa_q h,\pa_q h)\,,
\eq
where $P_S(D,E)=- \widehat{D}^{AB} \widehat{E}_{AB}/2$ only involves the angular components.

Moreover in \cite{L17} the asymptotic behaviour of all metric components is derived. In fact,
for asymptotically flat initial data for Einstein equations of the form
\begin{equation}\label{eq:asymptoticallyflatdata}
g_{ij}|_{t=0}=(1+M r^{-1})\, \delta_{ij} + o(r^{-1-\gamma}),\quad
\pa_t g_{ij}|_{t=0}=o(r^{-2-\gamma}),\quad 0<\gamma<1,\quad M>0,
\end{equation}
it was shown that asymptotically the metric is given by $g=m+h$, where $h$ satisfies
\begin{equation}\label{eq:Einsteinasymptotics}
h_{\mu\nu}(t,r\omega)\sim \frac{H_{\mu\nu}(r^*\!-t,\omega)}{\langle t+r\rangle }+\frac{K_{\mu\nu}\big(\tfrac{\langle t+r^*\rangle}{\langle r^*\!-t\rangle},\omega,r^*\!-t\big)}{\langle t+r\rangle},
\quad r^*\!\sim r+M\ln{r},\quad \omega\!=\!\frac{x}{|x|}\,.
\end{equation}
In other words for some components $h_{\mu\nu}$ has a radiation field $H_{\mu\nu}$ which is concentrated close to the outgoing light cones, and $K_{\mu\nu}$ appears for the transversal components to the lightcone and to leading order is homogeneous of degree $0$ with a logarithmic singularity at the light cone:
 \begin{align}\label{eq:Hdecay}
 |H(q^*\!,\omega)|&\lesssim (1\!+| q^*|)^{-\gamma^\prime},\qquad \gamma^\prime<\gamma\,,\\
 |K(s,\omega,q)|&\lesssim \ln{s}
 \end{align}
What is particularly important is the additional decay of $H$ in
\eqref{eq:Hdecay}. Here $\gamma^\prime$ is any number less than $\gamma$ in the asymptotic flatness condition \eqref{eq:asymptoticallyflatdata}. In the interior of the light cone $r^*<t$ this was very difficult
to prove for the forward problem, see \cite{L17}.

This paper is motivated by the backward problem: we simply give the radiation fields $H$ satisfying \eqref{eq:Hdecay} and we want to prove that there is a solution to Einstein's equations with the given asymptotics
\eqref{eq:Einsteinasymptotics} which initially satisfy the asymptotic flatness condition \eqref{eq:asymptoticallyflatdata} for any $\gamma<\gamma^\prime$. For the backward problem this is particularly difficult to prove in the exterior of the light cone $r^*>t$.

\subsection{Further relations to other work}

 The radiation field $F_0$ takes its name from Friedlander who in a series of papers \cite{F62,F64,F67,F80} defined the notion and first studied  existence, uniqueness and completeness properties of solutions to wave equations for a given radiation field. Along these lines  Baskin and S\'{a} Baretto have studied the scattering problem for critical semilinear wave equations in \cite{BB15}. Related to the study of the forward map  asymptotic expansions of the radiation field have been established in some generality in \cite{BVW15}.

The scattering problem for Einstein's equations \emph{in higher dimensions} close to Minkowski space was studied by Wang \cite{W14}.
However, the three dimensional case is more delicate and requires a description of the asymptotic behaviour of solutions of the forward problem given recently by Lindblad in \cite{L17}.

A scattering theory  for the wave equation on \emph{black hole spacetimes} in \emph{non-degenerate}  energy spaces is established in \cite{DRS18} (see references therein for previous work on this topic), and in the extremal case in \cite{AAG19}. Here already the existence of finite energy solutions (\emph{without} weights) to the backwards problem is nontrivial due to the redshift and superradiance \cite{DRS16}. The radiation field on Schwarzschild was also defined in \cite{BB15}. For \emph{exponentially} decaying scattering data a scattering construction for the Einstein equations was given in \cite{DHR14}.

The global existence problems from scattering data at infinity considered in this paper are \emph{well-posed}. For a study of uniqueness properties of solutions with a given radiation field in an \emph{ill-posed} setting see also \cite{ASS16}.

In terms of the \emph{method} the weighted estimates in this paper for the backward problem are derived by using a multiplier of the form
\begin{equation}
  \label{eq:intro:K:s}
K_s=\langle t+r\rangle^{s}(\partial_t+\partial_r)+\langle t-r\rangle^{s}(\partial_t-\partial_r)\,,
\end{equation}
which is a generalisation of the classical conformal Morawetz vectorfield:
\begin{equation}
  \label{eq:intro:K} K=(t+r)^2\bigl(\partial_t+\partial_r\bigr)+(t-r)^2\bigl(\partial_t-\partial_r\bigr)\,.
\end{equation}
The latter is intimately tied to the geometry of Minkowski space.
In view of applications to \emph{black hole spacetimes}, Dafermos and Rodnianski introduced in \cite{DR09} estimates for the forward problem of the linear  wave equation on asymptotically flat spacetimes which as opposed to the conformal Morawetz estimate associated to \eqref{eq:intro:K}  avoid weights in $t$.
In fact the \emph{ $r^p$-weighted estimates} are derived in the wave zone from multipliers of the form $r^p(\partial_t+\partial_r)$. For extensions of this method, and the resulting hierarchy of estimates see \cite{S13, M16} as well as \cite{AAG18a,AAG18b}.
In the context of this paper the weights in $\langle t-r\rangle$
in \eqref{eq:intro:K:s}  are essential, and not dropped from \eqref{eq:intro:K}.

Recently  the $r^p$-weighted estimates were also applied  to prove global existence for a large class of wave equations satisfying the \emph{weak null condition} in \cite{K18}. The forward asymptotics of a class of quasilinear wave equations satisfying the weak null condition are analysed in \cite{DP20}.

In addition to the original proof of the stability of Minkowski space \cite{CK93}, and the proof in wave coordinates \cite{LR05,LR10}, another proof in generalised harmonic gauge was given in \cite{HV17}. Therein a related set of vectorfields to \eqref{eq:intro:K} of the form $r(\partial_t+\partial_r)$, $(r-t)(\partial_t-\partial_r)$ are used by Hintz and Vasy to derive weighted spacetime estimates, and are used to establish a \emph{polyhomogeneous} expansion of the metric near infinity.

\begin{description}
\item[Acknowledgements.] H.L. is supported in part by NSF grant DMS--1500925. V.S.~would like to acknowledge the support of ERC consolidator Grant 725589 EPGR. We would like to thank the \emph{Institut Henri Poincar\'e} for their hospitality in the fall semester 2016, and the \emph{Institut Mittag Leffler} in the fall semester 2019.
  We especially would like to thank Lili He for a careful reading of an earlier version of the manuscript which led to several corrections.
\end{description}

\section{Energy estimates}

\subsection{Weighted space-time energy estimates from infinity}
\label{sec:weighted}

We start by deriving several weighted energy estimates for the equation
\begin{equation}\label{eq:wave:F}
  \Box \phi = F
\end{equation}
using the energy momentum tensor
\begin{equation}
  T_{\mu\nu}[\phi]=\partial_\mu \phi\,\partial_\nu \phi - \frac{1}{2} m_{\mu\nu}\, \partial^\alpha\phi\partial_\alpha\phi\,.
\end{equation}

For any given $t_2>t_1$, $R>(t_2-t_1)$ let us denote by
\begin{equation*}
  \mathcal{D}=\bigcup_{t_1\leq \tau\leq t_2}\Sigma_\tau^{R+(\tau-t_2)}
\end{equation*}
where $\Sigma_\tau^R=\{ (t,x):t=\tau, |x|\leq R_+ \}$. Note that $\partial \mathcal{D}=\Sigma_{t_1}^{R-(t_2-t_1)}\cup\Sigma_{t_2}^{R}\cup C_R$, where $C_R=\{(t,x):t-|x|=t_2-R,t_1<t<t_2\}$ is a (truncated) outgoing null hypersurface. $\mathcal{D}$ is the backward domain of dependence of a ball of radius $R$ in $t=t_2$, see Fig.~\ref{fig:SigmaR}.

\begin{figure}[tb]
  \centering
  \includegraphics{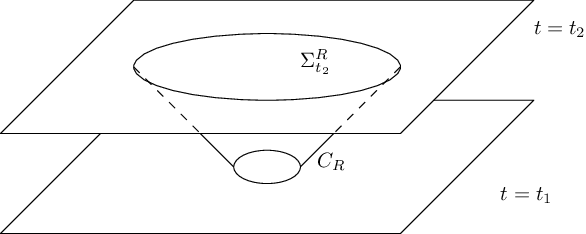}
  \caption{Depiction of domain $\mathcal{D}$.}
  \label{fig:SigmaR}
\end{figure}

In the following we will also denote by
\begin{align}
  |\partial \phi|^2 :=&   (\partial_t\phi)^2+|\nabla\phi|^2\,,\\
    |\pl \phi|^2 :=& (\partial_t\phi+\partial_r\phi)^2+\vert\nablas \phi\rvert^2\,,\qquad |\nablas\phi|^2:=|\nabla\phi|^2-\bigl\lvert\frac{x}{|x|}\cdot\nabla\phi\bigr\rvert^2\,.
  \end{align}
 Thus $\pl$ denotes the derivatives tangential to the lightcone.

\begin{prop} \label{prop:w}
Let $\phi$ be a solution to \eqref{eq:wave:F} on $\mathcal{D}$.
Then
\begin{multline}
  \int_{\Sigma_{t_2}^{R}}|\partial \phi|^2 \, w(q)\, \ud x
  +  \int_{\mathcal{D}}  2F \partial_t\phi \, w(q)\,\ud x\ud t+  \int_{\mathcal{D}}  |\pl \phi|^2\, w^\prime (q)\,\ud x\ud t\\
  =\int_{\Sigma_{t_1}^{R-(t_2-t_1)}}  |\partial \phi|^2 \, w(q)\,\ud x+\int_{t_1}^{t_2}\int_{\partial \Sigma_{t}^{R-(t_2-t)}}|\pl \phi|^2\, w(q)\, dS \ud t
\end{multline}
where $w(q)$ is an arbitrary function of $q=r-t$.
\end{prop}

\begin{proof}
  Set $W=w(q)\partial_t$ where $w$ is only a function of $q=r-t$, and
  \begin{equation*}
    J_\mu = T_{\mu\nu} W^\nu
  \end{equation*}
then
\begin{equation*}
  \nabla^\mu J_\mu= F \,(W\phi) + T_{\mu\nu}\pi^{\mu\nu}
\end{equation*}
where
\begin{equation*}
  \pi_{\mu\nu}=\frac{1}{2}\mathcal{L}_W m_{\mu\nu}
\end{equation*}
It makes sense to write this out in a null frame $(L=\partial_t+\partial_r,\Lb=\partial_t-\partial_r,e_A:A=1,2)$
\begin{equation*}
  T_{\mu\nu}\pi^{\mu\nu}= 2 m^{L\Lb}m^{\Lb L} T(L,\Lb)\pi(\Lb,L)+m^{L\Lb}m^{L\Lb}T(L,L)\pi(\Lb,\Lb)
\end{equation*}
where we used that the only non-vanishing components of the deformation tensor of $W$ are
\begin{gather*}
  \pi(\Lb,L)=\frac{1}{2}(\Lb w) m(\partial_t,L)=\partial_q w,\\
  \pi(\Lb,\Lb)=(\Lb w)m(\partial_t,\Lb)=2\partial_q w .
\end{gather*}
Thus
\begin{equation*}
  T_{\mu\nu}\pi^{\mu\nu}= \frac{1}{2}(\partial_q w)|\nablas\phi|^2+\frac{1}{2}(\partial_q w) (L\phi)^2\,.
\end{equation*}
Finally the boundary terms are
\begin{gather*}
  T(W,\partial_t)=w(q) T(\partial_t,\partial_t)=\frac{1}{2} w(q) \bigl( (\partial_t\phi)^2+|\nabla\phi|^2\bigr),\\
  T(W,L)=w(q)T(\partial_t,L)=\frac{1}{2}w(q)\bigl((L\phi)^2+|\nablas\phi|^2\bigr)
\end{gather*}
and thus we obtain by Stokes theorem
\begin{multline*}
  \int_{\mathcal{D}} F(W\phi)+ \frac{1}{2}(\partial_q w)|\nablas\phi|^2+\frac{1}{2}(\partial_q w) (L\phi)^2\,  \ud x\ud t=\\=\int_{\mathcal{D}}\nabla^\mu J_\mu\, \ud x \ud t
  = \int_{\mathcal{D}}d{}^\ast J=\int_{\partial\mathcal{D}}{}^\ast J=\int_{\Sigma_{t_2}^{R-(t_2-t_1)}} J^0-\int_{\Sigma_{t_1}^R} J^0+\int_{C_R}{}^\ast J \\
  = -\int_{\Sigma_{t_2}^{R}} T(W,\partial_t)+\int_{\Sigma_{t_1}^{R-(t_2-t_1)}} T(W,\partial_t)+\int_{t_1}^{t_2}\int_{\mathbb{S}^2} T(W,L)\, \dm{\gammac}\ud t,
\end{multline*}
which completes the proof.
\end{proof}

The following two corollaries correspond to different choices of the function $w$.

\begin{cor}\label{cor:w:D}
  Let $R>0$, $t_2>t_1$, and  $\phi$ be a solution to \eqref{eq:wave:F} on $\mathcal{D}$. Then for all $\mu>0$,
  \begin{multline}
    \lVert \partial \phi \rVert_{\mathrm{L}^2(\Sigma_{t_1}^{R-(t_2-t_1)})}+\Bigl(\int_{\mathcal{D}} |\pl \phi|^2 \frac{\mu \,\ud x\ud t/4}{(1+|q|)^{1+2\mu}} \Bigr)^{\!1/2} +\Bigl(\sup_{R'\leq R} \int_{C_{R'}}|\pl \phi|^2\, dS \ud t \Bigr)^{\!1/2}\\
    \lesssim  \lVert \partial \phi \rVert_{L^2(\Sigma_{t_2}^R)} + \int_{t_1}^{t_2}  \lVert F \rVert_{L^2(\Sigma_t^{R+(t-t_2)})}\, \ud t.
      \end{multline}
\end{cor}

\begin{proof}
  Let us choose $w=w_0$, where
  \begin{equation}\label{eq:w}
    w_0(q) = \begin{cases} 1+(1+q)^{-2\mu}, & q>0, \\ 2+2\mu \int_q^0(1+|s|)^{-1-2\mu}\ud s, & q<0,
    \end{cases}
 \end{equation}
 then $1\leq w_0(q)\leq 3$, and in both cases
 \begin{equation}
   w_0^\prime(q) = - 2\mu (1+|q|)^{-1-2\mu}
 \end{equation}
 Note also that since $\mathcal{D}$ lies to the future of $C_R$, we have $q<R-t_2$.

 Therefore by Prop.~\ref{prop:w} we obtain
\begin{multline*}
  \int_{\Sigma_{t_1}^{R-(t_2-t_1)}} \frac{1}{2} w_0(q) |\partial \phi|^2\ud x+\int_{C_R} \frac{1}{2}w_0(R-t_2)|\pl \phi|^2 \ud S\ud t
  +  \int_{\mathcal{D}} \frac{\mu}{4}\frac{ |\pl \phi|^2}{(1+|q|)^{1+2\mu}}\, \ud x \ud t\\\leq
    \int_{\Sigma_{t_2}^{R}} \frac{1}{2} w_0(q) |\partial \phi|^2\ud x
  +  \int_{\mathcal{D}} F w_0(q) \partial_t\phi\, \ud x\ud t .
\end{multline*}
Also by the standard energy estimate (corresponding to choosing $w\equiv 1$ in Prop.~\ref{prop:w})
\begin{multline*}
  \Bigl\lvert  \int_{\mathcal{D}} F  \partial_t\phi\, \ud x\ud t \Bigr\rvert \leq  \int_{t_1}^{t_2} \lVert \partial_t \phi \rVert_{L^2(\Sigma_t^{R+(t-t_2)})} \lVert F \rVert_{L^2(\Sigma_t^{R+(t-t_2)})} \ud t\\
  \leq  C \Bigl( \lVert \partial \phi \rVert_{L^2(\Sigma_{t_2}^R)} + \int_{t_1}^{t_2}  \lVert F \rVert_{L^2(\Sigma_t^{R+(t-t_2)})} \ud t \Bigr)^2.
\end{multline*}
The two inequalities combined give the estimate for the second term in the statement of the Corollary, and the standard energy estimate the first term.
\end{proof}

In the following Corollary we distinguish between the exterior and interior, and introduce the notation:
\begin{gather}
  \Sigma_t^i=\bigl\{(t,x)\in\Sigma_t:|x|\leq t\bigr\},\qquad\Sigma_t^e=\Sigma_t\setminus\Sigma_t^i,\\
    \mathcal{D}^i = \bigl\{(t,x)\in\mathcal{D}: |x|\leq t\bigr\},\qquad \mathcal{D}^e = \mathcal{D}\setminus \mathcal{D}^i .
\end{gather}

\begin{cor}\label{cr:weightsspacetimeenergy}
  For any $\mu\geq 0$ and $\gamma\geq-\frac{1}{2}$,
we have
\begin{multline*}
  \int_{\Sigma_{t_1}}\!\!|\partial \phi|^2 w_-^\gamma\ud x +
  \frac{\mu}{2}\int_{\mathcal{D}^e}\frac{ |\pl \phi|^2 }{(1\!+|q|)^{1+2\mu}}\ud x\ud t+\frac{1\!+2\gamma}{4}\int_{\mathcal{D}^i} (1\!+q_-)^{2\gamma} |\pl \phi|^2  \ud x\ud t\\
  +\sup_{R'\leq R} \int_{C_{R'}}\!\!|\pl \phi|^2\,w_-^\gamma dS \ud t   \leq \int_{\Sigma_{t_2}}  |\partial \phi|^2 w_-^\gamma\ud x
  +  2\int_{\mathcal{D}} F \partial_t\phi w_-^\gamma\ud x\ud t
\end{multline*}
where $w_-^\gamma=(1+q_-)^{1+2\gamma}$, $q_-:=-\min\{q,0\}$.
\end{cor}
\begin{proof}
For  $\mu\geq 0$ and $\gamma\geq -\frac{1}{2}$ let
  \begin{equation}\label{eq:w}
    w(q) = \begin{cases} 1+(1+|q|)^{-2\mu}, & q>0, \\ 1+(1+|q|)^{1+2\gamma}, & q<0,
    \end{cases}
  \end{equation}
  in which case
  \begin{equation}
    w^\prime(q) =
   \begin{cases}
      - 2\mu (1+|q|)^{-1-2\mu}, & q>0, \\
      - (1\!+2\gamma)(1+|q|)^{2\gamma}, & q<0.
   \end{cases}
 \end{equation}

 We obtain
\begin{multline*}
  \int_{\Sigma_{t_1}^e}\!\! |\partial \phi|^2\ud x +\int_{\Sigma_{t_1}^i} \!\!|\partial \phi|^2 w(q) \ud x
  + \int_{\mathcal{D}^i} \frac{1\!+2\gamma}{4} (1\!+|q|)^{2\gamma} |\pl \phi|^2 \, \ud x \ud t +  \int_{\mathcal{D}^e} \frac{ |\pl \phi|^2\mu/2}{(1+|q|)^{1+2\mu}}\, \ud x \ud t\\
    \leq \int_{\Sigma_{t_2}^i}  w(q) |\partial \phi|^2\ud x +2\int_{\Sigma_{t_2}^e} |\partial \phi|^2\ud x
  +  2\int_{\mathcal{D}} F w(q) \partial_t\phi\,\ud x\ud t
\end{multline*}
which implies the statement of the Corollary.
\end{proof}

\begin{remark}
  The choice of $w(q)$ is inferred from the forward problem with the role of the exterior $q>0$ and the interior $q<0$ interchanged \cite{LR10}.
\end{remark}

\subsection{Conformal energy estimates from infinity}
\label{sec:conformal:energy}

In this section we derive a ``fractional Morawetz estimate'' for the backward problem.
This estimate will not be applied to the solutions  directly,
but rather to a remainder $\phi$ after subtraction of an approximate solution.

In this paper the following weighted generalized energy plays an important role:
\begin{equation}\label{eq:norm:k}
\|\phi(t_1,\cdot)\|_{k,s-1}:={\sum}_{|I|\leq k} \| \langle t-r\rangle^{s-1} (Z^I \phi)(t_1,\cdot)\|_{\mathrm{L}^2(\mathbb{R}^3)} \qquad (k\in\mathbb{N}, s\geq 1)
\end{equation}

We obtain control of the solution in this norm using a ``fractional Morawetz'' estimate from infinity.
For any $s\geq 1$ these estimates will involve the following ``conformal energy''  on $\Sigma_t$,
\begin{equation}\label{eq:energy:conformal}
E^s_R[\phi](t) = \frac{1}{2}\int_{\Sigma_t^R} \langle t+r\rangle^{2s}|(\pa_t+\pa_r)(r\phi)|^2 + \big( \langle t+r\rangle^{2s}\!+\langle t-r\rangle^{2s}\big) |\nabb
(r\phi)|^2 +  \langle t-r\rangle^{2s}|(\pa_t-\pa_r)(r\phi)|^2 \, \frac{dx}{r^2}
\end{equation}
and the corresponding flux through the cones $C_R$:
\begin{equation}\label{eq:flux:conformal}
F_R^s(t_1,t_2) = \frac{1}{2}\int_{t_1}^{t_2} \int_{\partial \Sigma_t^{R-(t_2-t)}} \langle t+r\rangle^{2s}|(\pa_t+\pa_r)(r\phi)|^2+\langle t-r\rangle^{2s} |\nabb (r\phi)|^2 dS(\omega) \ud t.
\end{equation}
We also denote by $E^s(t)=\lim_{R\to \infty} E^s_R(t)$, and
\begin{equation}
  E^s_{(k)}[\phi](t):=\sum_{|I|\leq k}E^s[Z^I\phi](t)\,.
\end{equation}

\begin{prop}[Fractional Morawetz estimate from infinity] \label{prop:frac:morawetz}
Let $t_1\leq t_2$,  then for $s\geq 1$
\begin{equation}\label{eq:energyestimate}
 E^s(t_1)^{{1}/{2}}+\sup_{R} F_R^s(t_1,t_2)^{1/2}\lesssim E^s(t_2)^{{1}/{2}}+\int_{t_1}^{t_2} \|\langle t+r\rangle^s \Box \phi(t,\cdot)\|_{L^2_x}\ud t\,.
\end{equation}
Moreover we have the  interior estimate
\begin{equation}\label{eq:energyestimateinterior}
 E^s_{R-(t_2-t_1)}(t_1)^{{1}/{2}}+F_R^s(t_1,t_2)^{1/2}\lesssim E^s_R(t_2)^{{1}/{2}}+\int_{t_1}^{t_2} \|\langle t+r\rangle^s \Box \phi(t,\cdot)\|_{L^2_x\big( \Sigma_t^{R-(t_2-t)}\big)} dt
\end{equation}
and the exterior estimate
\begin{equation}\label{eq:energyestimateexterior}
 E^s_{ \mathbb{R}^3\setminus \Sigma_t^{R-(t_2-t_1)}}(t_1)^{{1}/{2}}\lesssim F_R^s(t_1,t_2)^{1/2} +E^s_{\mathbb{R}^3\setminus \Sigma_t^{R}}(t_2)^{{1}/{2}}+\int_{t_1}^{t_2} \|\langle t+r\rangle^s \Box \phi(t,\cdot)\|_{L^2_x\big( \mathbb{R}^3\setminus \Sigma_t^{R-(t_2-t)}\big)} dt .
\end{equation}
\end{prop}

Once one has obtained an estimate for the ``conformal energy'' this translates into a statement for the norms \eqref{eq:norm:k}.
The proof of the following proposition is given at the end of this section.

\begin{prop}\label{prop:normbyenergy} Let  $\lim_{R\to\infty}\sup_{\partial\Sigma_{t_1}^R} r^{s+1/2}|\phi|=0$. Then we have
\begin{equation}\label{eq:normestimate}
  \|\phi(t,\cdot)\|_{1,s-1}
  \lesssim \|\phi(t,\cdot)\|_{1,+,s-1} \,,
\end{equation}
\beq\label{eq:normenergyestimate}
\|\phi(t,\cdot)\|_{1,+,s-1}\lesssim E^s(t)^{1/2},
\eq

where
\begin{multline}\label{eq:norm:s}
    \|\phi(t,\cdot)\|_{1,+,s-1}^2 := \\ \!\int_{\mathbb{R}^3}\! \langle\, t\!+r\rangle^{2s} \Big( \big|(\pa_t\!+\pa_r)(r \phi)\big|^2\!\! + |\nabb (r \phi)|^2\Big)
 +  \langle\, t\!-\!r\rangle^{2s}\Big(\big|(\pa_t\!-\!\pa_r)(r \phi)\big|^2\!\! + \phi^2\!+\frac{(r\phi)^2}{\!\langle \, t\!-\!r\rangle^2\!} \Big)\, \frac{dx}{r^2} .
\end{multline}

\end{prop}

\bigskip
Proposition~\ref{prop:frac:morawetz} is central to this paper and follows --- as we will now prove --- from  the following generalisation of the classical conformal Morawetz estimate for the wave equation on Minkowski space, which corresponds to the choice $s=1$.

\begin{prop}\label{prop:frac:morawetz:id}
  For $s\geq 1$ let $E_R^s$, and $F_R^s$ be defined by \eqref{eq:energy:conformal}, and \eqref{eq:flux:conformal} respectively.
Then for $t_1\leq t_2$
\begin{equation}\label{eq:frac:morawetz:id}
E^s_{R-(t_2-t_1)}(t_1)+F_R^s(t_1,t_2)=E^s_R(t_2)+\int_{t_1}^{t_2}\int_{ \Sigma_t^{R-(t_2-t)}} r^{-1}
X (r\phi)\square \phi -r\pa_r \Omega|\nabb\phi|^2 dx dt
\end{equation}
where $X=\langle t+r \rangle^{2s}(\partial_t+\partial_r)+\langle t-r \rangle^{2s}(\partial_t-\partial_r)$ and $\Omega=(\langle t+r\rangle^{2s}-\langle t-r\rangle^{2s})/r$.
\end{prop}

For the proof, we consider a more general class of energy identities associated to the following multipliers $X$, and functions $f\in\mathrm{C}^3(\mathbb{R})$:
\begin{equation}\label{def:X}
  X=L_+ +L_-\,,\qquad L_\pm=f(t\pm r)(\pa_t\pm\pa_r)\,,\qquad \pa_r=\omega^i\pa_i\,.
\end{equation}
The classical conformal Morawetz vectorfield corresponds to the choice $f(v)=v^2$, while for the fractional Morawetz estimates we will choose
\begin{equation}\label{eq:f:frac}
  f(v)=\frac{1}{a}(1+v^2)^{a/2}\,,\qquad a\geq 2\,.
\end{equation}

For this section also see \cite{LS06b} where the forward version of the above theorem was proved.

Let $T_{\alpha\beta}$ be a energy momentum tensor for the linear wave equation,
$$
T_{\alpha\beta}= \phi_{\alpha}\phi_{\beta}-\frac{1}{2}
m_{\alpha\beta} m^{\mu\nu}\phi_\mu \phi_\nu,\quad
\phi_\alpha=\pa_\alpha\phi, \quad
 \pa_\alpha T^{\,\alpha}_{\,\,\,\,\beta}=(\square \phi)
 \phi_\beta,\quad T_\alpha^{\,\,\,
 \alpha}=-m^{\alpha\beta}\phi_\alpha\phi_\beta,
$$
where $\square=m^{\alpha\beta}\pa_\alpha\pa_\beta$.
Then for any vector field $X$,
$$
\pa_\alpha P^{\,\alpha}\!=(\pa_\alpha
T^{\,\alpha}_{\,\,\,\,\beta})X^\beta\!
 +\frac{1}{2}T^{\alpha\beta}{}^{(X)}\pi_{\alpha\beta},
 \quad\text{where}\quad P_{\!\alpha}\! =T_{\alpha\beta} X^\beta
 \quad \text{and}\quad  {}^{(X)}\pi_{\alpha\beta}=\pa_\alpha
 X_\beta+\pa_\beta X_\alpha
$$

\begin{lemma} With $X$ defined by \eqref{def:X} we have the identity
$$
\frac{{}^{(X)}\pi_{\alpha\beta}\!\!}{2}\,\, T^{\alpha\beta}\!\! = \Big(
\frac{f(t\!+r)-\!f(t\!-\!r)\!\!}{r}\,-f^{\,\prime\!}(t+r)-f^{\,\prime\!}(t-r)\Big)|\nabb
\phi|^2 - \frac{f(t\!+r)-\!f(t\!-\!r)\!}{r}\,m^{\mu\nu}\!\phi_\mu \phi_\nu
$$
\end{lemma}
\begin{proof} With $L=\pa_t+\pa_r$ and $\underline{L}=\pa_t-\pa_r$  we compute
  \begin{align*}
    {}^{(L_+)}\pi_{\alpha\beta}&= f(t+ r)\big(\pa_\alpha L_\beta+\pa_\beta L_\alpha\big)-f^{\,\prime\!}(t+ r)(\underline{L}_\alpha L_\beta+\underline{L}_\beta L_\alpha)\\
    {}^{(L_-)}\pi_{\alpha\beta}&= f(t- r)\big(\pa_\alpha \underline{L}_\beta+\pa_\beta \underline{L}_\alpha\big)-f^{\,\prime\!}(t- r)(\underline{L}_\alpha L_\beta+\underline{L}_\beta L_\alpha)
  \end{align*}
  where we used that $\underline{L}_\alpha=-L^\alpha$. Moreover, recall that $m_{\alpha\beta}=-\big(\underline{L}_\alpha
L_\beta+L_\alpha \underline{L}_\beta\big)/2+\slashm_{\alpha\beta}$,
where
$\slashm^{\alpha\beta}\phi_\alpha\phi_\beta=(\delta^{ij}-\omega^i\omega^j)\phi_i\phi_j
=|\nabb \phi|^2$ is the angular gradient. Furthermore $\pa_0 L_\beta=\pa_\beta L_0=0$ and $\pa_i L_j=\pa_i\omega_j=(\delta_{ij}-\omega_i\omega_j)/r=\slashm^{ij}/r$. Hence $\pa_\alpha L_\beta=\slashm_{\alpha\beta}$ and
$$
\frac{{}^{(L_\pm)}\pi_{\alpha\beta}}{2}=\pm \frac{f(t\pm
r)}{r}\,\slashm_{\alpha\beta}-f^{\,\prime\!}(t\pm
r)(\slashm_{\alpha\beta}-m_{\alpha\beta}).
$$
Here
$$
(\slashm_{\alpha\beta}-m_{\alpha\beta})
T^{\alpha\beta}=\frac{(\underline{L}_\alpha
L_\beta+\underline{L}_\beta L_\alpha)}{2}\, T^{\alpha\beta}=|\nabb
\phi|^2.
$$
It follows that
$$
\frac{{}^{(L_\pm)}\pi_{\alpha\beta}}{2}\,\,T^{\alpha\beta} =
\Big(\pm \frac{f(t\pm r)}{r}-f^{\,\prime\!}(t\pm r)\Big)|\nabb \phi|^2
\mp \frac{f(t\pm r)}{r}m^{\mu\nu}\phi_\mu \phi_\nu
$$
which proves the lemma.
\end{proof}

We can rewrite the identity of the previous Lemma as
 \begin{equation}\label{def:Omega}
\frac{{}^{(X)}\pi_{\alpha\beta}\!\!}{2}\, \,T^{\alpha\beta}\!=(-r\pa_r\Omega)|\nabb \phi|^2-\,\Omega\, m^{\mu\nu}\!\phi_\mu \phi_\nu,\quad\text{where}\quad \Omega(t,r)=\frac{f(t\!+r)-f(t\!-\!r)}{r}
\end{equation}
and thus
\begin{equation}
\pa_\alpha P^{\,\alpha}
=(\square\phi) X^\alpha\phi_\alpha -r(\pa_r \Omega)|\nabb\phi|^2- \Omega\, m^{\mu\nu}\phi_\mu \phi_\nu .
\end{equation}

Moreover, observe that $\square \,\Omega=0$, and hence
\begin{equation}\label{def:P:tilde}
\pa_\alpha\widetilde{P}^{\,\alpha}=\big(\Omega \phi+
X^\alpha\phi_\alpha\big)\,\square \phi +(-r\pa_r \Omega)|\nabb
\phi|^2,\quad\text{where}\quad
\widetilde{P}_\alpha=P_\alpha+\Omega\phi\, \phi_\alpha
-\frac{1}{2}\phi^2 \Omega_\alpha .
\end{equation}
Also note that $X r=f(t+r)-f(t-r)$, hence $(\Omega+X^\alpha\pa_\alpha)\phi= r^{-1}X^\alpha\pa_\alpha (r\phi)$.

Let us summarize the divergence properties of the \emph{modified current} $\widetilde{P}_\alpha$ as follows:
\begin{lemma}\label{lemma:div:P}
  Let the current $\widetilde{P}_\alpha$ be defined by \eqref{def:P:tilde}, where $X$ is chosen as in \eqref{def:X}, and $\Omega$ defined by \eqref{def:Omega}, then we have
  \begin{equation*}
    \partial^\alpha\widetilde{P}_\alpha= r^{-1}X(r\phi) \Box\phi-r\pa_r \Omega|\nabb \phi|^2.
  \end{equation*}
\end{lemma}

Let us now turn to the flux terms
\begin{gather}
  \widetilde{P}_\alpha \partial_t^\alpha=T(X,\partial_t)+\Omega\phi\partial_t\phi-\frac{1}{2}\phi^2\partial_t\Omega,\\
  \widetilde{P}_\alpha L^\alpha= T(X,L)+\Omega\phi L\phi-\frac{1}{2}\phi^2 L\Omega .
\end{gather}

\begin{lemma}\label{lemma:P:flux}
  Let $C_R^+$ the outgoing null hypersurface from $\partial \Sigma_{t_1}^R$, truncated at $t=t_2$, then
  for any differentiable function $f$, we have with $X$, $\Omega$ and $\tilde{P}$ as defined above:
\begin{multline*}
  \frac{1}{2}\int_{\Sigma_{t_1}^R}  \Bigl[f(t+r)\bigl(L(r\phi)\bigr)^2+f(t-r)\bigl(\Lb(r\phi)\bigr)^2  +\bigl(f(t+r)+f(t-r)\bigr)r^2|\nabb\phi|^2\Bigr]\frac{dx}{r^2}\\
  +\int_{C_R^+}  \Bigl[ f(t+r)(L(r\phi))^2+f(t-r)r^2|\nabb\phi|^2\Bigr] \frac{r^2\ud\omega\ud v}{r^2}= \\= \int_{\Sigma_{t_1}^R}\widetilde{P}\cdot \partial_t  dx+\int_{C_R^+}{}^\ast\widetilde{P}
  +\frac{1}{2}\int_{\mathbb{S}^2}r\bigl(f(t-r)+f(t+r)\bigr)\phi^2\ud\omega
  \rvert_{t=t_2,r=R+(t_2-t_1)}.
\end{multline*}
 \end{lemma}
 \begin{proof}

First $\pa_t=(L+\underline{L})/2$, and with  $X=f(t+r) L+f(t-r)\underline{L}$ we have
 \begin{align*}
 T(X,L)&=f(t+r) (L\phi)^2+f(t-r) |\nabb\phi|^2\\
   T(X,\underline{L})&=f(t+r) |\nabb\phi|^2+f(t-r) (\Lb\phi)^2\\
   T(X,\pa_t)&=\frac{1}{2}f(t+r) (L\phi)^2+\frac{1}{2}\big(f(t+r)+f(t-r)\big) |\nabb\phi|^2+\frac{1}{2}f(t-r) (\Lb\phi)^2.
 \end{align*}
  Secondly note  that
\beq
  \Omega\partial_t\phi = \frac{f(t+r)-f(t-r)}{r}\partial_t\phi=\frac{f(t+r)}{r}L\phi
  -\frac{f(t-r)}{r}\Lb\phi-\frac{f(t+r)+f(t-r)}{r}\partial_r\phi .
\eq
Now on one hand
\begin{multline*}
  -  \int_{\Sigma_t^R}\frac{f(t+r)+f(t-r)}{r}\phi\, \partial_r\phi\, \ud x=-\frac{1}{2}\int_{\Sigma_t^R} r \bigl(f(t+r)+f(t-r)\bigr) \partial_r \phi^2 \ud \omega\ud r\\
  =\!\frac{1}{2}\!\int_{\Sigma_t^R}\Bigl[\frac{1}{r}\bigl(f^{\,\prime\!}(t+r)-f^{\,\prime\!}(t-r)\bigr)
  +\frac{f(t\!+\!r)\!+\!f(t\!-\!r)}{r^2}\Bigr]\phi^2\ud x
  -\frac{1}{2}\!\int_{\Sigma_t^R}\!\!\!\!\partial_r\big( r \bigl(f(t+r)+f(t-r)\bigr) \phi^2 \big)\ud \omega\ud r
\end{multline*}
and on the other hand
\begin{equation*}
  \partial_t\Omega=\frac{f^{\,\prime}(t+r)-f^{\,\prime}(t-r)}{r} .
\end{equation*}

Therefore
\begin{multline*}
  \int_{\Sigma_{t_1}^R}\widetilde{P}_\alpha\partial_t^\alpha
  +  \frac{1}{2}\!\int_{\Sigma_t^R}\!\!\partial_r \Big(\! r \bigl(f(t+r)+f(t-r)\bigr) \phi^2 \!\Big)\ud \omega\ud r\\
=\frac{1}{2}\int_{\Sigma_t^R}\frac{1}{r^2}\Bigl[f(t+r)(rL\phi)^2+f(t-r)(r\Lb\phi)^2+\bigl(f(t+r)+f(t-r)\bigr)r^2|\nabb\phi|^2\\
  +2f(t+r)r\phi L\phi-2f(t-r)r\phi\Lb\phi+\bigl(f(t+r)+f(t-r)\bigr)\phi^2\Bigr]\ud x\\
  =\frac{1}{2}\int_{\Sigma_t^R}\frac{1}{r^2}\Bigl[f(t+r)\bigl(rL\phi+\phi\bigr)^2+f(t-r)\bigl(r\Lb\phi-\phi\bigr)^2+\bigl(f(t+r)+f(t-r)\bigr)r^2|\nabb\phi|^2\Bigr]
\end{multline*}
which proves the formula:
   \begin{multline}\label{eq:P:tilde:flux:Sigma}
 \frac{1}{2}\int_{\Sigma_t^R}
  \Bigl[f(t+r)\bigl(L(r\phi)\bigr)^2+f(t-r)\bigl(\Lb(r\phi)\bigr)^2
  +\bigl(f(t+r)+f(t-r)\bigr)r^2|\nabb\phi|^2\Bigr]\frac{dx}{r^2}\\
  = \int_{\Sigma_t^R}\widetilde{P}\cdot \partial_t  dx
  +\frac{1}{2}\!\int_{\partial\Sigma_t^R}\!\! r \bigl(f(t+r)+f(t-r)\bigr) \phi^2 \ud \omega .
\end{multline}

Regarding the null flux, we proceed similarly by writing
\begin{equation*}
  \Omega L\phi=\frac{2f(t+r)}{r} L\phi-\frac{f(t+r)+f(t-r)}{r} L\phi
\end{equation*}
and integrating by parts, on the null hypersurface $C_R$ truncated by $t=t_2$, and $t=t_1$,
\begin{multline*}
  -\frac{1}{2}\int_{C_R} \frac{f(t+r)+f(t-r)}{r}L(\phi^2)=-\frac{1}{2}\int_{v_1}^{v_2}\frac{f(2v)+f(2u)}{r}\partial_v(\phi^2)\, r^2\ud\omega\ud v\\
  =-\frac{1}{2}\int_{\mathbb{S}^2}r\bigl(f(t-r)+f(t+r)\bigr)\phi^2\ud\omega\Bigr\rvert^{v_2}_{v_1}+\frac{1}{2}\int_{C_R}\Bigl[\frac{f(t+r)+f(t-r)}{r^2}+\frac{2f'(t+r)}{r}\Bigr]\phi^2
\end{multline*}
where we denoted by $2u=t-r$, $2v=t+r$, hence $L=\partial_v$, and $r=v-u$.
Since
\begin{equation*}
  L\Omega=\frac{2f'(t+r)}{r}-\frac{f(t+r)-f(t-r)}{r^2}
\end{equation*}
we obtain
\begin{multline*}
 \!\!\! \int_{C_R}\!\!\!\!\!{}^\ast\widetilde{P}=\!\int_{v_1}^{v_2}\!\!\!\widetilde{P}\cdot L \,r^2\ud\omega\ud v=\!\int_{v_1}^{v_2}\!\!\!\Big( f(t+r)(L\phi)^2+f(t-r)|\nabb\phi|^2
  +\frac{2f(t\!+\!r)\!\!}{r}\phi L\phi\,+\frac{f(t\!+\!r)\!\!}{r^2}\phi^2\Big)r^2 \ud\omega\ud v\\
  -\frac{1}{2}\int_{\mathbb{S}^2}r\bigl(f(t-r)+f(t+r)\bigr)\phi^2\ud\omega\Bigr\rvert^{v_2}_{v_1}
\end{multline*}
which proves the identity:
\beq\label{eq:P:tilde:flux:C}
  \int_{v_1}^{v_2}\Bigl[ f(t\!+r)(L(r\phi))^2+f(t-r)r^2|\nabb\phi|^2\Bigr] \ud\omega\ud v
  =\int_{C_R}\!\!\!\!{}^\ast\widetilde{P}
  +\frac{1}{2}\int_{\mathbb{S}^2}r\bigl(f(t-r)+f(t+r)\bigr)\phi^2\ud\omega\Bigr\rvert^{v_2}_{v_1} .
\eq

Adding the formulas \eqref{eq:P:tilde:flux:Sigma} and \eqref{eq:P:tilde:flux:C} we see that the boundary terms at $\partial \Sigma_{t_1}^R$ exactly cancel, and we obtain the statement of the Lemma.
\end{proof}

We can now prove the identity \eqref{eq:frac:morawetz:id}.

\begin{proof}[Proof of Proposition~\ref{prop:frac:morawetz:id}]
  By Stokes theorem applied to the 1-form $\widetilde{P}$ on the domain $\mathcal{D}$ we have
  \begin{equation*}
    \int_{\Sigma_{t_2}^R}\widetilde{P}\cdot \partial_t \ud x + \int_{t_1}^{t_2}\int_{\Sigma_t^{R-(t_2-t)}} \partial^\alpha\widetilde{P}_\alpha \ud x\ud t=
    \int_{\Sigma_{t_1}^{R-(t_2-t_1)}}\widetilde{P}\cdot \partial_t \ud x+\int_{C_R}{}^\ast\widetilde{P}.
  \end{equation*}
  A formula for the divergence of $\widetilde{P}$ is given in Lemma~\ref{lemma:div:P}.
  Moreover the right hand side of the above identity is computed in Lemma.~\ref{lemma:P:flux}; see also \eqref{eq:P:tilde:flux:Sigma} for the flux term on the left hand side, and note thus that the boundary terms at $\partial\Sigma_{t_2}^R$ cancel.

   If we set $f(v)=\langle v \rangle^{2s}$, $s\geq 1$ as in \eqref{eq:f:frac} with $a=2s$ the statement of the proposition follows.
\end{proof}

To prove Proposition~\ref{prop:frac:morawetz} we exploit that with our choice of $f$ the ``bulk'' has the correct sign for the backward problem.

\begin{lemma} \label{lemma:f:bulk} Suppose that $f$ is a function such that
\begin{equation}\label{eq:convexcond}
f^{\,\prime\prime\prime\!}(v)+f^{\,\prime\prime\prime\!}(-v)\geq 0, \qquad
\text{and}\quad f^{\,\prime\prime\prime\!}(v)\geq 0, \quad \text{for}
\quad v\geq 0.
\end{equation}
Then for $t\geq 0$ and $r\geq 0$
\begin{equation}\label{eq:deformfactor}
\frac{1}{r}\big( f(t+r)-f(t-r)\big) - \big(
f^{\,\prime\!}(t+r)+f^{\,\prime\!}(t-r)\big)\leq 0 .
\end{equation}
In particular, this condition \eqref{eq:convexcond} is true for
$f(v)=(1+v^2)^{a/2}$, if $a\geq 2$.
\end{lemma}
\begin{proof} We have
\begin{equation}
f(t+r)-f(t-r)=\int_{t-r}^{t+r} f^{\,\prime\!}(v)\, dv=
(v-t)f^{\,\prime\!}(v)\Big|_{t-r}^{t+r}- \int_{t-r}^{t+r}
(v-t)f^{\,\prime\prime\!}(v)\, dv
\end{equation}
and it follows that
\begin{equation}\label{eq:deformfactor2}
\frac{1}{r}\big( f(t+r)-f(t-r)\big) - \big(
f^{\,\prime\!}(t+r)+f^{\,\prime\!}(t-r)\big)= - \frac{1}{r}\int_{t-r}^{t+r}
(v-t)f^{\,\prime\prime\!}(v)\, dv.
\end{equation}
Integrating by parts once more we see that
\begin{equation}\label{eq:deformfactor3}
\frac{1}{r}\big( f(t+r)-f(t-r)\big) - \big(
f^{\,\prime\!}(t+r)+f^{\,\prime\!}(t-r)\big)= \int_{t-r}^{t+r}
\frac{(v-t)^2-r^2}{2r} f^{\,\prime\prime\prime\!}(v)\, dv .
\end{equation}
Since $(v-t)^2-r^2\leq 0$, when $t-r\le v\leq t+r$,
\eqref{eq:deformfactor} follows directly if $t\geq  r$. When
$t\leq r$ we divide the domain of integration up into two parts.
In the region $[r-t,t+r]$ we use that
$f^{\,\prime\prime\prime\!}(v)\!\geq \!0$ as before. In the region
$[t\!-\!r,r\!-\!t]$ we will use both inequalities in \eqref{eq:convexcond}. We have
\begin{multline}
\int_{t-r}^{r-t} \!\!\frac{(v\!-t)^2\!-r^2\!\!\!}{2r}
f^{\,\prime\prime\prime\!}(v)\, dv=\int_{0}^{r-t}
\!\!\frac{(v\!-t)^2\!-r^2\!\!\!}{2r} f^{\,\prime\prime\prime\!}(v) +
\frac{(-v\!-t)^2\!-r^2\!\!\!}{2r} f^{\,\prime\prime\prime\!}(-v)\, dv\\
=\int_{0}^{r-t} \frac{(-v-t)^2\!-r^2\!\!\!}{2r}\big(
f^{\,\prime\prime\prime\!}(v) + f^{\,\prime\prime\prime\!}(-v)\big)
-\frac{2tv}{r}
 f^{\,\prime\prime\prime\!}(v) \, dv \leq 0.
\end{multline}
 If $f(v)\!=\!(1\!+v^2)^{a/2}\!/a$ then
$f^{\,\prime\!}(v)\!=\!v(1\!+v^2)^{a/2-1}\!\!$,
 $f^{\,\prime\prime\!}(v)\!=\!(1\!+v^2)^{a/2-1} \!+(1\!+v^2)^{a/2-2} v^2(a-2) $
 and $f^{\,\prime\prime\prime\!}(v)=(a-2)3v(1+v^2)^{a/2-2}\! +
 (1+v^2)^{a/2-3}(a-2)(a-4) v^3$. Taking out a common factor  we get $f^{\,\prime\prime\prime\!}(v)=(a-2)v(1+v^2)^{a/2-3}\big(
 3(1+v^2)+(a-4)v^2\big)$. Hence
 $$
 f^{\,\prime\prime\prime\!}(v)=(a-2)v(1+v^2)^{a/2-3}\big((a-1)v^2+3\big)\geq
 0,\qquad \text{if}\quad v\geq 0\quad\text{and}\quad a\geq 2.
 $$
 If $f(-v)\!=\!f(v)$ then
 $f^{\,\prime\prime\prime\!}(-v)\!=\!-f^{\,\prime\prime\prime\!}(v)$
 so \eqref{eq:convexcond} hold for $f(v)\!=\!(1\!+v^2)^{a/2}\!/a$, if $a\!\geq \!2$.\!\!\!\!
\end{proof}

\begin{proof}[Proof of Proposition~\ref{prop:frac:morawetz}.]
  Recall that we have chosen $f(v)=\langle v \rangle^{2s}$, $s\geq 1$ to obtain \eqref{eq:frac:morawetz:id}.
  Therefore by Lemma~\ref{lemma:f:bulk} $(-r\partial_r\Omega)\leq 0$, and thus
\begin{equation*}
E^s_{R-(t_2-t_1)}(t_1)+F_R^s(t_1,t_2)\leq E^s_R(t_2)+\int_{t_1}^{t_2}\int_{ \Sigma_t^{R-(t_2-t)}} r^{-1}
X (r\phi)\square \phi  dx dt .
\end{equation*}
Moreover,
\begin{multline*}
  \int_{\Sigma_t}\frac{1}{r}\bigl(X(r\phi)\bigr)\Box\phi\ud x\leq \biggl(\int_{\Sigma_t}\langle t+r\rangle^{2s} \bigl(L(r\phi)\bigr)^2\frac{\ud x}{r^2}\biggr)^{{1}/{2}}\lVert \langle t+r\rangle^s \Box \phi\rVert_{\mathrm{L}^2(\Sigma_t)}\\
  +\biggl(\int_{\Sigma_t}\langle t-r\rangle^{2s}\bigl(\Lb\,(r\phi)\bigr)^2\frac{\ud x}{r^2}\biggr)^{{1}/{2}}\lVert \langle t-r\rangle \Box \phi\rVert_{\mathrm{L}^2(\Sigma_t)}\leq E_R^s(t)^{{1}/{2}}\lVert \langle t+r\rangle^s \Box \phi\rVert_{\mathrm{L}^2(\Sigma_t)} .\tag*{\qedhere}
\end{multline*}
\end{proof}

In Proposition \ref{prop:normbyenergy} we also control the zeroth order term which we obtain by a weighted Hardy inequality.

\begin{lemma} \label{lemma:hardy} Let $f\geq 0$ be an even twice differentiable function on $\mathbb{R}$, such that $f'(v)\geq 0$ for $v\geq 0$,  and  ${f'}^2(v)\leq C f''(v)f(v)$.  Then
  \begin{equation*}
    \int_{\Sigma_t^R} f^{\,\prime\prime}(t-r) \phi^2 dx\lesssim \int_{\Sigma_t^R}  \Bigl[f(t+r)\bigl(L(r\phi)\bigr)^2+f(t-r)\bigl(\Lb(r\phi)\bigr)^2\Bigr]\frac{\ud x}{r^2}+\int_{\partial \Sigma_t^R}|f'|\phi^2 \, r^2\ud \omega .
  \end{equation*}
\end{lemma}
\begin{proof}
With $\psi=r\phi$ we have
\[\int_{\Sigma_t^R} f^{\,\prime\prime}(t-r) \phi^2 \ud x=\int_0^R \int_{\mathbb{S}^2} f^{\,\prime\prime}(t-r) \psi^2 \ud r \ud\omega  \]
and
\begin{multline*}
\int_0^R f^{\,\prime\prime}(t-r) \psi^2 \ud r = -\int_0^R \frac{\ud}{\ud r}\bigl[f^{\,\prime}(t-r)\bigr] \psi^2 \ud r=-f^{\,\prime}(t-r) \psi^2 \Big|_{r=0}^R +\int_0^R f^{\,\prime}(t-r) 2 \psi \psi_r dr\\
\leq |f^{\,\prime}(t-R)| \psi^2(R)+\sqrt{ \int_0^R \frac{f^{\,\prime}(t-r)^2}{f(t-r)} \psi^2 \ud r}\sqrt{\int_0^R  f(t-r) \psi_r^2 \ud r} .
\end{multline*}
By the assumption on $f$, namely   $f^{\,\prime}(q)^2/f(q)\leq Cf^{\,\prime\prime}(q)$,
it follows that
\[\int_0^R f^{\,\prime\prime}(t-r) \psi^2 dr \lesssim |f'(t-R)|\psi^2(R)+ \int_0^R  f(t-r) \psi_r^2 dr\]
and hence
\[\int_{\Sigma_t^R} f^{\,\prime\prime}(t-r) \phi^2 dx\lesssim \int_{\partial \Sigma_t^R}|f'|\phi^2\, r^2\ud\omega+\int_{\Sigma_t^R} f(t-r) (\pa_r (r\phi))^2 \frac{dx}{r^2}  . \]
By assumption $f$ is even and increasing on $\mathbb{R}_+$, hence $f(t+r)\geq f(t-r)$ and
\begin{equation*}
  \frac{1}{2}\int_{\Sigma_t^R}  \Bigl[f(t+r)\bigl(L(r\phi)\bigr)^2+f(t-r)\bigl(\Lb(r\phi)\bigr)^2\Bigr]\frac{\ud x}{r^2} \geq \int_{\Sigma_t^R}f(t-r)\bigl(\partial_r (r\phi)\bigr)^2\frac{\ud x}{r^2} .\tag*{\qedhere}
\end{equation*}
\end{proof}

\begin{lemma} \label{lemma:hardystuff} We have
  \begin{equation*}
    \int_{\Sigma_t} \langle t-r\rangle^{2s} \phi^2 \frac{\ud x}{r^2}
\lesssim \int_{\Sigma_t} \langle t-r\rangle^{2s-2} \phi^2 \ud x+\int_{\Sigma_t} \langle t-r\rangle^{2s} (\pa_r (r\phi))^2 \frac{\ud x}{r^2} .
\end{equation*}
provided $\lim_{R\to\infty}\sup_{\partial \Sigma_t^R}\langle t-r\rangle^{2s} r\phi^2=0$.
\end{lemma}
\begin{proof}
  We have
$\phi^2=2\phi \pa_r(r\phi)-\pa_r(r\phi^2)$
and hence
\begin{equation*}
\int_0^R\langle t-r\rangle^{2s} \phi^2 \ud r=\int_0^R 2\langle t-r\rangle^{2s} \phi \pa_r(r\phi)
+(\pa_r \langle t-r\rangle^{2s})  r\phi^2 \ud r-\langle t-r\rangle^{2s} r\phi^2\bigr\rvert^R_0 .
\end{equation*}
By assumption the boundary term vanishes, and after integrating over the sphere we obtain by Cauchy-Schwarz that
\begin{multline*}
  \int_0^\infty\int_{\mathbb{S}^2}\langle t-r\rangle^{2s} \phi^2 \ud S(\omega)\ud r\\
  \lesssim \Bigl(\int\int \langle t-r\rangle^{2s} \phi^2 \ud S(\omega)\ud r\Bigr)^{{1}/{2}}
  \Bigl(\int\int \langle t-r\rangle^{2s} \bigl( \pa_r(r\phi)\bigr)^2+ \langle t-r\rangle^{2s-2}  (r\phi)^2 \ud S(\omega) \ud r\Bigr)^{{1}/{2}}
\end{multline*}
where we used that $|\partial_r\langle t-r\rangle^{2s}|\leq 2s \langle t-r\rangle^{2s-1}$.
This proves the Lemma.
\end{proof}

\begin{proof}[Proof of Proposition~\ref{prop:normbyenergy}.]
Thus to infer the estimate \eqref{eq:normenergyestimate}  when $R\to\infty$, it remains to control the zeroth order term.
Note that with our choice of $f$, $f''(v)\geq 2s\langle v\rangle^{2s-2}$, and $|f'(v)|\leq 2s\langle v\rangle^{2s-1}$.
Hence by Lemma.~\ref{lemma:hardy}
  \begin{equation*}
    \int_{\Sigma_t^R} \langle t-r\rangle^{2s-2} \phi^2 dx\lesssim E_R^s(t)+\int_{\partial \Sigma_t^R}\phi^2 \,\langle t-r\rangle^{2s-1} r^2\ud \omega .
  \end{equation*}
  Furthermore, in the limit $R\to\infty$, the boundary term vanishes by assumption, and by Lemma~\ref{lemma:hardystuff}:
    \begin{equation*}
    \int_{\Sigma_t} \langle t-r\rangle^{2s} \phi^2 \frac{\ud x}{r^2}
\lesssim \int_{\Sigma_t} \langle t-r\rangle^{2s-2} \phi^2 \ud x+\int_{\Sigma_t} \langle t-r\rangle^{2s} (\pa_r (r\phi))^2 \frac{\ud x}{r^2} .
\end{equation*}
Here the last term is also controlled by $E_R^s(t)$, as $R\to\infty$, as demonstrated in the last step of the proof of Lemma~\ref{lemma:hardy}.
This shows the first estimate of the Proposition.

Finally let us turn to the bound \eqref{eq:normestimate}. The estimate is obviously true for $|I|=0$. For $|I|=1$ we have $Z\in\{\partial_t,\partial_i, \Omega_{ij}, S, \Omega_{0i}\}$.
Since
\begin{multline*}
  \langle t+r\rangle^2 \bigl|(\partial_t+\partial_r)(r\phi)\bigr|^2+\langle t-r\rangle^2 \bigl|(\partial_t-\partial_r)(r\phi)\bigr|^2 =\\
  =2\bigl(\partial_t(r\phi)\bigr)^2+2\bigl(\partial_r(r\phi)\bigr)^2+2\bigl(t\partial_t(r\phi)+r\partial_r(r\phi)\bigr)^2+2\bigl(r\partial_t(r\phi)+t\partial_r(r\phi)\bigr)^2
\end{multline*}
and with $S=t\partial_t+r\partial_r$, $\underline{S}=r\partial_t+t\partial_r$, and $\epsilon>0$,
\begin{multline*}
  \frac{2}{r^2}\bigl(S(r\phi)\bigr)^2+\frac{2}{r^2}\bigl(\underline{S}(r\phi)\bigr)^2+\phi^2+\frac{\langle t-r\rangle^2}{r^2}\phi^2\geq\\\geq 2(1-2\epsilon)(S\phi)^2+2(1-2\epsilon)(\underline{S}\phi)^2+\frac{1+(3-\epsilon^{-1})t^2-2tr+(4-\epsilon^{-1})r^2}{r^2}\phi^2
\end{multline*}
we can infer, with some $\epsilon<1/2$,  that
\begin{multline*}
  \|\phi(t,\cdot)\|_{1,+,s-1}^2\gtrsim  \int_{\mathbb{R}^3} \langle t-r\rangle^{2s-2} \Bigl[(\partial_t\phi)^2+(\partial_r\phi)^2\Bigr]\ud x\\+\int_{\mathbb{R}^3}\langle t-r\rangle^{2s-2}\Bigl[\bigl(S\phi\bigr)^2+\bigl(\underline{S}\phi\bigr)^2+\phi^2\Bigr]+\langle t+r\rangle^{2s-2}\sum_{i<j}\bigl(\Omega_{ij}\phi\bigr)^2 \ud x
\end{multline*}
and thus the inequality holds for  time-translations, scalings, and rotations,
\begin{equation*}
  \|\phi(t,\cdot)\|_{1,+,s-1}^2\gtrsim\sum_{Z\in\{\partial_t,S,\Omega_{ij}\}}\|\langle t-r\rangle^{s-1}(Z \phi)(t,\cdot)\|^2
\end{equation*}

Given that  $Z=\partial_i=\langle \partial_i ,\partial_r\rangle \partial_r+\nabb_i$,
\begin{equation*}
 \int \langle t-r\rangle^{2s-2} \bigl(\partial_i\phi\bigr)^2 \ud x\lesssim \int \langle t-r\rangle^{2s-2} \Bigl\{ (\partial_r\phi)^2+\lvert \nabb\phi\rvert^2 \Bigr\}\ud x 
\end{equation*}
the estimate also holds for translations $Z=\partial_i$.
Similarly  we can express the boosts  as
\begin{equation*}
\Omega_{0i}=t\partial_i+x^i\partial_t=\frac{x^i}{r}\Bigl(t\partial_r+r\partial_t\Bigr)+t\nablas_i=\frac{x^i}{r}\underline{S}+t\nablas_i
\end{equation*}
and conclude in view of the above that the inequality also holds for $Z=\Omega_{0i}$.

This proves \eqref{eq:normestimate} for $|I|=1$, and thus the Proposition.
\end{proof}

\section{Decay estimates}

In this section we prove a number of weighted Klainerman Sobolev inequalities which yield pointwise estimates from bounds on generalised conformal energies.

\subsection{Classical decay estimates}

A standard application of the Klainerman Sobolev inequality is the improved pointwise bound on the tangential derivatives to the outgoing cone.

\begin{lemma}\label{lemma:pl}
Suppose $v(t,\cdot)\in\mathrm{C}^\infty(\Sigma_t)$ such that $v(t,x)\to 0$ as $|x|\to\infty$. Then
\begin{equation}
| \pa v|\lesssim \frac{1}{\langle t+r\rangle \langle t-r \rangle^{1/2}}  \sum_{|{I}|\leq 2} \lVert \partial{Z}^{I} v\rVert_{L^2(\Sigma_t)}\,,
\end{equation}
\begin{equation}
| v|\lesssim \frac{1}{\langle t+r\rangle^{{1}/{2}}}  \sum_{|{I}|\leq 2} \lVert \partial{Z}^{I} v\rVert_{L^2(\Sigma_t)}\,,\qquad  \text{and}\qquad   |\pl v|\lesssim \frac{1}{\langle t+r\rangle^{{3}/{2}}} \sum_{|{I}|\leq 3}\lVert \partial {Z}^{I} v\rVert_{L^2(\Sigma_t)}\,.
  \end{equation}
\end{lemma}

\begin{proof} The first is the standard Klainerman-Sobolev inequality.
  The second inequality then follows from
  \begin{multline*}
    |v(t,r\omega)|\leq \int_r^\infty |\partial v(t,s\omega)| \ud s
    \leq C_S \sum_{|{I}|\leq 2}\lVert \partial {Z}^{I} v\rVert_{L^2(\Sigma_t)} \int_r^\infty\frac{\ud s}{(1+t+s)(1+|t-s|)^{{1}/{2}}}\\
    \leq \frac{C_S}{(1+t+r)^{{1}/{2}}}  \sum_{|{I}|\leq 2} \lVert \partial{Z}^{I} v\rVert_{L^2(\Sigma_t)}\,.
  \end{multline*}
  For the last recall that $|\pl v|$ gains a weight in $t+r$ in the sense that we have the pointwise inequality $|\pl v| \lesssim \langle t+r\rangle^{-1}\sum_{|I|=1} |Z^I v|$. Since by the first inequality $|Zv|\lesssim \langle t+r\rangle |\partial v| \lesssim \langle t-r\rangle^{-1/2}\to 0$ as $r\to\infty$ on each $\Sigma_t$, we can apply the second inequality also to $Zv$, and conclude that
    \begin{equation*}
    |\pl v|\leq C\sum_{|{I}|=1}\frac{|{Z}^{I} v|}{\langle t+r\rangle}\leq  \frac{C C_S}{\langle t+r\rangle^{{3}/{2}}} \sum_{|{I}|\leq 3}\lVert \partial {Z}^{I} v\rVert_{L^2(\Sigma_t)}\,.\tag*{\qedhere}
  \end{equation*}

\end{proof}

When we construct a solution from data on $\Sigma_T$, and take $T\to\infty$, we cannot a priori assume that the solution is compactly supported on all  $\Sigma_t$, $t\leq T$. A variant of the above proof applies even if the solution is compactly supported only on  $\Sigma_T$.

\begin{lemma}\label{lemma:pl:sigma}
  Let $T>0$ and $\phi\in C^\infty(\mathbb{R}^{3+1})$, and assume that $\phi$ is compactly supported on $\Sigma_T$, while for all $t\leq T$
  \begin{equation}
   {\sum}_{|{I}|\leq 3} \lVert \partial {Z}^{I} \phi\rVert_{L^2(\Sigma_t)} \leq C
 \end{equation}

  Then for all $t\leq T$
  \begin{equation}
    \sup_{\Sigma_t}|\pl \phi|\lesssim \frac{C}{\langle t\rangle^\frac{3}{2}}
  \end{equation}
\end{lemma}

\begin{proof}
  We integrate from the point $(t,r\omega)$ along the outward directed line
  \begin{equation*}
    \gamma(\lambda):\lambda\mapsto(t+\lambda,(r+\sigma\lambda)\omega)
  \end{equation*}
  for some $\sigma> 1$, until it hits  $t=T$; then for $|{I}|=1$,
  \begin{multline*}
    |{Z}^{I} \phi(t,r\omega)|\leq  |{Z}^{I} \phi|(\gamma(T-t))+\int_{0}^{T-t} \sigma|\partial {Z}^{I} \phi(\gamma(\lambda))| \ud \lambda\\
    \leq C C_S  \int_{0}^\infty\frac{\sigma\ud \lambda}{(1+t+r+(1+\sigma)\lambda)(1+|r-t+(\sigma-1)\lambda|)^{{1}/{2}}}
    \lesssim  CC_S\frac{1}{(1+t)^{{1}/{2}}}
  \end{multline*}
  and therefore
    \begin{equation*}
    |\pl \phi|\leq C{\sum}_{|{I}|=1}\frac{|{Z}^{I} v|}{1+t+r}\leq  CC_S\frac{1}{(1+t+r)(1+t)^{{1}/{2}}}\,.\tag*{\qedhere}
  \end{equation*}
\end{proof}

\subsection{Decay estimates from the conformal energy}
\label{sec:decay:conformal}

Finally we show a weighted version of the Klainerman Sobolev inequality that is adapted to the energy introduced in Section~\ref{sec:conformal:energy}.

\begin{lemma}\label{lemma:KS:weights} Let $v(t,\cdot)\in \mathrm{C}^\infty(\mathbb{R}^3)$ such that $|v(t,x)|\to 0$ as $|x|\to\infty$. Then for $s\geq 1$
  \begin{equation*}
   \langle t+r\rangle \langle t-r\rangle^{s-1/2} |v| \lesssim {\sum}_{|I|\leq 2}\|\langle t-r\rangle^{s-1} Z^I v(t,\cdot)\|_{\mathrm{L}^2(\mathbb{R}^3)} .
  \end{equation*}
\end{lemma}
\begin{proof}
  The case $s=1$ this is the standard Klainerman Sobolev inequality. For $s>1$, let us then apply this inequality to the function $\langle t-r\rangle^{s-1} v$:
  \begin{equation*}
    \langle t+r\rangle \langle t-r\rangle^{1/2}\langle t-r\rangle^{s-1}|v| \lesssim {\sum}_{|I|\leq 2}\| Z^I\big(\langle t-r\rangle^{s-1} v\big)(t,\cdot)\|_{\mathrm{L}^2(\mathbb{R}^3)} .
  \end{equation*}
  Since
  \begin{equation*}
    Z\langle t-r\rangle^{s-1}\leq (s-1)\langle t-r\rangle^{s-2}|Z(t-r)|\lesssim \langle t-r\rangle^{s-1}
  \end{equation*}
  and hence $|Z^I  \langle t-r\rangle^{s-1}|\lesssim \langle t-r\rangle^{s-1}$, for $|I|\leq 2$, the statement of the Lemma follows.
\end{proof}

\subsection{Weighted interior decay estimates using energies on light cones}
Another way to get improved interior decay from infinity is to use energies on light cones.
\begin{lemma} Suppose that $v=\partial_t v=0$ when $t=t_2$. Then
\begin{equation}
t_1^{1/2} |w_-^\gamma|^{1/2}|v(t_1,x)|\lesssim \Big(\int_{t_1}^{t_2}\!\!\!\!\int_{\partial \Sigma_{t}^{t-t_1}}\bigtwo( |\pl S v|^2\!+(t\!-r)^2 |\Box v|^2\bigtwo)\,w_-^\gamma dS \ud t
+ \sum_{|I|\leq 3} \int_{\Sigma_{t_1}}|\partial Z^I v|^2 w_-^\gamma\ud x\Big)^{\! 1/2}.
\end{equation}
\end{lemma}
We first give the estimate at the origin. Once we have an estimate for $v$ at the origin we can integrate the derivative from the origin and estimate the derivative using a weighted Klainerman Sobolev inequality as in \cite{LS06b} to get the better decay estimate everywhere in the interior,
which proves the above Lemma.

\begin{lemma} Suppose that $v=\partial_t v=0$ for $t=t_2$, $t_2>q_-$. Then
\begin{align}\label{eq:firstconedecayestimate}
{r^{1/2}}|w_-^\gamma|^{1/2}|v(r+q_-,r\omega)|&\lesssim \sum_{|I|\leq 2}\Big(\int_{r+q_-}^{t_2}\!\int_{\partial \Sigma_{t}^{t-q_-}} |\pl\Omega^I v|^2\,w_-^\gamma dS \ud t\Big)^{\! 1/2}\!\!\!\!\!,\\
\label{eq:secondconedecayestimate}
{q_-^{1/2}}|w_-^\gamma|^{1/2}|v(q_-,0)|&\lesssim \Big(\int_{q_-}^{t_2}\!\int_{\partial \Sigma_{t}^{t-q_-}}\!\!\!\bigtwo( |\pl S v|^2\!+(t\!-r)^2 |\Box v|^2\bigtwo)\,w_-^\gamma dS \ud t\Big)^{\! 1/2}\!\!\!\!\!,\\
\label{eq:thirdconedecayestimate}
{q_-^{1/2}}|w_-^\gamma|^{1/2}|v(r\!+\!q_-,r\omega)|&\lesssim \sum_{|J|\leq 3}\!\!\Big(\int_{\!q_-\!\!}^{t_2}\!\!\int_{\partial \Sigma_{t}^{t-q_-}}\!\!\!\bigtwo( |\pl Z^J\! v|^2\!+(t\!-\!r)^2 |Z^J\Box v|^2\bigtwo)\,w_-^\gamma dS \ud t\Big)^{\! 1/2}\!\!\!\!.\,\,\,\,\,\,\,\,\,\,
\end{align}
\end{lemma}
\begin{proof}
   Since $w_-^\gamma\!=(1\!+ q_-)^{1+2\gamma}$  on $\partial \Sigma_{t}^{t-q_-}$,  it is enough to prove it for $\gamma=-1/2$.
  We will be using the identity
  \begin{equation} \label{eq:decayconeidentityone}
\int_a^b |\partial_\xi (\xi V) |^2 \, d\xi+aV(a)^2=\int_a^b |\partial_\xi V |^2 \, \xi^2d\xi+bV(b)^2,
  \end{equation}
obtained from expanding the derivative and integrating by parts. Applying this identity to
$V(\xi)=v\big(\xi+q_-,\xi\omega\big)$, with $b=t_2-q_-$ where $V$ vanishes and $a=r$, and using Sobolev's inequality on the sphere gives
$$
r \sup_{\omega\in \Bbb S^2} v(r+q_-,r\omega)^2\leq \sum_{|I|\leq 2}\int_{\Bbb S^2} \int_{r}^{t_2-q_-}
\big|\pa_\xi \Omega^I v(\xi+q_-,\xi\omega)\big|^2\, \xi^2 d\xi \, dS(\omega),
$$
which proves \eqref{eq:firstconedecayestimate}.
 Let
  $$
   V_+(\xi):=(\partial_t+\partial_r)(r v)\big(\xi+q_-,\xi\omega\big).
  $$
 We will also be using
 \begin{equation}\label{eq:decayconeidentitytwo}
 (2\xi+q_-)\pa_\xi V_+(\xi)
 =\bigtwo(2(\pa_t+\pa_r)(r Sv)+(t-r) r\big( \Box \, v-r^{-2}\triangle_\omega v\big)\bigtwo)(\xi+q_-,\xi\omega),
 \end{equation}
 which follows  using the identity
 \begin{multline*}
 (t+r)(\partial_t+\partial_r)(\partial_t+\partial_r)(r{v})= (\partial_t+\partial_r)\big((t+r)(\partial_t+\partial_r)(r{v})\big)-2(\partial_t+\partial_r)(r{v})\\
 =(\partial_t+\partial_r)\big( 2S-2-(t-r)(\pa_t-\pa_r)\big)(r{v})
 =2(\pa_t+\pa_r) (rS\,{v})-(t-r)(\pa_t+\pa_r)(\pa_t-\pa_r)(r{v})\,.
\end{multline*}
 Let
$$
\overline{v}(t,r)=\frac{1}{4\pi}\int_{\Bbb S^2} v(t,r\omega) \, dS(\omega),
$$
denote the spherical mean.
We have
\begin{equation*}
v(q_-,0)=\overline{v}(q_-,0)=\overline{V}_+(0)=\int_0^{t_2-q_-}\partial_\xi \overline{V}_+\, d\xi.
\end{equation*}
Hence using \eqref{eq:decayconeidentitytwo} and the fact that the tangential Laplacian integrates to $0$ over the sphere
\begin{equation*}
v(q_-,0)
=\int_0^{t_2-q_-}\!\!\!\!\int_{\Bbb S^2}  \bigtwo(2(\pa_t+\pa_r) (rS v )+(t-r)r\Box v\bigtwo)\big(q_-\!+\xi,\xi\omega\big)\, dS(\omega)\frac{ d\xi}{2\xi\!+q_-}.
\end{equation*}
Using Cauchy-Schwarz we get
\begin{equation*}
|v(q_-,0)|
\leq \frac{C}{\sqrt{q_-}}\Big(\int_0^{t_2-q_-}\!\!\!\!\int_{\Bbb S^2}  \bigtwo|\bigtwo(2(\pa_t+\pa_r) (rS v )+(t-r)r\Box v\bigtwo)\big(q_-\!+\xi,\xi\omega\big)\bigtwo|^2 \,dS(\omega) d\xi\Big)^{1/2}.
\end{equation*}
To deal with $(\pa_t+\pa_r)(rSv)$ we use the identity \eqref{eq:decayconeidentityone}.
This proves \eqref{eq:secondconedecayestimate}. The proof of \eqref{eq:thirdconedecayestimate}
is a variation of the above. We have
\begin{equation*}
(\partial_t+\partial_r)(r \Omega^I\! v)\big(q_-\!+r,r\omega\big)\!
=\!\!\int_r^{t_2-q_-}\!\!\!\!\!\!\!\! \!\!\!\bigtwo(2(\pa_t+\pa_r) (rS\Omega^I \! v )
+(t-r_{\!})r\big(\Box \Omega^I\! v-r^{-2}\triangle_\omega\Omega^I \! v\big)\!\bigtwo)\big(q_-\!+\xi,\xi\omega\big)\, \frac{ d\xi}{\!2\xi\!+\!q_-\!\!\!}.
\end{equation*}
Proceeding as above we get
\begin{multline*}
\sup_{\omega\in\Bbb S^2}
\big|(\partial_t+\partial_r)(r \! v)\big(q_-\!+r,r\omega\big)\big|\!\\
\leq \frac{C}{\!\!\!\!\sqrt{q_-}\!\!}\sum_{|I|\leq 2}\!\!\Big( \int_{\Bbb S^2}
\!\int_r^{t_2-q_-}\!\!\!\!\!\!\!\!\bigtwo|\bigtwo(2(\pa_t+\pa_r) (rS\Omega^I  v )
+(t-r)r\big(\Box \Omega^I\! v-r^{-2}\triangle_\omega\Omega^I \! v\big)\bigtwo)\big(q_-\!+\xi,\xi\omega\big)\bigtwo|^2 d\xi\,d S(\omega)\Big)^{\!1/2}
\end{multline*}
The first two terms are dealt with as above. To deal with last term we write
\begin{equation}
\label{eq:somega}
\pab_{i}= \pa_{i}- \omega_{i}\pa_{r}
=\frac{\omega^k\Omega_{ik}}{r} = \frac {-\omega_{i}\omega^{k}\Omega_{0k} +
\Omega_{0i}}t.
\end{equation}
to express the tangential Laplacian as follows:
\begin{equation*}
\overline{\triangle}:= r^{-2}\triangle_\omega  =\delta^{ij}\overline{\pa}_i\overline{\pa}_j
=\frac{-\delta^{ij}(\overline{\pa}_i\omega_{j})\omega^{k}\Omega_{0k} +
\delta^{ij} \overline{\pa}_i\Omega_{0j}}t=-\frac{2\omega^{k}}{t}\frac{\Omega_{0k}}{r}+\frac{\delta^{ij} }t \overline{\pa}_i\Omega_{0j} .
\end{equation*}
Hence
\begin{multline*}
\sum_{|I|\leq 2}\int_{\Bbb S^2}
\!\int_r^{t_2-q_-}\!\!\!\!\!\!\!\!\bigtwo|\big((t-r)r\overline{\triangle}\Omega^I \! v\big)\big(q_-\!+\xi,\xi\omega\big)\bigtwo|^2 d\xi\,d S(\omega)\\
\lesssim\sum_{|J|\leq 3}\int_{\Bbb S^2}\! \int_{r}^{t_2-q_-}\!\!\!\!\!
\big|\overline{\pa} Z^J v(q_-\! +\xi,\xi\omega)\big|^2\xi^2 d\xi \, dS(\omega)
+\!\!\int_{\Bbb S^2} \!\int_{r}^{t_2-q_-}\!\!\!\!\!
\big| Z^J v(q_-\! +\xi,\xi\omega)\big|^2  d\xi \, dS(\omega) .
\end{multline*}
Here the second term is bounded by a constant times the first. In fact it follows from
\eqref{eq:decayconeidentityone}
\begin{equation}
\int_a^b | V |^2 \, d\xi\leq 4\int_a^b |\partial_\xi V |^2 \, \xi^2d\xi+2bV(b)^2.
  \end{equation}
  We have hence proven that
\begin{equation*}
\sup_{\omega\in\Bbb S^2}
\big|(\partial_t+\partial_r)(r \! v)\big(q_-\!+r,r\omega\big)\big|\!\\
\lesssim \sum_{|J|\leq 3}\!\!\Big(\int_{\! r\!\!}^{t_2}\!\!\!\!\int_{\partial \Sigma_{t}^{t-q_-}}\!\!\!\bigtwo( |\pl Z^J\! v|^2\!+(t\!-\!r)^2 |Z^J\Box v|^2\bigtwo)\,w_-^\gamma dS \ud t\Big)^{\! 1/2}.
\end{equation*}
Then we use the mean value theorem to estimate $v$ in terms of $(\pa_t+\pa_r)(rv)$.
\end{proof}

\section{Scattering and weighted energies for the homogeneous solution}
\label{sec:scattering}

In this section we want to solve the linear homogeneous wave equation $\Box \psi=0$ from infinity
with ``asymptotic data'' given in the form
\beq\label{eq:radiationfield}
\psi\sim \frac{F_0(r-t,\omega)}{r} +\psi_{e},\quad \text{where} \quad \psi_{e}=\frac{M\chi_e(r-t)}{r},\qquad \omega={x}/{|x|}.
\eq
  Here $F_0$ is the Friedlander radiation field, and
 $\psi_e$
 is an \emph{exact solution}  picking up the ``mass term'' in the exterior,
 where $\chi_e(s)= 1$, for $s\geq 2$, and $\chi_e(s)=0$ for $s\leq 1$ is a smooth function.\footnote{In particular, $\psi_e$ is not supported in $\{r\leq 1+t\}\cap\{t\geq 0\}$; see also Lemma~\ref{lemma:cutoff} below.}

The main assumption is that for some $1/2<\gamma<1$:
\begin{equation}\label{eq:decayassumption}
\|F_0\|_{N,\gamma-1/2}^2:=\sum_{|\alpha|+k\leq N} \int_{\mathbb{R}} \int_{\mathbb{S}^2} \big|(\langle\, q\rangle\pa_q)^k\pa_\omega^\alpha F_0(q,\omega)\big|^2\langle q\rangle^{2\gamma -1}\,  \ud S(\omega) dq \leq C
\end{equation}
which ensures suitable pointwise decay of the radiation field in $q=r-t$.
In fact, we have Lemma~\ref{lemma:F:Sobolev} below.
The main result of this section, yielding the existence of a scattering solution with the prescribed radiation field, is Theorem~\ref{thm:intro:hom} stated in the Introduction.

Here we will work with an approximate solution that goes one term further in the ``asymptotic expansion''
\begin{equation}\label{eq:improvedlinearexpansion}
\psi\sim \frac{\!F_0(r\!-\!t,\omega)\!\!}{r}
+\frac{F_{\!1}(r\!-\!t,\omega)}{r^2}+\psi_e\,;
\end{equation}
in fact for this to be a ``good approximation'' we will require $F_1$ to satisfy an ODE.

We will construct a scattering solution by considering a sequence of solutions $\psi_T$ with prescribed  data  when $t=T$ induced by  the following \emph{approximate solution} which is supported in the wave zone, away from the origin,
\begin{equation}\label{eq:linearexpansionpsi0}
\psi_{0}=\frac{F_0(r-t,\omega)}{r}\chi\big(\tfrac{\langle\, t-r\rangle}{r}\big),
\end{equation}
where $\chi$ is a smooth decreasing function such that
$\chi(s)=1$, when $s\leq 1/8$ and $\chi(s)=0$, when $s\geq 1/4$.
More precisely, we set $\psi_T=\psi_{01}+\psi_e+v_T$, where $\psi_{01}$ is the second order approximation
\begin{equation}\label{eq:improvedlinearexpansionpsi01}
\psi_{01}:=\psi_0+\psi_1,\quad\quad \psi_1=\frac{F_{\!1}(r\!-\!t,\omega)}{r^2}\chi\big(\tfrac{\langle\, t-r\rangle}{r}\big)\,.
\end{equation}
and $v_T$ satisfies $\Box v_T=-\Box\psi_{01}$ (as $T\to\infty$) and has trivial data at $t=T$. We then show that as $T\to \infty$, the limit $\psi=\lim_{T\to\infty}\psi_T$ exists in the norms introduced in \eqref{eq:norm:k}. 

\smallskip
Also let $F_i^{\,\prime}(q,\omega)=\pa_q F_i(q,\omega)$ and
\begin{equation}\label{eq:improvedlinearexpansionpsi:prime}
 \psi_i^\prime=\frac{F_{\!i}^{\,\prime}(r\!-\!t,\omega)}{r^{1+i}}\chi\big(\tfrac{\langle\, t-r\rangle}{r}\big)\,, i=0,1\,.
\end{equation}

\begin{lemma}\label{lemma:F:Sobolev}
  Let $F_0(q,\omega)$ satisfy \eqref{eq:decayassumption}. Then for $N'\leq N-3$
\begin{equation}\label{eq:decayassumptioninfty}
\|F_0\|_{N',\infty,\gamma}:=\sum_{|\alpha|+k\leq N'} \sup_{q\in\mathbb{R}}\sup_{\omega\in\mathbb{S}^2} \big|(\langle\, q\rangle\pa_q)^k\pa_\omega^\alpha F_0(q,\omega)\big|\langle q\rangle^{\gamma}\,   \leq C\,.
\end{equation}

\end{lemma}
\begin{proof}
  Consider the case $N'=0$.
First we show that the there is a sequence $q_i'\to\infty$ such that $\sum_{|\alpha|\leq 2}\int_{\mathbb{S}^2}|\partial_\omega^\alpha F_0|^2(q_i',\omega)\ud S\leq C \langle q'_i\rangle^{-2\gamma}\to\infty$. Indeed, since \eqref{eq:decayassumption} holds with $\mathbb{R}$ replaced by $I_k=[q_k,q_{k+1}]$, where $q_{k+1}=2q_k$, this follows from the mean value theorem.
  Then for $p=2\gamma-1$, using the Sobolev inequality on $\mathbb{S}^2$,
  \begin{equation}
    \begin{split}
      |F_0(q,\omega)|\lesssim& \sum_{|\alpha|\leq 2}\int_q^{q_i'} |\partial_q\partial_\omega^\alpha F_0| \ud q+\langle q_i'\rangle^{-2\gamma}
      \\\lesssim& \Bigl(\sum_{|\alpha|\leq 2}\int_{\mathbb{R}}\langle q\rangle^p (\langle q\rangle\partial_q\partial_\omega^\alpha F_0)^2\ud q\Bigr)^\frac{1}{2}\Bigl(\int_q^\infty\frac{\ud q}{\langle q\rangle^{2+p}}\Bigr)^\frac{1}{2}
    \lesssim \langle q\rangle^{\frac{-p-1}{2}}\lVert F_0 \rVert_{3,\gamma-1/2}
  \end{split}
\end{equation}
which proves the inequality because $(p+1)/2=\gamma$.
\end{proof}

\subsection{The  asymptotics of the approximate solutions}\label{sec:asymptotics:second:approx}

Let us quantify the error by which the approximations $\psi_0$ and $\psi_{01}$ fail to be solutions of the homogeneous wave equation.

\begin{lemma} \label{lemma:approx:error} We have
\begin{gather}
  |\,Z^I\Box(\psi_{0})|\lesssim \frac{\chi\big(\tfrac{\langle\, t-r\rangle}{2 r}\big)}{\langle\, t+r\rangle^3}
  \sum_{k+|\alpha|\leq 2+|I|} \big|(\langle\, q\rangle\pa_q)^k\pa_\omega^\alpha F_0(q,\omega)\big|\,,\label{eq:Z:Box:approx}\\
|\,Z^I\big(\Box{({\psi_0+\psi_1})}-(\triangle_\omega \psi_0/r^2-2\psi_1^\prime/r)\big)|\lesssim \frac{\chi\big(\tfrac{\langle\, t-r\rangle}{2 r}\big)}{\langle\, t+r\rangle^4}
\sum_{|\alpha|+k\leq 2+|I|,\, i=0,1} \langle\, q\rangle^{1-i}\big|(\langle\, q\rangle\pa_q)^k\pa_\omega^\alpha F_i(q,\omega)\big|\,. \label{eq:improvedlinearexpansionestimate}
\end{gather}
\end{lemma}
\begin{proof} Consider first the case $|I|=0$. We have
\begin{multline*}
\Box\psi_{0} =\frac{1}{r^3}\triangle_\omega F_0(r-t,\omega)\chi\big(\tfrac{\langle\, t-r\rangle}{r}\big)-\frac{2}{r}F_{0}^{\,\prime}(r-t,\omega)\chi^\prime\big(\tfrac{\langle\, t-r\rangle}{r}\big)\tfrac{\langle\, t-r\rangle}{r^2}\\
-\frac{1} {r}F_0(r-t,\omega)(\pa_r-\pa_t)\Big(\chi^\prime\big(\tfrac{\langle\, t-r\rangle}{r}\big)\tfrac{\langle\, t-r\rangle}{r^2}\Big),
\end{multline*}
where $F_{0}^{\,\prime}(q,\omega)=\pa_q F_0(q,\omega)$.
Recall that $\psi_0$ is supported in $\langle t-r\rangle\leq r/4$, hence $t\sim r$ in the support of the r.h.s~above, and the first estimate follows.
Moreover
\begin{multline*}
\Box{({\psi_0+\psi_1})} =\frac{1}{r^3}\triangle_\omega F_0(r-t,\omega) \chi\big(\tfrac{\langle\, t-r\rangle}{r}\big)-\frac{2F_1^\prime(r-t,\omega)}{r^3} \chi\big(\tfrac{\langle\, t-r\rangle}{r}\big)\\-\frac{2F_{0}^{\,\prime}(r-t,\omega)}{r}\chi^\prime\big(\tfrac{\langle\, t-r\rangle}{r}\big)\tfrac{\langle\, t-r\rangle}{r^2}
-\frac{F_0(r-t,\omega)}{r}(\pa_r-\pa_t)\Big(\chi^\prime\big(\tfrac{\langle\, t-r\rangle}{r}\big)\tfrac{\langle\, t-r\rangle}{r^2}\Big)\\
+\frac{\triangle_\omega F_1(r-t,\omega)}{r^4}\chi\big(\tfrac{\langle\, t-r\rangle}{r}\big)-\frac{2 F_{1}^{\,\prime}(r-t,\omega)}{r}\chi^\prime\big(\tfrac{\langle\, t-r\rangle}{r}\big)\tfrac{\langle\, t-r\rangle}{r^3}\\
-\frac{F_1(r-t,\omega)}{r}(\pa_r-\pa_t)\Big(\chi^\prime\big(\tfrac{\langle\, t-r\rangle}{r}\big)\tfrac{\langle\, t-r\rangle}{r^3}+\chi\big(\tfrac{\langle\, t-r\rangle}{r}\big)\tfrac{1}{r^2}\Big).
\end{multline*}
and note
\begin{equation}\label{eq:psi:cancellation}
  \triangle_\omega \psi_0/r^2-2\psi_1^\prime/r =\frac{\triangle_\omega F_0-2F_1^\prime}{r^3} \chi\big(\tfrac{\langle\, t-r\rangle}{r}\big)\,.
\end{equation}
Since $\langle r-t\rangle \sim\langle t+r\rangle$ in the support of $ \chi^\prime\big(\langle\, t-r\rangle/r\big)$ the second estimate follows.

The general case $|I|\geq 1$ is straightforward. To prove for example \eqref{eq:Z:Box:approx} in the case $|I|=1$, observe first that for $Z=\Omega_{(ij)}=x^i\partial_j-x^j\partial_i$ we have
\begin{equation*}
  \Omega_{(ij)}F_0(q,\omega)=\omega^i\Pi^k_j\partial_{\omega^k}F_0-\omega^j\Pi_i^k\partial_{\omega^k}F_0\,,\quad |\Omega_{ij}F_0(q,\omega)|\leq |\partial_\omega F_0|\,,
\end{equation*}
where $\Pi^i_j=\delta^i_j-\omega^i\omega^j$, and $\omega=x/|x|$; c.f.~\eqref{eq:xi:omegai}.
Moreover, for $Z=S=t\partial_t+r\partial_r$ we have
\begin{equation*}
 S F_0(r-t,\omega)=(r-t)\partial_q F_0(r-t,\omega)\,,\quad |S F_0|\leq \langle q\rangle \partial_qF_0\,.
\end{equation*}
The estimate then follows from the formula above for $\Box\psi_0$.
\end{proof}

To get a second order approximation for the homogeneous equation we choose $F_1$ so that
\beq
2F_1^{\,\prime}(q,\omega)\!= \triangle_\omega F_0(q,\omega),\quad F_1(0,\omega)=0.
\label{eq:F1homogeneouscondition}
\eq
In particular, this ODE allows us to estimate $F_1$ in terms of $F_0$:
\begin{lemma}\label{lem:F1byF0} Suppose $F_1$ satisfies \eqref{eq:F1homogeneouscondition}, with $\|F_0\|_{N,\gamma-1/2}+\|F_0\|_{N,\infty,\gamma}<\infty$, for some $N\geq 2$, and $1/2<\gamma<1$. Then
\beq
\|F_1\|_{N-2,\gamma-3/2}\lesssim \|F_0\|_{N,\gamma-1/2},\quad\text{and}\quad
\|F_1\|_{N-2,\infty,\gamma-1}\lesssim \|F_0\|_{N,\infty,\gamma}.
\eq
\end{lemma}

\begin{proof}
  We first show $\|F_1\|_{0,\gamma-3/2}\lesssim \|F_0\|_{2,\gamma-1/2}$.
  For $q\geq 0>q_0$ set \[f(q)=\int_{q_0}^q\int^\infty_{q'} \langle q^{\prime\prime}\rangle^{2 \gamma-3} dq^{\prime\prime} dq'\] then
  $0\leq -f''(q)= \langle q\rangle^{2\gamma-3}$, $f'(q)\lesssim \langle q\rangle^{2\gamma-2}$, $f(q)\lesssim \langle q\rangle^{2\gamma-1}$,
  and $f^{\,\prime}(q)^2\lesssim (-f''(q))f(q)$.
We write
\begin{equation*}
\int_0^{q_1}\!\!\int_{\mathbb{S}^2} F_1^2(q,\omega) (-f^{\prime\prime})(q) \ud S \ud q = -\int_{\mathbb{S}^2} f^\prime F_1^2 \ud S\Big|_{0}^{q_1}
+2\int_0^{q_1}\!\! \int_{\mathbb{S}^2} f^{\,\prime} F_1 F_1^\prime \ud S \ud q
\end{equation*}
and note  that
\begin{equation*}
0\leq   \int_{\mathbb{S}^2} f^\prime F_1^2 \ud S\Big|_{0}^{q_1} \lesssim  \|F_0\|_{2,\infty,\gamma} <\infty
\end{equation*}
because $F_1(0)=0$ and
\begin{equation*}
  |F_1(q)|\leq \int_0^q|\partial_\omega^2F_0|\ud q\leq \|F_0\|_{2,\infty,\gamma}\langle q\rangle^{1-\gamma} \,.
\end{equation*}
This also shows that $\|F_1\|_{0,\infty,\gamma-1}\lesssim \|F_0\|_{2,\infty,\gamma}$.
Moreover
\begin{equation*}
  \int_0^{q_1} \int_{\mathbb{S}^2} f^{\,\prime} F_1 F_1^\prime \ud S \ud q \lesssim
\sqrt{\!\!\int_0^{q_1} \!\!\int_{\mathbb{S}^2}F_1^2 (-f^{\prime\prime}) \ud S\ud q}
\sqrt{\int_0^{q_1} \!\!\int_{\mathbb{S}^2}{F_1^\prime}^2 f \ud S\ud q }
\end{equation*}
and therefore
\begin{equation*}
  \int_0^{q_1}\!\!\int_{\mathbb{S}^2} F_1^2(q,\omega) \langle q\rangle^{2\gamma-3} \ud S \ud q\lesssim \int_0^{q_1} \!\!\int_{\mathbb{S}^2}{F_1^\prime}^2 \langle q\rangle^{2\gamma-1} \ud S\ud q
\end{equation*}
as desired. Similarly for $q_1\leq 0$.
\end{proof}

In view of the energy estimates in Section~\ref{sec:conformal:energy}, we need to control $\Box {({\psi_0+\psi_1})}$ with the relevant weights in $\mathrm{L}^2_x$. Here the choice \eqref{eq:F1homogeneouscondition} is essential.

\begin{lemma}\label{lem:boxapprox} Suppose $F_1$ satisfies \eqref{eq:F1homogeneouscondition},  with $\|F_0\|_{N,\gamma-1/2}<\infty$, for some $N\geq 4$,  $1/2<\gamma< 1$. Then we have,  for all $s\geq 0$, $|I|\leq N-4$,
\begin{equation}
\|\langle t+r\rangle^s \Box (Z^I {({\psi_0+\psi_1})})(t,\cdot)\|_{L^2_x}\lesssim \frac{\|F_0\|_{4+|I|,\gamma-1/2}}{\langle\, t\rangle^{3/2+\gamma-s}}.
\end{equation}
\end{lemma}
\begin{proof}
  Recall that by \eqref{eq:psi:cancellation} with our choice of $F_1$, $\triangle_\omega \psi_0/r^2-2\psi_1^\prime/r =0$.
  Hence by Lemma~\ref{lemma:approx:error} we have
\begin{multline*}
  \|\langle t+r\rangle^s Z^I\Box {({\psi_0+\psi_1})}(t,\cdot)\|_{L^2_x}^2
  \lesssim\int \langle t+r\rangle^{2s}\Big(\frac{\chi\big(\tfrac{\langle\, t-r\rangle}{2 r}\big)}{\langle\, t+r\rangle^4}\Big)^2\times\\\times
  \sum_{|\alpha|+k|\leq 2+|I|}\Big(\langle\, q\rangle^2\big|(\langle\, q\rangle\pa_q)^k\pa_\omega^\alpha F_0(q,\omega)\big|^2+\big|(\langle\, q\rangle\pa_q)^k\pa_\omega^\alpha F_1(q,\omega)\big|^2\Big) \ud x\,.\\
\end{multline*}
Moreover, in polar coordinates,
\begin{multline*}
  \|\langle t+r\rangle^s Z^I \Box {({\psi_0+\psi_1})}(t,\cdot)\|_{L^2_x}^2
  \lesssim \int_0^R\!\!\int_{\mathbb{S}^2} \frac{ r^2\langle\, t-r\rangle^{3-2\gamma}}{\langle\, t+r\rangle^{8-2s}}\chi\big(\tfrac{\langle\, t-r\rangle}{2 r}\big)\times\\\times
\sum_{|\alpha|+k|\leq 2+|I|}\Big(\langle\, q\rangle^{2\gamma-1}
\big|(\langle\, q\rangle\pa_q)^k\pa_\omega^\alpha F_0(q,\omega)\big|^2
+\langle\, q\rangle^{2\gamma-3}
\big|(\langle\, q\rangle\pa_q)^k\pa_\omega^\alpha F_1(q,\omega)\big|^2\Big) \ud S(\omega) \ud r
\end{multline*}
and we have  by Lemma~\ref{lem:F1byF0}:
\begin{equation*}
\|\langle t+r\rangle^s Z^I \Box {({\psi_0+\psi_1})}(t,\cdot)\|_{L^2_x}^2
\lesssim \frac{\|F_0\|_{2+|I|,\gamma-1/2}^2}{\langle\, t\rangle^{3+2\gamma-2s}}
+\frac{\|F_1\|_{2+|I|,\gamma-3/2}^2}{\langle\, t\rangle^{3+2\gamma-2s}}
\lesssim \frac{\|F_0\|_{4+|I|,\gamma-1/2}^2}{\langle\, t\rangle^{3+2\gamma-2s}}\,.\tag*{\qedhere}
\end{equation*}
\end{proof}

Below we will also need an estimate on the conformal flux of the approximate solutions.
\begin{lemma}  \label{lem:approx:flux}
  For all $s<\gamma+1/2$, $1/2<\gamma<1$, we have the following estimate for $k\geq 1$, provided $F_1$ satisfies \eqref{eq:F1homogeneouscondition}:
  \begin{equation}
    E^s_{(k-1)}[\psi_0+\psi_1](t)\lesssim  \|F_0\|_{3+k,\gamma-1/2}^2\qquad E^s_{(k-1)}[\psi_e]\lesssim M^2
  \end{equation}
\end{lemma}

\begin{proof} We prove the case $k=1$.
  First we check, using Lemma~\ref{lemma:cutoff}, that $E^s[\psi_0]\lesssim \|F_0\|_{1,\gamma-1/2}^2$,
  and similarly, by Lemma~\ref{lem:F1byF0}, $E^s[\psi_1]\lesssim \|F_1\|_{1,\gamma-3/2}^2\lesssim \|F_0\|_{3,\gamma-1/2}^2$.
  Also $E^s[\psi_e]\lesssim M^2$.
\end{proof}

\subsection{Energy bounds for the linear homogeneous scattering problem}

Let us prove an energy bound using the second order asymptotics \eqref{eq:improvedlinearexpansion}.
Here we always assume that $F_1$ satisfies \eqref{eq:F1homogeneouscondition}.

A sequence of linear solutions $\Box\psi_T=0$, with  $\psi_T=v_T+\psi_{01}+\psi_e$ is obtained if we set
\begin{equation}
  \Box v_T=-\Box{({\psi_0+\psi_1})}.
\end{equation}
Recall that  $\Box\psi_e=0$. However, in view of Lemma~\ref{lemma:Z:v} we set
\begin{equation}\label{eq:v:T:homogeneous}
  \Box v_T=-\chi(t/T)\Box{({\psi_0+\psi_1})}.
\end{equation}
and assume that $v_T$ has trivial initial data at $t=T$.

\begin{prop} \label{prop:hom:energy}
   Let $v_T$  be a solution to \eqref{eq:v:T:homogeneous} with trivial data at $t=T$, and $\|F_0\|_{N,\gamma-1/2}<\infty$ for some $N\geq 6$, and $1/2<\gamma<1$. Then for all $0\leq t\leq T$,
   \begin{align}
     \|\pa (Z^I v_T)(t,\cdot)\|_{L^2_x} &\lesssim \frac{\|F_0\|_{4+|I|,\gamma-1/2}}{\langle\, t\rangle^{1/2+\gamma}} \qquad (|I|\leq N-4)\\
|Z^Iv_T(t,x)| &\lesssim \frac{\|F_0\|_{6+|I|,\gamma-1/2}}{\langle \,t+r\rangle^{{1}/{2}}\langle\, t\rangle^{1/2+\gamma}}  \qquad (|I|\leq N-6)
\end{align}
where the constants are independent of $T$. Moreover for $\psi_T=v_T+\psi_{01}+\psi_e$ we have
\begin{equation} \label{eq:scattering:hom:energy}
  \| \partial (Z^I\psi_T)(t,\cdot) \|_{L^2_x} \lesssim \|F_0\|_{4+|I|,\gamma-1/2}\,.
\end{equation}
\end{prop}

In the following we will use the standard notation $\|v(t,\cdot)\|_2=\|v(t,\cdot)\|_{L^2_x}=\|v\|_{L^2(\Sigma_t)}$ for the $L^2$ norm.

 \begin{proof}
   Since $v_T$ satisfies \eqref{eq:v:T:homogeneous} and  has vanishing data at $t=T$, the standard energy estimate implies
\begin{equation}
\|\pa v_T(t,\cdot)\|_2\lesssim \int_{t}^{T} \| \Box  v_T(t,\cdot)\|_2 \ud t
\lesssim \frac{\|F_0\|_{4,\gamma-1/2}}{\langle\, t\rangle^{1/2+\gamma}}.
\end{equation}
where in the last step we have applied  Lemma~\ref{lem:boxapprox} with $s=0$, and $|I|=0$.

In view of Lemma~\ref{lemma:Z:v} the boundary terms at $t=T$ also vanish for the higher order energy estimates:
\begin{equation}
\|\pa Z^I v_T(t,\cdot)\|_2\lesssim \int_{t}^{T} \| \Box  Z^I v_T(t,\cdot)\|_2 \ud t
\end{equation}
Moreover by Lemma~\ref{lemma:Z:chi},
\begin{equation}
  \|\Box Z^Iv_T\|_{L^2(\Sigma_t)}\lesssim \sum_{|J|\leq |I|} \|Z^J\Box\psi_{01}\|_{L^2(\Sigma_t)}
\end{equation}
where we have used that $\psi_{01}$ is supported in the wave zone.
   Hence in view of Lemma~\ref{lem:boxapprox},
\begin{equation}\label{eq:vl2}
\|\pa (Z^I v_T)(t,\cdot)\|_2
\lesssim \frac{\|F_0\|_{4+|I|,\gamma-1/2}}{\langle\, t\rangle^{1/2+\gamma}}.
\end{equation}

Moreover by Lemma~\ref{lemma:pl},
\begin{equation}
|Z^Iv_T(t,x)|
 \lesssim  \frac{1}{\langle \,t+r\rangle^{{1}/{2}}}  \sum_{|{J}|\leq 2+|I|} \lVert \partial({Z}^{J} v_T)\rVert_{L^2(\Sigma_t)}
 \lesssim \frac{\|F_0\|_{6+|I|,\gamma-1/2}}{\langle \,t+r\rangle^{{1}/{2}}\langle\, t\rangle^{1/2+\gamma}}.
\end{equation}

To obtain the bound on $\psi_T$, we can apply the same argument to $\psi_0+\psi_1+\psi_e$ in place of $v_T$. The energy of the approximate solutions on $t=T$ does not vanish,  but are estimated in Lemma~\ref{lemma:energy:approx} below.
 \end{proof}

\begin{lemma}\label{lemma:energy:approx}
  The approximate solutions have finite energy,
   \begin{equation}
     \|\partial Z^I(\psi_0+\psi_1+\psi_e) \|_{L^2(\Sigma_t)}\lesssim \|F_0\|_{3+|I|,\gamma-1/2}+M
   \end{equation}

 \end{lemma}
 \begin{proof} We show the case $|I|=0$.
   Given $\psi_0$ we can easily compute $L\psi_0$, $\Lb\psi_0$, and $\nablas\psi_0$, and find that
   \begin{equation*}
     \begin{split}
       \| \partial\psi_0 \|^2 \leq& \int_\mathbb{R}\int_{\mathbb{S}^2}\Bigl( (\Lb\psi_0)^2+(L\psi_0)^2+ |\nablas\psi_0|^2 \Bigr) r^2\ud r\ud S\\
       \lesssim& \int_{\mathbb{R}}\int_{\mathbb{S}^2}\sum_{k+|\alpha|\leq 1}|(\langle q \rangle \partial_q)^k \partial_\omega^\alpha F_0|^2 \langle q\rangle^{-2}  \ud q\ud S\lesssim \|F_0\|_{1,-1}^2
   \end{split}
 \end{equation*}
and similarly for $\psi_1$, $\| \partial \psi_1 \|\lesssim \| F_1 \|_{1,-2}\lesssim \| F_0 \|_{3,-1}$,
and finally also $\|\partial\psi_e\|\lesssim M$.
 \end{proof}

\begin{remark} \label{remark:loss:hom}
  Note that there is a loss of derivatives in Prop.~\ref{prop:hom:energy}: To show that the solution has finite energy  at $t=0$, we need several derivatives of $F_0$ in $\mathrm{L}^2$.

  For the homogeneous wave equation, in Sobolev spaces \emph{without} weights --- namely in the case $\gamma=1/2$ ---  it is known that this loss of derivatives can be avoided.
In fact, Friedlander proved in \cite{F80} that the map $\mathscr{R}$ which associates to initial data $(\psi_0,\psi_1)$ for the wave equation the function
\begin{equation}
  \label{eq:friedlander:R}
  [\mathscr{R}(\psi_0,\psi_1)](q,\omega)=\lim_{r\to\infty} (r\partial_t\psi)(r-q,r\omega)
\end{equation}
where $\psi$ is the solution to \eqref{eq:intro:wave} with $\psi(0,\cdot)=\psi_0$, $\partial_t\psi(0,\cdot)=\psi_1$, \emph{is a bijective isometry} from energy space $\mathcal{H}=\dot{H}^1(\mathbb{R}^3)\times L^2(\mathbb{R}^3)$ \emph{to} $L^2(\mathbb{R}\times\mathbb{S}^2)$.

In particular, it is shown in \cite{F80}~(Lemma~3.12)  that given $F_0'(q,\omega)\in C^\infty_0(\mathbb{R}\times\mathbb{S}^2)$ of compact support in $q$, there exists a unique solution $\psi\in C^\infty(\mathbb{R}^{3+1})$ to \eqref{eq:intro:wave} with radiation field $F_0(q,\omega)=\int_{-\infty}^q(F_0')(q',\omega)\ud q'$, and moreover
\begin{equation}
  \label{eq:friedlander}
   \| \nabla\psi(0,\cdot) \|_{L^2(\mathbb{R}^3)}+\| \partial_t\psi(0,\cdot) \|_{L^2(\mathbb{R}^3)} = \| F_0' \|_{L^2(\mathbb{R}\times\mathbb{S}^2)}\,.
 \end{equation}
 Furthermore there there exists a map $\mathscr{I}:L^2(\mathbb{R}\times\mathbb{S}^2)\to \mathcal{H}$ such that $\mathscr{R}\circ\mathscr{I}$ is the identity on $L^2(\mathbb{R}\times\mathbb{S}^2)$, and $\mathscr{I}\circ\mathscr{R}$ is the identity on $\mathcal{H}$.

\end{remark}

\subsection{Weighted conformal energy bounds for the homogenous scattering problem}
\label{sec:scattering:homogeneous}

As in the previous section  we consider a sequence  $\psi_T=v_T+({\psi_0+\psi_1})$, where $v_T$  is a solution to the equation
\begin{equation}\label{eq:Box:v:psi}
 \Box v_T=-\chi(t/T)\Box({\psi_0+\psi_1})
\end{equation}
with  trivial  data on $t=T$, which yields a solution $\psi=\lim_{T\to\infty}\psi_T$ to the homogeneous equation  $\Box\psi=0$, as shown in Section~\ref{sec:existence:homogeneous}.

Instead of using the standard energy estimate, we now apply the results of Section~\ref{sec:conformal:energy} and derive weighted energy bounds using the second order asymptotics of Section~\ref{sec:asymptotics:second:approx}.

\begin{prop}\label{prop:scattering:homogeneous}
  Suppose $v_T$ is a solution to \eqref{eq:Box:v:psi} with trivial initial data at $t=T$. As above $\psi_0+\psi_1$ is an approximate solution of the form \eqref{eq:improvedlinearexpansionpsi01}, where $F_1$ satisfies \eqref{eq:F1homogeneouscondition}, and $\|F_0\|_{N,\gamma-1/2}<\infty$, for some $N\geq 5$. If $1<s<\gamma+1/2$, then for all  $0\leq t\leq T$,  uniformly in $T>0$,
    \begin{gather}\label{eq:scattering:homogeneous:v}
\|\psi_T(t,\cdot)\|_{k,s-1}\lesssim \|F_0\|_{3+k,\gamma-1/2}\,,\\
      \|v_T(t,\cdot)\|_{k,s-1} \lesssim \frac{\|F_0\|_{3+k,\gamma-1/2}}{\langle\, t\rangle^{1/2+\gamma-s}}\,,
\end{gather}
for all $k\leq N-3$,
and for $|I|\leq N-5$,
\begin{gather}
  |Z^I\psi_T(t,x)|+|Z^I({\psi_0+\psi_1})(t,x)|\lesssim \frac{\|F_0\|_{5+|I|,\gamma-1/2}}{\langle\,t+r\rangle\langle\, t-r\rangle^{s-1/2}}\,,\\
    |Z^I v_T(t,x)|\lesssim \frac{\|F_0\|_{5+|I|,\gamma-1/2}}{\langle t+r\rangle\langle t\rangle^{1/2+\gamma-s}\langle t-r\rangle^{s-1/2}}\,.
\end{gather}

\end{prop}

\begin{proof}

First recall Proposition~\ref{prop:normbyenergy}, which gives us a bound on
  \begin{equation}
\|v_T(t,\cdot)\|_{k,s-1} \lesssim E^s_{(k-1)}[v_T]^{1/2}
\end{equation}
in terms of the higher order conformal energies.
Now by Proposition~\ref{prop:frac:morawetz}, we have for $s>1$,
\begin{equation}
  E^s_{(k-1)}[v_T]^{1/2}(t)\lesssim E^s_{(k-1)}[v_T]^{1/2}(T) +\sum_{|I|\leq k-1}\int_{t}^{T} \|\langle t+r\rangle^s \Box (Z^I v_T)(t,\cdot)\|_{L^2_x} \ud t .
\end{equation}
Note here also that for any fixed $t$ the boundary condition in Prop.~\ref{prop:frac:morawetz} is verified:
Since $\Box v_T=0$ when $r>2t$, and since we pose vanishing data for $v$ at $t=T$ it follows by finite speed of propagation that the solution $v$ vanishes when $r>2T+(T-t)$. Similarly $\Box(Z^I v_T)=0$ when $r>2t$, and it follows from Lemma~\ref{lemma:Z:v} that the initial data for $Z^Iv_T$ vanishes for $r>2t$, on $t=T$.
Furthermore by Lemma~\ref{lem:boxapprox} we have
\begin{equation}\label{eq:L:box:v}
  \begin{split}
\sum_{|I|\leq k-1}\int_{t}^{T} \|\langle t+r\rangle^s \Box (Z^Iv_T)(t,\cdot)\|_{L^2_x} \ud t
&\leq \int_{t}^{T}\frac{\|F_0\|_{4+k-1,\gamma-1/2}}{\langle\, t\rangle^{3/2+\gamma-s}}\ud t\\
&\lesssim \frac{\|F_0\|_{3+k,\gamma-1/2}}{\langle\, t\rangle^{1/2+\gamma-s}},\quad\text{if}\quad
s<\gamma+1/2.
\end{split}
\end{equation}
With regard to the exponent we have $3/2+\gamma-s<2$ because $\gamma<1$, and for this estimate $s\geq 0$.
However, for the integral to be finite we need $3/2+\gamma-s>1$, which gives the stated restriction $s<1/2+\gamma$.
For the conformal Morawetz estimate above we required $s>1$, and thus all conditions can be verified in view of $\gamma>1/2$.

Let us return to the boundary term at $t=T$ above.  From the definition \eqref{eq:energy:conformal} it follows, c.f.~also Lemma~\ref{lemma:hardystuff}, that $ E^s[v_T](T)\lesssim \int\langle t+r\rangle^{2s}|\partial v|^2\ud x\rvert_{t=T}$,
and hence by Lemma~\ref{lemma:Z:v}, $E^s_{(k-1)}[v_T](T)=0$. 
Therefore we have shown \eqref{eq:scattering:homogeneous:v}, and the pointwise bound for $v_T$ follows from Lemma~\ref{lemma:KS:weights}:
  \begin{equation}
   \langle t+r\rangle \langle t-r\rangle^{s-1/2} |Z^Iv_T| \lesssim {\sum}_{|J|\leq 2+|I|}\|\langle t-r\rangle^{s-1} Z^Jv_T(t,\cdot)\|_{\mathrm{L}^2} .
  \end{equation}

We can also apply Prop.~\ref{prop:normbyenergy} directly to the function $\psi_0+\psi_1$, and obtain in view of Lemma~\ref{lem:approx:flux}
  \begin{equation}\label{eq:psi01:s}
\|(\psi_0+\psi_1)(t,\cdot)\|_{k,s-1} \lesssim E^s_{(k-1)}[\psi_0+\psi_1]^{1/2}(t)\lesssim  \|F_0\|_{3+k,\gamma-1/2}
\end{equation}
and we infer the pointwise bound again by Lemma~\ref{lemma:KS:weights}.
\end{proof}

\subsection{Existence and uniqueness of  the limit as $T$ tends to infinity}\label{sec:existence:homogeneous}
Let us discuss the proof that the sequence $\psi_T$ constructed in Section~\ref{sec:scattering:homogeneous} converges as $T\to\infty$.
Let $T_2>T_1$ and for $i=1,2$ let $\psi_i=\psi_{T_i}=\psi_{01}+\psi_e+v_{T_i}$,
where $v_i=v_{T_i}$ solves \eqref{eq:Box:v:psi} with vanishing data at $t=T_i$, for $i=1,2$. We need to show that the difference $v=v_2-v_1$ tend to $0$ as $T_2>T_1\to\infty$ in the norms we used above.
Since
\begin{equation}
  \Box v= \bigl(\chi(t/T_1)-\chi(t/T_2)\bigr)\Box \psi_{01}
\end{equation}
and $v=v_2$ at $T_1$ we have for $t<T_1$, with a constant independent of $T_1$,
\begin{equation}
  \|v(t,\cdot)\|_{k,s-1}\lesssim E^s_{(k-1)}[v_2]^{1/2}(T_1)+\sum_{|I|\leq k-1}\int_{T_1/8}^{T_1}\|\langle t+r\rangle^s\Box Z^I \psi_{01}\|\ud t\,,
\end{equation}
and that this tends to zero as $T_1\to\infty$ is precisely what we have proven above.

For uniqueness of the limit, assume $F_0=0$. Then by \eqref{eq:scattering:homogeneous:v}, $\|\psi_T(t,\cdot)\|_{k,s-1}\lesssim \|F_0\|_{3+k,\gamma-1/2}=0$, and so
\begin{equation}
  \|\psi(t,\cdot)\|_{k,s-1}\leq \|\psi(t,\cdot)-\psi_T(t\cdot)\|_{k,s-1}\to 0\qquad (T\to\infty)\,.
\end{equation}

\section{The classical null condition model}
\label{sec:classical:null}
The first nonlinear model is an equation satisfying the classical null condition
\beq\label{eq:wave:null:u}
 \Box \phi = Q(\pa \phi,\pa \phi)
\eq
for which we  give asymptotic data in the form of a solution to the homogeneous equation:
\begin{equation}
   \phi\sim \phi_0\,, \qquad  \text{ where }\Box \phi_0=0\,.
 \end{equation}

 In other words, for a \emph{given} linear solution $\phi_0$ we will find a solution to \eqref{eq:wave:null:u} with the same asymptotics. To do this we take $\phi$ to be of the form $\phi=\phi_0+v$ where the difference $v=\phi-\phi_0$ satisfies
\begin{equation}
    \Box v = Q(\partial(v+\phi_0),\partial(v+\phi_0))
     = Q(\partial v,\partial v) + Q(\partial v,\partial \phi_0)+ Q(\partial \phi_0,\partial v) + Q(\partial \phi_0,\partial \phi_0)\,.
   \end{equation}

   We then consider a sequence of solutions $v_T$ corresponding to  \emph{trivial data for $v$ at $t=T$}, of the equation
   \begin{equation}\label{eq:Box:v:T:chi}
     \Box v_T = \chi(t/T) Q(\pa (\phi_0+v_T) , \pa (\phi_0 +v_T))
   \end{equation}
where $\chi$ is a smooth decreasing cutoff function sich that $\chi(s)=1$ for $0\leq s\leq 1/8$ and $\chi(s)=0$ for $s\geq 1/4$,
   and we obtain the solution $\phi$ as the limit of $\phi_T=\phi_0+v_T$ as $T\to\infty$.

 \begin{prop}\label{prop:wave:null}
   Let  $\varepsilon>0$, and $\phi_0$ be a solution to the homogeneous equation $\Box\phi_0=0$ with compact support at $t=0$, such that for $k=7$,
  \begin{equation}
    D_k:=   {\sum}_{|{I}|\leq k}\lVert \partial {Z}^{I} \phi_0\rVert_{L^2(\mathbb{R}^3)}<\varepsilon\,.
  \end{equation}
  Then for sufficiently small $\varepsilon>0$, and any $0<\mu<1/2-\varepsilon$, there exists a unique  solution $\phi$ to \eqref{eq:wave:null:u}  of the form $\phi=\phi_0+v$ where
    \begin{equation}
  \sum_{|{I}|\leq 5} \lVert \partial {Z}^{I} v (t,\cdot)\rVert_{L^2(\mathbb{R}^3)}  \lesssim \frac{\varepsilon}{\langle{}_{\,} t\rangle^{{1}/{2}-\mu}}\qquad (t\geq 0) \,.
  \end{equation}

\end{prop}

First note that   the higher order equations for $v_T$ are
\begin{equation}\label{eq:Box:Z:v}
  \Box {Z}^{I} v_T= \sum_{|J|+|K|\leq |I|} c^{IJ}_K Z^J(\chi(\tfrac{t}{T})) F_K
\end{equation}
where $c_{K}^{IJ}$ are constants and $F_{K}=Z^KQ(\pa \phi_T,\pa \phi_T)$ satisfies the following property:

\begin{lemma}\label{lemma:F:alpha}

  \begin{equation}
  |F_{I}| \lesssim \sum_{ |{J}|+|{K}|\leq |{I}|} \Bigl(|\partial {Z}^{J} v_T|\bigl(|\overline{\partial}{Z}^{K} v_T|+ |\overline{\partial} {Z}^{K}  \phi_0|\bigr) + |\partial {Z}^{J}  \phi_0|\bigl(|\overline{\partial}{Z}^{K}  \phi_0|+|\overline{\partial}{Z}^{K}  v_T|\bigr)\Bigr).
\end{equation}

\end{lemma}

\begin{proof}
  This follows from standard properties of null forms, namely that if $Q$ is a null from then ${Z} Q(\partial \phi,\partial v)=Q(\partial {Z} \phi,\partial v)+Q(\partial \phi,\partial {Z} v)+\tilde{Q}(\partial \phi,\partial v)$ where $\tilde{Q}$ is another null form, and that for all null forms we have the pointwise bound: $|Q(\partial \phi,\partial v)|\lesssim |\partial \phi||\overline{\partial} v|+|\overline{\partial }\phi||\partial v|$.
\end{proof}

Moreover since we take $\phi_0$ to be compactly supported on $t=0$, $F_I$ vanishes for $|x|\gtrsim T$. Thus in the support of the right hand side of  \eqref{eq:Box:Z:v} the $Z$-derivatives of the cutoff are bounded:

\begin{lemma} \label{lemma:Z:chi}
  For $|x|\lesssim T$,
  \begin{equation}
    \bigl| Z^I \chi(\tfrac{t}{T}) \bigr|  \lesssim \chi(\tfrac{t}{2T})\,.
  \end{equation}
\end{lemma}
\begin{proof}
  The only vectorfields for which $Z\chi(t/T)$ does not vanish are scalings, boosts, and time translations. For $Z=t\partial_t+x^i\partial_i$ we have $Z\chi=(t/T)\chi'$, and for $Z=x^i\partial_t+t\partial_i$, $Z\chi=(x^i/T)\chi'$, and in general $|Z^I \chi|\lesssim 1$ for $(t+|x|)/T \lesssim 1$. Note that $\chi(t/2T)=1$ in the support of $\chi^{(k)}(t/T)$.

\end{proof}

The reason we have introduced the cutoff is to ensure that not only $v_T$ but also $Z^I v_T$ has trivial data at $t=T$. To see this we estimate the boundary terms pointwise using  that $v_T$ satisfies an equation:

 \begin{lemma}\label{lemma:Z:v}
   Let $v_T$ satisfy the equation $-\Box v_T=g$ and have trivial initial data on $t=T$.
   Then
   \begin{equation}
     |\partial (Zv_T) |\lesssim \langle t+r\rangle |g|  \qquad \text{: on } t=T\,.
   \end{equation}
Moreover if $g$ vanishes identically for $t>T/2$, the $\pa( Z^I v_T)=0$ for all $|I|\geq 0$ at $t=T$.

 \end{lemma}

 \begin{proof}
Let us first consider the more general situation:
\begin{gather}
  -\Box u=f_2\\
  u\rvert_{t=T}=f_0\qquad \partial_t u\rvert_{t=T}=f_1
\end{gather}
We compute $\partial(Zu)$ at $t=T$ for all vectorfields $Z=\partial_t,\partial_i,\Omega_{ij},\Omega_{0j},S$:
\begin{align}
  \label{eq:2}
  \partial_t(\partial_tu)=&-\Box u+\Delta u=f_2+\Delta f_0\\
                            \partial_i(\partial_t u)=&\partial_if_1
\end{align}
Here and below we use the equation satisfied by $u$ to eliminate second time derivatives of $u$. For all tangential derivatives to $t=T$ we use that $u=f_0$ and $\partial_tu=f_1$ are known on $t=T$.
\begin{align}
  \label{eq:3}
  \partial_t(\Omega_{ij}u)=&\Omega_{ij}f_1\\
  \partial_k(\Omega_{ij}u)=&\delta^i_k\partial_j f_0+\delta^j_k\partial_if_0+\Omega_{ij}\partial_kf_0
\end{align}
Moreover
\begin{align}
  \label{eq:4}
  \partial_t(\Omega_{0j}u)=&\partial_jf_0+t\partial_jf_1+x^j\bigl(f_2+\Delta f_0\bigr)\\
  \partial_i(\Omega_{0j}u)=&t\partial_j\partial_i f_0+\delta^j_if_1+x^j\partial_if_1
\end{align}
and
\begin{align}
  \label{eq:5}
  \partial_t(Su)=&f_1+t(f_2+\Delta f_0)+x^j\partial_j f_1\\
  \partial_i(Su)=&t\partial_i f_1+\partial_if_0+x^j\partial_j\partial_i f_0
\end{align}

We can apply this formula to $v_T$ which corresponds to the case $f_0=f_1=0$, and $f_2=g$. This immediately yields for some constants $c_\beta$:
\begin{equation}
  \partial(Zv)=\sum_{|\beta|\leq 1}c_\beta x^\beta g
\end{equation}
and in particular
\begin{equation}
  \lvert \partial(Zv) \rvert \lesssim \langle t+r\rangle |g|\,.
\end{equation}

Now consider the case that $g=0$ vanishes identically for $t>T/2$.
Then by the commutation properties of the vectorfields $\Box Z^I v_T=0$ at $t=T$ for all $|I|\geq 1$, and we can apply the above to $u=Z^I v_T$ and proceed by induction on $|I|=k\geq 1$ to show  $u$ has trivial initial data at $t=T$. Indeed if $u=Z^I v_T$ has trivial data at $t=T$ for $|I|=k$, then since $\Box u=0$ the above implies $\pa (Z u)=0$ at $t=T$ hence $Z^Iv_T$ has trivial data for $|I|=k+1$.

 \end{proof}

We are now ready to prove the main result of this section.

\begin{proof}[Proof of Proposition~\ref{prop:wave:null}]

  We set
  \begin{equation*}
    \phi_T=\phi_0+v_T
  \end{equation*}
  where $\phi_0$ is a given linear solution and $v_T$ is a solution to \eqref{eq:Box:v:T:chi} with trivial data at $t=T$.

  Therefore by Corollary~\ref{cor:w:D} applied to $\phi_0$,  we have that for all $t<T$,
    \begin{equation*}
    \sum_{|{I}|\leq 5}\lVert \partial {Z}^{I} \phi_0 \rVert_{\mathrm{L}^2(\Sigma_{t})}+\sum_{|{I}|\leq 5}\Bigl(\int_t^T\int_{\mathbb{R}^3} \frac{\mu}{4}\frac{ |\pl {Z}^{I} \phi_0|^2}{(1\!+|q|)^{1+2\mu}}\ud t \ud x\Bigr)^{{1}/{2}} \leq \sum_{|{I}|\leq 5}\lVert \partial {Z}^{I} \phi_0 \rVert_{\mathrm{L}^2(\Sigma_{T})}
      \end{equation*}
      and similarly applied $v_T$ we have,  for any $\mu>0$,
      \begin{multline*}
    \sum_{|{I}|\leq 5}\lVert \partial {Z}^{I} v_T \rVert_{\mathrm{L}^2(\Sigma_{t})}+\sum_{|{I}|\leq 5}\Bigl(\int_t^T\int_{\mathbb{R}^3} \frac{\mu}{4}\frac{ |\pl {Z}^{I} v_T|^2}{(1\!+|q|)^{1+2\mu}} \ud x \ud t\Bigr)^{{1}/{2}} \leq \\\leq  C\Bigl(   \sum_{|{I}|\leq 5}\lVert \partial {Z}^{I} v_T \rVert_{\mathrm{L}^2(\Sigma_{T})} +  \sum_{|{I}|\leq 5}\int_{t}^{T}  \lVert \Box Z^I v_T \rVert_{L^2(\Sigma_t)}\Bigr)
  \end{multline*}

The boundary term for $Z^Iv_T$ at $t=T$ vanishes by Lemma~\ref{lemma:Z:v}.

For the estimate of the bulk term note that the right hand side of \eqref{eq:Box:Z:v} is compactly supported: Since the initial data for $\phi_0$ is compactly supported at $t=0$, the support of $\phi_0$ is contained in $|x|\gtrsim t$. Moreover $v_T$ has trivial data at $t=T$, hence  $v_T$ satisfies $|\Box v_T|\leq|Q(\pa v_T,\pa v_T)|$ for $|x|\gtrsim T$ and thus vanishes for $|x|\gtrsim 2T$, for all $0\leq t\leq T$. Therefore by Lemma~\ref{lemma:Z:chi} we have
\begin{equation}
  \lVert \Box Z^I v_T(t,\cdot) \rVert_2\lesssim \sum_{|K|\leq |I|}\lVert  F_K (t,\cdot) \rVert_2\,.
\end{equation}
  which we will estimate below using Lemma~\ref{lemma:F:alpha}.

 \paragraph{Bootstrap assumption:}
 Now assume that for some  $\delta\geq 0$,
  \begin{equation}\label{eq:bootstrap:v}
  \sum_{|{I}|\leq 5} \lVert \partial {Z}^{I} v_T \rVert_{L^2(\Sigma_t)}+\sum_{|{I}|\leq 5}\Bigl(\int_{t}^T\int_{\mathbb{R}^3} \frac{\mu}{4}\frac{ |\pl {Z}^{I} v_T|^2}{(1\!+|q|)^{1+2\mu}}\ud t' \ud x  \Bigr)^{{1}/{2}}  \leq \frac{\varepsilon}{\langle t\rangle^\delta}\tag{$\ast$} .
\end{equation}

This is evidently true at $t=T$ where the left hand side vanishes.
We will show that \eqref{eq:bootstrap:v} holds for all $0\leq  t\leq T$, provided $0\leq \delta<{1/2}-\mu$, and $\varepsilon>0$ is sufficiently small. In the remainder of the proof  we will drop the subscript in $v_T$ for ease of notation.

From the Klainerman-Sobolev inequality it then follows that
  \begin{equation*}
    \langle\, t\!-\!r\rangle^{{1}/{2}}\sum_{|{I}|\leq 3}|\partial {Z}^{I} v|    \leq \frac{\varepsilon C_S}{\langle\, t\!+r\rangle \langle t\rangle^{\delta}}
  \end{equation*}
  and by Lemma~\ref{lemma:pl:sigma}
  \begin{equation*}
    \sum_{|{I}|\leq 1}|\pl {Z}^{I} v|\leq \frac{C_S \varepsilon}{\langle t\rangle^{{3}/{2}+\delta}} .
  \end{equation*}

\paragraph{Spacetime terms:}
Consider the first string of terms for $F_I$  in Lemma~\ref{lemma:F:alpha}:
      \begin{equation*}
        \sum_{|{J}|+|{K}|\leq 5}\!\!\!\!\! \lVert |\partial {Z}^{J} v||\overline{\partial}{Z}^{K}  v|\rVert_{L^2(\Sigma_t)}\leq       \sum_{|{J}|\leq 3,|{K}|\leq 5}\!\!\!\!\!  \lVert |\partial {Z}^{J} v||\overline{\partial}{Z}^{K}  v|\rVert_{L^2(\Sigma_t)}+ \!\!\!\!\!    \sum_{|{J}|\leq 5,|{K}|\leq 1}\!\!\!\!\!  \lVert |\partial {Z}^{J} v||\overline{\partial}{Z}^{K}  v|\rVert_{L^2(\Sigma_t)}.
      \end{equation*}
      We will estimate these terms given the improved decay rates of $\pl v$, over $\partial v$, first in the $L^\infty$, then in the $L^2$-sense, as expressed in Corollary~\ref{cor:w:D} and Lemma~\ref{lemma:pl:sigma}  above.

      Firstly, using the $L^\infty$ estimate,
            \begin{multline*}
              \sum_{|{I}|\leq 5,|{J}|\leq 1}\int_t^T\lVert |\partial {Z}^{I} v || \overline{\partial}{Z}^{J}  v |\rVert_{L^2(\Sigma_{t'})}\ud t' \leq \sum_{|{I}|\leq 5,|{J}|\leq 1}\int_t^T\lVert \overline{\partial}{Z}^{J} v\rVert_{L^\infty(\Sigma_{t'})}\lVert \partial {Z}^{I} v\rVert_{L^2(\Sigma_{t'})}\ud t'\\
              \leq \sum_{|{I}|\leq 5}\int_t^T\frac{C_S\varepsilon}{\langle t'\rangle^{{3/2}+\delta}}\lVert \partial {Z}^{I} v\rVert_{L^2(\Sigma_{t'})}\ud t'
              \leq C_S\varepsilon^2\int_t^\infty \frac{1}{\langle t'\rangle^{{3/2}+2\delta}}\ud t' \leq \frac{\varepsilon^2C_S}{\langle t\rangle^{{1/2}+2\delta}}.
            \end{multline*}
            Secondly, using the $L^2$ estimate,
      \begin{multline*}
          \sum_{|{I}|\leq 3,|{J}|\leq 5}\int_t^T\lVert |\partial {Z}^{I}  v || \overline{\partial}{Z}^{J}  v |\rVert_{L^2(\Sigma_{t'})}\ud t' \leq \sum_{|{I}|\leq 3,|{J}|\leq 5}\int_t^T\Bigl(\int_{\mathbb{R}^3} |\partial {Z}^{I} v |^2| \overline{\partial}{Z}^{J} v |^2 \ud x\Bigr)^{1/2}\ud t'\\
          \leq \sum_{|{I}|\leq 3,|{J}|\leq 5}\int_t^T\sup_{\mathbb{R}^3}\Bigl[(1+(|x|-t'))^{{1/2}+\mu}|\partial {Z}^{I} v |\Bigr]\Bigl(\int_{\mathbb{R}^3} \frac{| \overline{\partial}{Z}^{J} v |^2}{(1+||x|-t'|)^{1+2\mu}} \ud x\Bigr)^{1/2}\ud t'\displaybreak[0]\\
          \leq \sum_{|{I}|\leq 3} \Bigl( \int_t^T \bigl[\sup_{\mathbb{R}^3}(1+(|x|-t'))^{{1/2}+\mu}|\partial {Z}^{I} v |\bigr]^2\ud t' \Bigr)^{1/2}\sum_{|{J}|\leq 5} \Bigr(\int_t^T\int_{\mathbb{R}^3} \frac{| \overline{\partial}{Z}^{J} v |^2}{(1+||x|-t'|)^{1+2\mu}} \ud x\ud t'\Bigr)^{1/2}\\
          \leq  \frac{2\varepsilon C_S}{\sqrt{\mu}}\Bigl(\int_{t}^\infty\frac{1}{\langle t'\rangle^{2+2\delta-2\mu}}\ud t'\Bigr)^{1/2}\frac{\varepsilon}{\langle t\rangle^\delta}
          \leq  \frac{2}{\sqrt{\mu(1+2\delta-2\mu)}}\frac{\varepsilon^2C_S}{\langle t\rangle^{{1/2}+2\delta-\mu}}.
      \end{multline*}

      Now we proceed similarly for the other terms in $F_{I}$.
      Indeed, for the purely linear terms, the above estimates formally apply if we replace $v$ by $\phi_0$, and set $\delta=0$.
      This yields:
      \begin{equation*}
          \int_t^T\sum_{|{I}|+|{J}|\leq 5} \lVert |\partial {Z}^{I} \phi_0||\overline{\partial}{Z}^{J}  \phi_0|\rVert_{L^2(\Sigma_{t'})}\ud t'\leq \frac{C_S}{\langle t\rangle^{{1/2}-\mu}}  \sum_{|{I}|\leq 5}\lVert \partial {Z}^{I} \phi_0 \rVert_{L^2(\Sigma_{T})}^2.
      \end{equation*}

      Similarly,
      \beq
        \sum_{|{I}|+|{J}|\leq 5} \int_t^T\lVert |\partial {Z}^{I} \phi_0||\overline{\partial}{Z}^{J}  v|\rVert_{L^2(\Sigma_{t'})}\ud t'\leq \frac{\varepsilon C_S}{\langle t\rangle^{{1/2}+\delta-\mu}}  \sum_{|{I}|\leq 5}\lVert \partial {Z}^{I} \phi_0 \rVert_{L^2(\Sigma_{T})},
        \eq
and
        \beq
                  \sum_{ |{I}|+|{J}|\leq 5} \int_t^T\lVert |\partial {Z}^{I} v||\overline{\partial}{Z}^{J}  \phi_0|\rVert_{L^2(\Sigma_{t'})}\ud t'\leq \frac{\varepsilon C_S}{\langle t\rangle^{{1/2}+\delta-\mu}}  \sum_{|{I}|\leq 5}\lVert \partial {Z}^{I} \phi_0 \rVert_{L^2(\Sigma_{T})}.
      \eq

      In conclusion,
      \begin{equation}\label{eq:F:null:spacetime}
        \sum_{|{I}|\leq 5}\int_{t}^{T}  \lVert F_{I} \rVert_{L^2(\Sigma_{t'})} \ud t'\leq \frac{CC_S}{\langle t\rangle^{{1/2}-\mu}} \Bigl(\frac{\varepsilon}{\langle t\rangle^\delta} + \sum_{|{I}|\leq 5}\lVert \partial {Z}^{I} \phi_0 \rVert_{L^2(\Sigma_{T})}\Bigr)^2\,.
      \end{equation}

      Therefore, 
      we have shown that
    \begin{multline*}
      \sum_{|{I}|\leq 5}\lVert \partial {Z}^{I} v_T \rVert_{\mathrm{L}^2(\Sigma_{t})}+\sum_{|{I}|\leq 5}\Bigl(\int_t^{T}\int_{\mathbb{R}^3} \frac{\mu}{4}\frac{ |\pl {Z}^{I} v_T|^2}{(1+|q|)^{1+2\mu}} \ud x \ud t\Bigr)^{1/2} \\ \lesssim   
      \sum_{|{I}|\leq 5}\int_{t}^{T}  \lVert F_{I} \rVert_{L^2(\Sigma_t)}\ud t 
      \leq \frac{CC_S}{\langle t\rangle^{{1/2}-\mu}} \Bigl(\frac{\varepsilon}{\langle t\rangle^\delta} + \sum_{|{I}|\leq 5}\lVert \partial {Z}^{I} \phi_0 \rVert_{L^2(\Sigma_{T})}\Bigr)^2\,,
      \end{multline*}
      which recovers the bootstrap assumption \eqref{eq:bootstrap:v}, provided $0\leq \delta<{1/2}-\mu$, and
      \begin{equation}
       CC_S\Bigl(\frac{\varepsilon}{\langle t\rangle^\delta} + \sum_{|{I}|\leq 5}\lVert \partial {Z}^{I} \phi_0 \rVert_{L^2(\Sigma_{T})}\Bigr)^2\leq  4 C C_S \varepsilon^2< \varepsilon\,,
      \end{equation}
      which can be achieved by choosing $\varepsilon<(4CC_S)^{-1}$.

    \paragraph{Existence and uniqueness of the limit:}
    To define $v=\lim_{T\to\infty}v_T$ consider two solutions $v_i=v_{T_i}$ to \eqref{eq:Box:v:T:chi} with trivial data at $t=T_i$, $i=1,2$, $T_2>T_1$. The difference $w=v_{1}-v_{2}$ satisfies the equation
    \begin{equation}\label{eq:Box:w:chi}
      \Box w=\chi(t/T_1)\bigl( Q(\pa \phi_0,\pa w)+Q(\pa w,\pa \phi_0) \bigr)+\bigl(\chi(t/T_2)-\chi(t/T_1)\bigr)Q(\pa (\phi_0+v_2),\pa (\phi_0+v_2)\bigr)
    \end{equation}
    We have by the energy estimate that for $t<T_1$,
    \begin{equation}
        \lVert \pa w(t,\cdot)\rVert_2\lesssim \lVert \pa w(T_1,\cdot) \rVert_2+ \int_t^{T_1}\lVert \Box w(t',\cdot) \rVert_2\ud t'
    \end{equation}
    and since $v_1$ has trivial data at $t=T_1$, we have precisely shown above that
    \begin{equation}
      \lVert \pa w(T_1,\cdot)\rVert = \lVert \pa v_2 (T_1,\cdot)\rVert_2 \lesssim D_5^2\langle T_1\rangle^{-1/2+\mu}
    \end{equation}

    Now consider the second term in the estimate for $\|\pa w(t,\cdot)\|$, namely the integral of $\lVert \Box w (t,\cdot) \rVert$.
    For the first term in \eqref{eq:Box:w:chi} we use the pointwise decay of the linear solution, $|\pa \phi_0|\lesssim D_2 \langle t\rangle^{-1}$ to infer that
    \begin{equation}
       \int_t^{T_1}\lVert \chi(t'/T_1)\bigl( Q(\pa \phi_0,\pa w)+Q(\pa w,\pa \phi_0)  \rVert_2\ud t'\lesssim \int_t^{T_1}\frac{D_5}{\langle t'\rangle} \lVert \pa w (t',\cdot)\rVert_2 \ud t'
    \end{equation}
    The second term in \eqref{eq:Box:w:chi} is only supported when $t>T_1/8$ (and $t<T_2/4$) and thus using the spacetime estimate established in \eqref{eq:F:null:spacetime}, which holds for $v_2$:
    \begin{equation}
      \int_{T_1/8}^{T_1}\lVert Q(\pa (\phi_0+v_2),\pa (\phi_0+v_2)) \rVert_2\ud t' \lesssim D_5^2\langle T_1\rangle^{-1/2+\mu}
    \end{equation}

    Similarly, for the higher orders
        \begin{multline}\label{eq:Box:Z:w:chi}
      \Box Z^Iw=\sum_{|J|+|K|\leq |I|}c^{IJ}_KZ^J(\chi(t/T_1))Z^K\bigl( Q(\pa \phi_0,\pa w)+Q(\pa w,\pa \phi_0) \bigr)\\+\sum_{|J|+|K|\leq |I|}c^{IJ}_K Z^J\bigl(\chi(t/T_2)-\chi(t/T_1)\bigr)Z^KQ(\pa (\phi_0+v_2),\pa (\phi_0+v_2)\bigr)
    \end{multline}
    we note that in the first term the vectorfields falling on the cutoff do not produce any weights by Lemma~\ref{lemma:Z:chi}, and we take the linear solution in $L^\infty$, $|\pa Z^I \phi_0|\lesssim D_7$, for $|I|\leq 5$. Moreover the second term is only supported for $t>T_1/8$, and controlled in $L^1_tL^2_x$ by \eqref{eq:F:null:spacetime} up to $|I|\leq 5$.
Finally by Lemma~\ref{lemma:Z:v} we have $\|\pa Z^I w (T_1,\cdot)\|=\|\pa Z^I v_2 (T_1,\cdot)\|_2\lesssim D_5^2\langle T\rangle^{-1/2+\mu}$.

    In the above estimates the constants are independent of $T_1$, and $T_2$.
    In view of the small data assumption $D_7\lesssim \varepsilon<1$ we have therefore shown that $E(t)= \sum_{|I|\leq 5}\|\partial  Z^Iw (t,\cdot)\|_2$ satisfies
\begin{equation}\label{ineq:energy:difference}
     E(t) \leq \varepsilon \langle T_1\rangle^{-1/2+\mu} + \int_t^{T_1} \frac{\varepsilon}{\tau} E(\tau) \ud\tau \qquad (t<T_1)\,,
   \end{equation}
   and thus  we conclude  by Gr\"onwall's inequality that
   \begin{equation}
     E(t)
     \leq (T_1/t)^\varepsilon \varepsilon \langle T_1\rangle^{-1/2+\mu}\to 0 \text{ as }T_1\to\infty \text{ if } \varepsilon+\mu <1/2\,.
   \end{equation}

   Morover the limit is unique, for if $\phi_0=0$ then by \eqref{eq:bootstrap:v} $\|\partial Z^I v_T(t,\cdot)\|_2=0$, so $\|\partial Z^I \phi(t,\cdot)\|_2\leq \|\partial Z^I (\phi-\phi_T)(t,\cdot)\|_2\to 0$ as $T\to\infty$, for all $|I|\leq 5$.
      \end{proof}

\section{The classical null condition model revisited}
\label{sec:classical:null:revisited}

In Section~\ref{sec:classical:null} we considered a scalar equation satisfying the classical null condition.
We now revisit the null condition model using the approximate solutions introduced in Section~\ref{sec:scattering},
for the following system:
\begin{subequations}\label{eq:nullcondition:system}
\begin{align}
\Box \psi &= Q(\pa \psi,\pa \phi)\label{eq:nullcondrevisited1}\\
 \Box \phi &= \widetilde{Q}(\pa \psi,\pa \phi)\label{eq:nullcondrevisited2}
\end{align}
\end{subequations}

The main result of this Section is the following more refined version of Theorem~\ref{thm:intro:nullcond} stated in the Introduction:

\begin{prop}\label{prop:nullcond:revisited}
  Let $F_0(q,\omega)$, $G_0(q,\omega)$, and $M,N>0$ be given such that for some $n\geq 5$, $k\geq 0$, $1/2<\gamma<1$, and $0<\varepsilon<1/2+\gamma$,
 \begin{equation}
   D_k:=\|F_0\|_{n+k+1,\gamma-1/2}+M+\|G_0\|_{n+k+1,\gamma-1/2}+N<\varepsilon\,.
 \end{equation}
 Then for $\varepsilon>0$ sufficiently small there exists a unique solution  $(\psi=\psi_{01e}+v,\phi=\phi_{01e}+u)$  of the system \eqref{eq:nullcondition:system} for $t\geq 0$,  where $\psi_{01e}$, and $\phi_{01e}$ are second order approximate solutions associated to the radiation fields $F_0$, $G_0$, and masses $M$, $N$. In fact, $\psi_{01e}$, $\phi_{01e}$ are given by \eqref{eqs:psi:01e:prime} and \eqref{eq:psi:i:prime}, where $F_1$, and $G_1$ satisfy the equations \eqref{eq:asymptoticnullcond1}, \eqref{eq:asymptoticnullcond2}, respectively.
Moreover, $u$, $v$ satisfy
 \begin{align}
   \|\pa Z^I v(t,\cdot)\|+\|\pa Z^I u(t,\cdot)\|&\lesssim \frac{\varepsilon}{\langle\, t\rangle^{1/2+\gamma}}\qquad (|I|\leq k)\label{prop:nullcond:system:eq:energybound}\\
  \|v(t,\cdot)\|_{1+k,s-1}  +\|u(t,\cdot)\|_{1+k,s-1}&\lesssim \frac{\varepsilon}{\langle t\rangle^{1/2+\gamma-s}} \qquad (0<s<1/2+\gamma)\label{prop:nullcond:system:eq:confenergybound}
 \end{align}
 and in particular for $k\geq 1$,
 \begin{equation}
   |Z^I v|+|Z^I u| \lesssim \frac{\varepsilon}{\langle t+r\rangle \langle t\rangle^{\gamma}}\qquad (|I|\leq k-1)\,.
 \end{equation}
\end{prop}

As in Sections~\ref{sec:classical:null}, \ref{sec:scattering} we construct the scattering solutions as limit of a sequence $(\psi_T=\psi_{01e}+v_T,\phi_T=\phi_{01e}+u_T)$ as $T\to\infty$, where $(\psi_T,\phi_T)$ are solutions to the system \eqref{eq:nullcondition:system} with \emph{trivial initial data} for $v_T$, and $u_T$, at $t=T$.

\subsection{Asymptotics for derivatives and the exterior mass term}
\label{sec:asymptotics:derivative}

Given  a solution $\psi$ to the \emph{linear} equation $\Box\psi=0$, the derivatives $\partial_t\psi$, $\partial_i\psi$ are themselves solutions of the wave equation and have a radiation field. The asymptotics of  the time derivative  are, c.f.~\eqref{eq:radiationfield},
\begin{equation}
\label{eq:radiationfieldderivative}
\partial_t \psi\sim -\frac{F_0^\prime(r-t,\omega)}{r} -\psi_{e}^\prime,\qquad  \psi_{e}^\prime =M\chi_e^\prime(r-t)/r \,,
\end{equation}
where $F_0^\prime(q,\omega)=\pa_q F_0(q,\omega)$, and $\chi_e'(s)$ is supported only for $1\leq s\leq 2$.
 The corresponding formula to the next order, c.f.~\eqref{eq:improvedlinearexpansionpsi01}, is
\begin{equation}
\partial_t\psi\sim -\frac{\!F_0^\prime(r\!-\!t,\omega)\!\!}{r}
-\frac{F_{\!1}^\prime(r\!-\!t,\omega)}{r^2}-\psi_e^\prime.
\end{equation}

This already motivates the definition of the following \emph{approximate solutions}; recall here also the form of the linear expansions \eqref{eq:linearexpansionpsi0}, \eqref{eq:improvedlinearexpansionpsi01}:
 \begin{gather}
\psi_{0e} := \psi_0+ \psi_e,\qquad
\psi_{0e}^\prime :=\psi_0^\prime +\psi_e^\prime, \label{eqs:psi:01e:prime}\\
\psi_{01e}:={\psi_0+\psi_1} +\psi_e,\qquad
\psi_{01e}^\prime :=\psi_0^\prime+\psi_1^\prime +\psi_e^\prime,
\end{gather}
where as in \eqref{eq:improvedlinearexpansionpsi:prime} above, we define
\begin{equation}\label{eq:psi:i:prime}
\psi_i^{\,\prime}:=\frac{\!F_i^{\,\prime}(r\!-\!t,\omega)\!\!}{r^{1+i}}\chi\big(\tfrac{\langle t-r\rangle}{r}\big)
\end{equation}
and $\chi$ is a smooth  function  such that $\chi(s)=1$ for $s\leq 1/8$, and $\chi(s)=0$ for $s\geq 1/4$, as in Section~\ref{sec:scattering}.
Note these expressions include the mass term in the exterior $\psi_e$.

The formulas for the asymptotics of the  spatial derivatives are slightly more involved but are obtained simply by differentiating \eqref{eq:improvedlinearexpansion}:
\begin{multline}
\partial_i\psi\sim \Big(\frac{\!F_0^\prime(r\!-\!t,\omega)\!\!}{r}
+\frac{F_{\!1}^\prime(r\!-\!t,\omega)\!-\!F_0(r\!-\!t,\omega)}{r^2}
-\frac{2F_{\!1}(r\!-\!t,\omega)}{r^3}+\psi_e^\prime-\frac{\psi_e}{r}\Big)\omega^i\\
+\frac{\pl_{\omega^i} F_{0}(r\!-\!t,\omega)\!\!}{r^2}
+\frac{\pl_{\omega^i} F_{1}(r\!-\!t,\omega)\!\!}{r^3},
\end{multline}
where $F^{\,\prime}_i(q,\omega)=\pa_q F_i(q,\omega)$, $\pl_{\omega^i}F(q,\omega)=(\delta_{i}^j-\omega^i\omega^j)\partial_{\omega^j}F(q,\omega)$, and we used that 
\begin{equation}\label{eq:xi:omegai}
    \partial_j\bigl( F(q,\omega) \bigr) = \partial_{j}\bigl(F(t-|x|,x/|x|)\bigr) = -F'(q,\omega)\omega^j+\frac{\delta_j^i-\omega^i\omega^j}{r}(\partial_{\omega^i}F)(q,\omega)\,;
\end{equation}
note in particular that $\Omega_{ij}F(q,\omega)=\omega^i\partial_{\omega^j}F-\omega^j\partial_{\omega^i}F$.

For simplicity in notation we will denote $\pl_{\omega^i}$ simply by $\partial_{\omega^i}$.
Moreover
 \begin{equation}\label{eq:psi:01:omega}
 \pa_{\omega^j}\psi_{0e}=\pa_{\omega^j}\psi_{0}, \qquad
 \pa_{\omega^j}\psi_{01e}=\pa_{\omega^j}\psi_{01}=\pa_{\omega^j}\psi_{0}+\pa_{\omega^j}\psi_{1},
\end{equation}
and we summarize our choice for the approximate solutions for the derivatives in
 \begin{equation}
 \psi_{01e,\,0}^\prime:=-\psi_{01e}^\prime,\qquad
  \psi_{01e,\,i}^\prime=\omega^i\psi_{01e}^\prime
  -\omega^i(\psi_{0e}\!+2\psi_1)/r+\pa_{\omega^i}\psi_{01}/r . \label{eqs:psi:01e:prime:components}
\end{equation}

\smallskip
In this section we will frequently use that the following comparisons hold in the support of the cut-off functions:
\begin{lemma}\label{lemma:cutoff}
  In the support of $\chi(\langle t-r\rangle/r)$ it holds $r^{-1}\lesssim \langle t-r\rangle^{-1}$, and $r^{-1}\lesssim \langle t+r\rangle^{-1}$. Moreover in the support of $\chi'(\langle t-r\rangle/r)$ we have in addition that $r\lesssim \langle t-r\rangle$. Finally, in the support of $\chi_e(r-t)$ we have $r^{-1}\lesssim \langle t+r\rangle^{-1}\lesssim \langle t-r\rangle^{-1}$, and the support of $\chi_e'(r-t)$ is contained in the support of $\chi(\langle t-r\rangle/r)$ for $t\gtrsim 1$.
\end{lemma}

\begin{lemma}\label{lemma:psi:01e:prime}
  We have
\begin{equation}\label{eq:timedifferentiatedcutofferror}
\big|\pa_\alpha\psi_{01e}-\psi_{01e,\,\alpha}^\prime  \big|
\lesssim \frac{\big|\chi^\prime\big(\tfrac{\langle\, t-r\rangle}{r}\big)\big|}{\langle\, t+r\rangle^3}
\big( \langle\, q\rangle
\big| F_0(q,\omega)\big| + \big|F_1(q,\omega)\big|\big).\,
\end{equation}

Moreover
\begin{equation}
  |\psi_{01e,\alpha}'| \lesssim \frac{\chi(\tfrac{\langle t-r\rangle}{r})}{\langle t+r\rangle}\frac{1}{\langle q\rangle}\biggl[\sum_{\substack{i=0,1\\ |\beta|+k\leq 1}}\langle q\rangle^{-i}|(\langle q\rangle\partial_q)^k\partial_\omega^\beta F_i(q,\omega)|+M\biggr]
\end{equation}
and
\begin{equation}\label{eq:psi:01e:higher}
  \bigl|\partial Z^I \psi_{01e}\bigr| \lesssim \frac{\chi(\tfrac{\langle t-r\rangle}{r})}{\langle t+r\rangle}\frac{1}{\langle q\rangle}\biggl[\sum_{\substack{i=0,1\\ |\beta|+k\leq 1+|I|}}\langle q\rangle^{-i}|(\langle q\rangle\partial_q)^k\partial_\omega^\beta F_i(q,\omega)|+M\biggr]
\end{equation}
\end{lemma}
\begin{proof}
  We have
\begin{equation*}
\pa_t ({\psi_0+\psi_1} +\psi_e)= -({\psi_0^\prime+\psi_1^\prime}+\psi_e^\prime)
+\Big(\frac{\!F_0(r\!-\!t,\omega)\!\!}{r}+\frac{F_{\!1}(r\!-\!t,\omega)}{r^2}\Big)\chi^{\,\prime}\big(\tfrac{\langle \,t-r\rangle}{r}\big)
\tfrac{t-r}{\langle \,t-r\rangle r}\,.
\end{equation*}
The bound  \eqref{eq:timedifferentiatedcutofferror}  in the case $\alpha=0$ then follows from Lemma~\ref{lemma:cutoff}.
Similarly, for $\alpha=i=1,2,3$,
\begin{equation*}
  \partial_i\psi_{01e}-\psi_{01e,i}^{\prime}=-\Big(\frac{\!F_0(r\!-\!t,\omega)\!\!}{r}+\frac{F_{\!1}(r\!-\!t,\omega)}{r^2}\Big)\chi^{\,\prime}\big(\tfrac{\langle \,t-r\rangle}{r}\big)\tfrac{(t-r)+\langle t-r\rangle^2}{\langle t-r\rangle r^2}\omega^i\,.
\end{equation*}

For the second estimate recall the support of the cutoff functions from Lemma~\ref{lemma:cutoff}.
We have
\begin{align*}
  |\psi_{01e,0}'| &=  |\psi_{01e}'|\lesssim \frac{|\chi|}{\langle r+t\rangle}\Bigl(|F_0'|+\langle q\rangle^{-1}|F_1'|\Bigr)+\frac{\langle q\rangle|\chi_e'|}{\langle r+t\rangle}\frac{M}{\langle q\rangle} \\
  |\psi_{01e,\,i}^\prime | &\leq |\psi_{01e}^\prime| +  \frac{|\chi|}{\langle t+r\rangle}\frac{|F_0|+\langle q\rangle^{-1}|F_1|+|\pa_{\omega^i}F_0|+\langle q\rangle^{-1}|\pa_{\omega^i}F_1|}{\langle q\rangle}+ \frac{|\chi_e|}{\langle t+r\rangle}\frac{M}{\langle q\rangle}\,.
\end{align*}

For the higher order version note in particular that
\begin{equation}
  |Z F_i(r-t,\omega)|\lesssim \sum_{k+|\alpha|=1} |(\langle q\rangle \partial_q)^k\partial_\omega^\alpha F_i|\,.
\end{equation}

\end{proof}

\subsection{Estimates for bilinear forms of the approximate solution}

Consider approximate solutions $\psi_{01e}=\psi_0+\psi_1+\psi_e$, and $\phi_{01e}=\phi_0+\phi_1+\phi_e$, associated to the radiation fields $F_i$, and $G_i$, and mass $M$, and $N$, respectively; c.f.~\eqref{eqs:psi:01e:prime}.
Recall also the norms \eqref{eq:decayassumption}, \eqref{eq:decayassumptioninfty} on the radiation fields introduced in Section~\ref{sec:scattering}.

\begin{lemma} \label{lem:Blemma} Let $B$ be a bilinear form. Then
\begin{multline}
  \big\|\langle t+r\rangle^s \big|B(\pa \psi_{01e},\pa \phi_{01e})-B(\psi_{01e}^\prime,\phi_{01e}^{\,\prime})\big|(t,\cdot)\big\|_{\mathrm{L}^2(\Sigma_t)}\lesssim\\
  \lesssim \frac{1}{\langle t\rangle^{5/2-s}}\sum_{i,j=0}^1 \Bigl(\|F_i\|_{0,-1/2-i}\bigl(\|G_j\|_{1,\infty,-j}+N\bigr)+\bigl(\|F_i\|_{1,\infty,-i}+M\bigr)\|G_j\|_{0,-1/2-j}\Bigr)\,.
\end{multline}
\end{lemma}

\begin{proof}
  Since $B$ is bilinear, $B(\xi,\zeta)=B^{\alpha\beta}\xi_\alpha\zeta_\beta$ for some coefficients $B^{\alpha\beta}$, we have
  \begin{equation*}
    \begin{split}
      B(\pa \psi_{01e},\pa \phi_{01e})-B(\psi_{01e}^\prime,\phi_{01e}^{\,\prime}) &= B^{\alpha\beta} \Bigl( \partial_\alpha \psi_{01e}\partial_\beta \phi_{01e} - \psi_{01e,\alpha}'\phi_{01e,\beta}'\Bigr)\\
      &\leq |B^{\alpha\beta}| \bigl\lvert  \partial_\alpha \psi_{01e} -\psi_{01e,\alpha}' \bigr\rvert |\partial_\beta \phi_{01e}|+|B^{\alpha\beta}||\psi_{01e,\alpha}'|\bigl\lvert \partial_\beta \phi_{01e}-\phi_{01e,\beta}' \bigr\rvert
    \end{split}
  \end{equation*}
  and thus the following inequality follows from Lemma~\ref{lemma:psi:01e:prime}:
  \begin{multline}\label{eq:box:w:rhs:second}
| B(\pa \psi_{01e},\pa \phi_{01e})-B(\psi_{01e}^\prime,\phi_{01e}^{\,\prime})|
\lesssim\frac{\big|\chi\!\big(\tfrac{\langle q\rangle}{r}\big)\big|}{\langle\, t\!+\!r\rangle^4} \biggl[ \sum_{i=0,1}\!\frac{\!|F_{\!i}|\!\!}{\langle q\rangle^{i}}\Bigl(\sum_{j=0,1}\!\frac{\!|G_{\!j}|\!\!+\!|\pa_{\omega}G_{\!j}|\!\!+\!\langle q\rangle|G_{\!j}^{\,\prime}|\!}{\langle q\rangle^{j}}+N\Bigr)\\+\sum_{i=0,1}\!\frac{\!|G_{\!i}|\!\!}{\langle q\rangle^{i}}\Bigl(\sum_{j=0,1}\!\frac{\!|F_{\!j}|\!\!+\!|\pa_{\omega}F_{\!j}|\!\!+\!\langle q\rangle|F_{\!j}^{\,\prime}|\!}{\langle q\rangle^{j}}+M\Bigr) \biggr]
\end{multline}

  Therefore
  \begin{align*}
       \big\|\langle t+r\rangle^s& \big|B(\pa \psi_{01e},\pa \phi_{01e})-B(\psi_{01e}^\prime,\phi_{01e}^{\,\prime})\big|(t,\cdot)\big\|^2_{\mathrm{L}^2(\Sigma_t)}\lesssim\displaybreak[0]\\
                                 &\lesssim \frac{1}{\langle t\rangle^{5-2s}}\int_\mathbb{R} \int_{\mathbb{S}^2} \frac{1}{\langle q\rangle} \sum_{i=0}^1\langle q\rangle^{-2i}|F_i|^2\Bigl(\sum_{j=0}^1\langle q\rangle^{-2j}\sum_{|\beta|+l\leq 1}|(\langle q\rangle\partial_q)^l\partial_\omega^\beta G_j|^2+N^2\Bigr)\ud q \ud S(\omega)\\
        &\qquad+ \frac{1}{\langle t\rangle^{5-2s}}\int_\mathbb{R} \int_{\mathbb{S}^2} \frac{1}{\langle q\rangle} \sum_{i=0}^1\langle q\rangle^{-2i}|G_i|^2\Bigl(\sum_{j=0}^1\langle q\rangle^{-2j}\sum_{|\beta|+l\leq 1}|(\langle q\rangle\partial_q)^l\partial_\omega^\beta F_j|^2+M^2\Bigr)\ud q \ud S(\omega)\\
    &\lesssim \frac{1}{\langle t\rangle^{5-2s}}\sum_{i,j=0}^1 \Bigl(\bigl(\|G_j\|^2_{1,\infty,-j}+N^2\bigr)\|F_i\|^2_{0,-1/2-i}
    +\bigl(\|F_i\|^2_{1,\infty,-i}+M^2\bigr)\|G_j\|^2_{0,-1/2-j}\Bigr)\,.
    \tag*{\qedhere}
\end{align*}
\end{proof}

The bilinear forms $Q(\partial\psi,\partial\phi)$ appearing in \eqref{eq:nullcondition:system} satisfy the null condition. Recall that all classical null forms $Q(\psi',\phi')$ are linear combinations of
\begin{equation}
 Q_{\alpha\beta}(\psi^\prime,\phi^{\,\prime})=\psi^\prime_\alpha \phi^{\,\prime}_\beta-\psi^\prime_\beta\phi^{\,\prime}_\alpha \quad (\alpha\neq \beta)\,, \qquad Q_{00}(\psi^\prime\!,\phi^{\,\prime})=m^{\alpha\beta}\psi^\prime_\alpha \phi^{\,\prime}_\beta\,.
\end{equation}
For $\psi'=\psi_{01e}'$ as defined in \eqref{eqs:psi:01e:prime} we view $\psi'$ as vector with components $\psi'_\alpha=\psi_{01e,\alpha}$ as defined in \eqref{eqs:psi:01e:prime:components}.
For the following algebraic identities note that by \eqref{eq:xi:omegai}
\begin{equation}
  \Omega_{ij}:=x^i\pa_{x^j}-x^j\pa_{x^i}=\omega_i\pa_{\omega^j}-\omega_{\!j\,}\pa_{\omega^i}  \,.
\end{equation}

\begin{lemma} Let $(\psi^\prime_{01e})_{\alpha}=\psi^\prime_{01e,\alpha}$, and $(\phi^\prime_{01e})_{\alpha}=\phi^\prime_{01e,\alpha}$. Then
  \begin{align}
Q_{ij}(\psi^{\prime}_{01e},\phi^{\,\prime}_{01e})
=&\frac{\Omega_{ij}\phi_{01}}{r}
\big(\psi_{01e}^\prime\!
  -\frac{\psi_{0e}\!+2\psi_1\!}{r}\big)
  -\frac{\Omega_{ij}\psi_{01}}{r}
\big(\phi_{01e}^{\,\prime}\!
  -\frac{\phi_{0e}\!+2\phi_1\!}{r}\big)\\
  &+\frac{\pa_{\omega_i}\psi_{01}\,\pa_{\omega_{\!j\,}}\phi_{01}
  -\pa_{\omega_i}\phi_{01}\,\pa_{\omega_{\!j\,}}\psi_{01}}{r^2}\\
Q_{0j}(\psi^{\prime}_{01e},\phi^{\,\prime}_{01e})=& \frac{\omega_j}{r}\big( \psi_{01e}^\prime(\phi_{0e}\!+2\phi_1)
- \phi_{01e}^{\,\prime}(\psi_{0e}\!+2\psi_1)  \big)
+\frac{1}{r}\big( \psi_{01e}^\prime\pa_{\omega^j}\phi_{01}
- \phi_{01e}^{\,\prime}\pa_{\omega^j}\psi_{01}\big)\,,\\
Q_{00}(\psi^\prime_{01e},\phi^{\,\prime}_{01e})=& -\frac{1}{r}\big( \psi_{01e}^\prime(\phi_{0e}\!+2\phi_1)
+ \phi_{01e}^{\,\prime}(\psi_{0e}\!+2\psi_1)  \big)\\
&+\frac{1}{r^2} \big((\psi_{0e}\!+\!2\psi_1)(\phi_{0e}\!+\!2\phi_1)
+\delta^{ij}\pa_{\omega^i}\psi_{01}\,\,\pa_{\omega^j}\phi_{01}\big).
  \end{align}

\end{lemma}
\begin{proof}
  By direct calculation using the expressions in  \eqref{eqs:psi:01e:prime:components}.
\end{proof}

Motivated by these formulas, let us define
\begin{subequations}
 \begin{align}
 Q_{ij}^0(\psi_{0e},\psi^\prime_{0e},\pa_{\omega}\psi_{0e},\phi_{0e},\phi^{\,\prime}_{0e},\pa_\omega\phi_{0e})
   &:=\frac{1}{r}\big( \psi_{0e}^\prime \Omega_{ij} \phi_0-  \phi_{0e}^\prime \Omega_{ij} \psi_0 \big) \label{eq:Q:ij}\\
   Q_{0j}^0(\psi_{0e},\psi^\prime_{0e},\pa_{\omega}\psi_{0e},\phi_{0e},\phi^{\,\prime}_{0e},\pa_\omega\phi_{0e})&:= \frac{1}{r}\big( \psi_{0e}^\prime \big(\omega_{\!j\,}\phi_{0e}+\pa_{\omega^j}\phi_{0}\big)- \phi_{0e}^{\,\prime}(\omega_{\!j\,}\psi_{0e}+\pa_{\omega^j}\psi_{0}) \big)\label{eq:Q:0j}\\
   Q_{00}^0(\psi_{0e},\psi^\prime_{0e},\pa_{\omega}\psi_{0e},\phi_{0e},\phi^{\,\prime}_{0e},\pa_\omega\phi_{0e})&:=- \frac{1}{r} \big( \psi_{0e}^{\,\prime} \phi_{0e}+ \phi_{0e}^{\,\prime} \psi_{0e} \big)+\frac{1}{r^2}\psi_e\phi_e\label{eq:Q:00}
 \end{align}
\end{subequations}
Note that only the leading order approximation is taken into account in $Q^0$, while all lower order terms in $Q$ are estimated as follows:

\begin{lemma}\label{lem:Qlemma}
  We have for all $\mu,\nu=0,1,2,3$,
\begin{multline}
 \big\|\langle t+r\rangle^s \big| Q_{\mu\nu}(\psi_{01e}^\prime,\phi_{01e}^\prime)
 -Q_{\mu\nu}^0(\psi_{0e},\psi^\prime_{0e},\pa_{\omega}\psi_{0e},\phi_{0e},\phi^{\,\prime}_{0e},
\pa_\omega\phi_{0e})\big|(t,\cdot)\big\|_{L^2_x(\Sigma_t)}\\
\lesssim \frac{1}{\langle\, t\rangle^{5/2-s}}\sum_{i,j=0,1}\Big(\|F_i\|_{1,-1/2-i}\big(\|G_j\|_{1,\infty,-j}+N\big)
+\big(\|F_i\|_{1,\infty,-i}+M\big)\|G_j\|_{1,-1/2-j}\Big).
\end{multline}
\end{lemma}

\begin{proof}
  Consider the component $\mu=\nu=0$. We have
  \begin{equation*}
    \begin{split}
      Q_{00}-Q_{00}^0 =& -\frac{1}{r}\psi_1'\bigl(\phi_{0e}+2\phi_1)-\frac{1}{r}\phi_1'\bigl(\psi_{0e}+2\psi_1\bigr)
       -\frac{1}{r}\bigl(\psi_{01e}'2\phi_1+\phi_{01e}'2\psi_1\bigr)\\
       & +\frac{1}{r^2} \big((\psi_{0e}\!+\!2\psi_1)(\phi_{0e}\!+\!2\phi_1)+\delta^{ij}\pa_{\omega^i}\psi_{01}\,\,\pa_{\omega^j}\phi_{01}\big)
       -\frac{1}{r^2}\psi_e\phi_e
    \end{split}
  \end{equation*}
  and in view of \eqref{eq:improvedlinearexpansionpsi01}, \eqref{eq:xi:omegai}, \eqref{eq:psi:01:omega}, and and Lemma~\ref{lemma:cutoff}, \ref{lemma:psi:01e:prime}:
\begin{multline}\label{eq:box:w:rhs:second}
\big|  Q_{00}(\psi_{01e}^\prime,\phi_{01e}^{\,\prime})
-Q_{00}^0(\psi_{0e},\psi^\prime_{0e},\pa_{\omega}\psi_{0e},\phi_{0e},\phi^{\,\prime}_{0e},
\pa_\omega\phi_{0e})\big|\lesssim\\
\lesssim          \frac{\big|\chi\big(\tfrac{\langle q\rangle}{r}\big)\big|}{\langle\, t+r\rangle^4} \sum_{i,j=0}^1\Bigl( {\langle q\rangle^{-i}} \sum_{\alpha+k\leq 1}|(\langle q\rangle\partial_q)^k\partial_\omega^\alpha F_{\!i}(q,\omega)| \Bigr)\Bigl( {\langle q\rangle^{-j}} \sum_{\alpha+k\leq 1}|(\langle q\rangle\partial_q)^k\partial_\omega^\alpha G_{\!j}(q,\omega)| \Bigr)\\
  + \frac{\big|\chi\big(\tfrac{\langle q\rangle}{r}\big)\chi_e\big|}{\langle\, t+r\rangle^4}\Bigl[M\sum_{j=0,1}{\langle q\rangle^{-j}} \sum_{k\leq 1}|(\langle q\rangle\partial_q)^k G_{\!j}(q,\omega)|+N\sum_{i=0,1}{\langle q\rangle^{-i}} \sum_{k\leq 1}|(\langle q\rangle\partial_q)^k F_{\!i}(q,\omega)|\Bigr]
\end{multline}
The statement of the Lemma then follows as in the proof of Lemma~\ref{lem:Blemma}.

Similarly, for the components $\mu=i,\nu=j;i,j=1,2,3$,
\begin{equation*}
  \begin{split}
    Q_{ij}-Q_{ij}^0 = &\frac{\Omega_{ij}\phi_{1}}{r}\big(\psi_{01e}^\prime\!  -\frac{\psi_{0e}\!+2\psi_1\!}{r}\big)  -\frac{\Omega_{ij}\psi_{1}}{r}\big(\phi_{01e}^{\,\prime}\!  -\frac{\phi_{0e}\!+2\phi_1\!}{r}\big)\\
    &+\frac{\Omega_{ij}\phi_{0}}{r}\big(\psi_{1}^\prime\!  -\frac{\psi_{0e}\!+2\psi_1\!}{r}\big)  -\frac{\Omega_{ij}\psi_{0}}{r}\big(\phi_{1}^{\,\prime}\!  -\frac{\phi_{0e}\!+2\phi_1\!}{r}\big)\\
  &+\frac{\pa_{\omega_i}\psi_{01}\,\pa_{\omega_{\!j\,}}\phi_{01}  -\pa_{\omega_i}\phi_{01}\,\pa_{\omega_{\!j\,}}\psi_{01}}{r^2}
  \end{split}
\end{equation*}
and we proceed as above to obtain
\begin{equation*}
  \begin{split}
    |Q_{ij}-Q_{ij}^0|\lesssim &\frac{\chi}{\langle t+r\rangle^4} \sum_{j=0,1} \langle q\rangle^{-j}|\partial_\omega G_j|\Bigl(\sum_{i=0,1}\langle q\rangle^{-i}\sum_{k+|\alpha|\leq 1}(\langle q\rangle\partial_q)^k \partial_\omega^\alpha F_i+M\Bigr)\\
    &+\frac{\chi}{\langle t+r\rangle^4} \sum_{i=0,1} \langle q\rangle^{-i}|\partial_\omega F_j|\Bigl(\sum_{j=0,1}\langle q\rangle^{-j}\sum_{k+|\alpha|\leq 1}(\langle q\rangle\partial_q)^k \partial_\omega^\alpha G_j+N\Bigr)
  \end{split}
\end{equation*}
Finally, in the case $\mu=0,\nu=j=1,2,3$,
\begin{equation*}
  \begin{split}
    Q_{0j}-Q_{0j}^0=& \frac{1}{r}\big( \psi_{1}^\prime(\omega_j\phi_{0e}\!+2\omega_j\phi_1+\partial_\omega\phi_{01})
    - \phi_{1}^{\,\prime}(\omega_j\psi_{0e}\!+2\omega_j\psi_1+\partial_\omega\psi_{01})  \big)\\
    &+\frac{1}{r}\big( \psi_{01e}^\prime(2\omega_j\phi_1+\partial_{\omega^j}\phi_1)- \phi_{01e}^{\,\prime}(2\omega_j\psi_1+\partial_{\omega^j}\psi_1)  \big)
  \end{split}
\end{equation*}
and thus
\begin{equation*}
  \begin{split}
  |Q_{0j}-Q_{0j}^0| \lesssim \frac{\chi}{\langle t+r\rangle^4}\Bigl[&\sum_{l+|\alpha|\leq 1}\frac{|(\langle q\rangle \partial_q)^l\partial_\omega^\alpha F_1|}{\langle q\rangle}\bigl(\sum_{\substack{j=0,1\\|\alpha|+k\leq 1}}\langle q\rangle^{-j}\partial_\omega^\alpha(\langle q\rangle\partial_q)^k G_j+N\bigr)\\&+\sum_{l+|\alpha|\leq 1}\frac{|(\langle q\rangle\partial_q)^l\partial_\omega^\alpha G_1|}{\langle q\rangle}\bigl(\sum_{\substack{i=0,1\\|\alpha|+k\leq 1}}\langle q\rangle^{-i}\partial_\omega^\alpha(\langle q\rangle\partial_q)^k F_i+M\bigr)\Bigr]
\end{split}
\end{equation*}
and the statement follows.
\end{proof}

\subsection{The second order asymptotics for a system satisfying the classical null condition}

In Section~\ref{sec:asymptotics:second:approx} we have imposed an equation on the radiation field of the approximate solution to obtain a good second order approximation; c.f.~\eqref{eq:F1homogeneouscondition} and Lemma~\ref{lem:boxapprox}.

Similarly here we want to find $F_1$ and $G_1$ such that modulo terms of order $\mathcal{O}(r^{-4})$
\begin{equation}\label{eq:nullcondrevisitedapprox}
\Box \psi_{01e} \sim Q(\pa \psi_{01e},\pa \phi_{01e}),\qquad\Box\phi_{01e} \sim \widetilde{Q}(\pa \psi_{01e},\pa \phi_{01e}).
\end{equation}
Let us set
\begin{subequations}
\begin{align}
2F_1^{\,\prime}(q,\omega)\!&= \triangle_\omega F_0(q,\omega)
-Q^0(F_0,F_0^\prime,\pa_\omega F_0,G_0,G_0^\prime,\pa_\omega G_0),\quad F_1(0,\omega)=0\label{eq:asymptoticnullcond1}\\
2G_1^{\,\prime}(q,\omega)\!&= \triangle_\omega G_0(q,\omega)-\widetilde{Q}^0(F_0,F_0^\prime,\pa_\omega F_0,G_0,G_0^\prime,\pa_\omega G_0)\,,\quad G_1(0,\omega)=0
\label{eq:asymptoticnullcond2}
\end{align}
\end{subequations}
where $Q^0$ as a bilinear form in the radiation fields $F_0$, and $G_0$ is defined in complete analogy to \eqref{eq:Q:ij}, \eqref{eq:Q:0j}, \eqref{eq:Q:00}:
\begin{subequations}
\begin{align}
  Q^0_{00}(F_0,G_0):=& \bigl(F_0'\chi+M\chi_e'\bigr)\Omega_{ij}G_0-\bigl(G_0'\chi+N\chi_e'\bigr)\Omega_{ij}F_0\label{eq:Q:00:FG}\\
  Q^0_{0j}(F_0,G_0):=& \bigl( F_0'\chi+M\chi_e'\bigr)\bigl(\omega_jG_0\chi+\omega_j N\chi_e+\partial_{\omega^j}G_0\chi\bigr)\\&-\bigl(G_0'\chi+N\chi_e'\bigr)\bigl(\omega_jF_0\chi+\omega_j M\chi_e+\partial_{\omega^j}F_0\chi\bigr)\label{eq:Q:0j:FG} \\
  Q^0_{ij}(F_0,G_0):=& -\bigl(F_0'\chi+M\chi_e'\bigr)\bigl(G_0\chi+N\chi_e\bigr)-\bigl(G_0'\chi+N\chi_e'\bigr)\bigl(F_0+M\chi_e\bigr)\label{eq:Q:ij:FG}
\end{align}
\end{subequations}

The following proposition shows that $\psi_{01e}$, $\phi_{01e}$ with $F_1$, $G_1$, thus defined are indeed a good approximate solution to the system, in the weighted $L^2$ norm required  by the conformal energy estimate:

\begin{prop}\label{prop:approximatesolution} Suppose that $F_1$ and $G_1$ satisfy \eqref{eq:asymptoticnullcond1}, \eqref{eq:asymptoticnullcond2}.
Then
\begin{multline}
 \big\|\langle t+r\rangle^s \big| \Box\psi_{01e}-Q(\pa \psi_{01e},\pa \phi_{01e})\big|(t,\cdot)\big\|_{L^2_x}
 +\big\|\langle t+r\rangle^s \big| \Box\phi_{01e}-\widetilde{Q}(\pa \psi_{01e},\pa \phi_{01e})\big|(t,\cdot)\big\|_{L^2_x}\lesssim\\
\lesssim \frac{1}{\langle\, t\rangle^{3/2+\gamma-s}}\sum_{i,j=0,1}\Big(MN+\|F_i\|_{2,\gamma-1/2-i}\big(1+N+\|G_j\|_{1,\infty,\gamma-j}\big)\\
+\|G_j\|_{2,\gamma-1/2-j}\big(1+M+\|F_i\|_{1,\infty,\gamma-i}\big)\Big).
\end{multline}

\end{prop}

\begin{proof}
  Since with our choice of $F_1$, and $G_1$,
  \begin{align*}
    \triangle_\omega \psi_0/r^2-2\psi_1^\prime/r &= \frac{1}{r^3}Q^0(F_0,G_0)\chi\big(\tfrac{\langle t-r\rangle}{r}\big)\\
    \triangle_\omega \phi_0/r^2-2\phi_1^\prime/r &= \frac{1}{r^3}\widetilde{Q}^0(F_0,G_0)\chi\big(\tfrac{\langle t-r\rangle}{r}\big)
  \end{align*}
  we get by Lemma~\ref{lemma:approx:error},
  \begin{equation*}
    \begin{split}
      \lvert \Box \psi_{01e}-Q^0(\psi_{0e},\phi_{0e})  \rvert \lesssim& \bigl\lvert \frac{1}{r^3}Q^0(F_0,G_0)\chi-Q^0(\psi_{0e},\phi_{0e})\rvert\\
      &+ \frac{\chi}{\langle\, t+r\rangle^4} \sum_{|\alpha|+k\leq 2,\, i=0,1} \langle\, q\rangle^{1-i}\big|(\langle\, q\rangle\pa_q)^k\pa_\omega^\alpha F_i(q,\omega)\big|\,.
    \end{split}
  \end{equation*}
  where
  \begin{equation*}
    \bigl\lvert \frac{1}{r^3}Q^0(F_0,G_0)\chi-Q^0(\psi_{0e},\phi_{0e})\rvert \lesssim \frac{1}{\langle t+r\rangle^4}MN\chi_e
  \end{equation*}
  and hence
  \begin{multline*}
    \|\langle t+r\rangle^s \lvert \Box \psi_{01e}-Q^0(\psi_{0e},\phi_{0e})  \rvert (t,\cdot)\|_{L^2_x}^2\lesssim\\
    \lesssim \frac{1}{\langle t+r\rangle^{3-2s+2\gamma}}\int_{\mathbb{R}}\int_{\mathbb{S}^2}\Bigl[\frac{M^2N^2}{\langle q\rangle^{3-2\gamma}}+\sum_{|\alpha|+k\leq 2,\, i=0,1} \langle\, q\rangle^{-1-2i+2\gamma}\big|(\langle\, q\rangle\pa_q)^k\pa_\omega^\alpha F_i(q,\omega)\bigr|^2\Bigr]\ud q\ud S(\omega)\\
    \lesssim \langle t+r\rangle^{-3+2s-2\gamma}\Bigl(M^2N^2+\sum_{i=0,1}\|F_i\|_{2,\gamma-1/2-i}^2\Bigr)
  \end{multline*}
  Furthermore by  Lemma \ref{lem:Blemma} and  Lemma \ref{lem:Qlemma} we conclude that
  \begin{equation*}
    \begin{split}
     \|\langle t+r\rangle^s \big| \Box\psi_{01e}-Q(\pa \psi_{01e},\pa \phi_{01e})\big|\|_{L^2_x}\leq& \|\langle t+r\rangle^s \big| \Box\psi_{01e}-Q^0(\psi_{0e},\phi_{0e})\big|\|_{L^2_x}\\&+\|\langle t+r\rangle^s \big|Q^0(\psi_{0e},\phi_{0e})-Q(\psi'_{01e},\phi'_{01e})\big|\|_{L^2_x}\\&+\|\langle t+r\rangle^s \big|Q(\psi'_{01e},\phi'_{01e})-Q(\pa \psi_{01e},\pa \phi_{01e})\big|\|_{L^2_x}\\
     \lesssim \frac{1}{\langle\, t\rangle^{5/2-s}}\sum_{i,j=0,1}\Big(\|F_i\|_{1,-1/2-i}\big(\|G_j\|_{1,\infty,-j}&+N\big)+\big(\|F_i\|_{1,\infty,-i}+M\big)\|G_j\|_{1,-1/2-j}\Big)\\
     &+\frac{1}{\langle t\rangle^{3/2-s+\gamma}}\Bigl(MN+\sum_{i=0,1}\|F_i\|_{2,\gamma-1/2-i}\Bigr)
    \end{split}
  \end{equation*}
  which implies the statement of the Proposition because $1/2<\gamma<1$.
\end{proof}

Finally we control the higher order radiation fields in terms of $F_0$, $G_0$:
\begin{lemma}\label{lem:F1G1byF0G0} Suppose that $F_1$ and $G_1$ satisfy \eqref{eq:asymptoticnullcond1}, \eqref{eq:asymptoticnullcond2}. Then
\begin{multline}
\|F_1\|_{N-2,\gamma-3/2}+\|G_1\|_{N-2,\gamma-3/2}
\lesssim \|F_0\|_{N,\gamma-1/2}+\|G_0\|_{N,\gamma-1/2}\\
+\bigl(\|F_0\|_{N+1,\gamma-1/2}+M\bigr)\bigl(\|G_0\|_{N+1,\gamma-1/2}+N\bigr)
\end{multline}
\begin{multline}
\|F_1\|_{N-2,\infty,\gamma-1}+\|G_1\|_{N-2,\infty,\gamma-1}
\lesssim \|F_0\|_{N,\infty,\gamma}+\|G_0\|_{N,\infty,\gamma}\\
+\bigl(\|F_0\|_{N-1,\infty,\gamma}+M\bigr)\bigl(\|G_0\|_{N-1,\infty,\gamma}+N\bigr)
\end{multline}
\end{lemma}
\begin{proof}
  As in the proof of Lemma~\ref{lem:F1byF0} we have
  \begin{equation*}
    \|F_1\|_{0,\gamma-3/2}\lesssim \|F_1'\|_{0,\gamma-1/2}\lesssim \|F_0\|_{2,\gamma-1/2}+\|Q^0(F_0,G_0)\|_{0,\gamma-1/2}\,.
  \end{equation*}
In view of \eqref{eq:Q:00:FG}, \eqref{eq:Q:0j:FG}, and \eqref{eq:Q:ij:FG} we have
\begin{align*}
  \|Q^0_{00}(F_0,G_0)\|_{0,\gamma-1/2}\lesssim& \|G_0\|_{1,\infty,0}\bigl(\|F_0\|_{1,\gamma-3/2}+M\bigr)+\|F_0\|_{1,\infty,0}\bigl(\|G_0\|_{1,\gamma-3/2}+N\bigr)\\
  \|Q^0_{0j}(F_0,G_0)\|_{0,\gamma-1/2}\lesssim& \bigl(\|G_0\|_{1,\infty,0}+N\bigr)\bigl(\|F_0\|_{1,\gamma-3/2}+M\bigr)\\&\quad+\bigl(\|F_0\|_{1,\infty,0}+M\bigr)\bigl(\|G_0\|_{1,\gamma-3/2}+N\bigr)\\
    \|Q^0_{ij}(F_0,G_0)\|_{0,\gamma-1/2}\lesssim& \bigl(\|G_0\|_{0,\infty,0}+N\bigr)\bigl(\|F_0\|_{1,\gamma-3/2}+M\bigr)\\&\quad+\bigl(\|F_0\|_{0,\infty,0}+M\bigr)\bigl(\|G_0\|_{1,\gamma-3/2}+N\bigr)
\end{align*}
and using Lemma~\ref{lemma:F:Sobolev},
\begin{equation*}
  \|F_0\|_{0,\infty,0}\lesssim \|F_0\|_{0,\infty,\gamma} \lesssim \|F_0\|_{3,\gamma-1/2}\qquad   \|G_0\|_{0,\infty,0} \lesssim \|G_0\|_{3,\gamma-1/2}\,.
\end{equation*}
Moreover as in the proof of Lemma~\ref{lem:F1byF0},
\begin{equation*}
  \|F_1\|_{0,\infty,\gamma-1}\lesssim \|F'_1\|_{0,\infty,\gamma} \lesssim \|F_0\|_{2,\infty,\gamma}+\|Q^0(F_0,G_0)\|_{0,\infty,\gamma}
\end{equation*}
and again in view of  \eqref{eq:Q:00:FG}, \eqref{eq:Q:0j:FG}, and \eqref{eq:Q:ij:FG} we have
\begin{align*}
  \|Q_{00}^0\|_{0,\infty,\gamma} \lesssim& \|G_0\|_{1,\infty,0}\bigl(\|F_0\|_{1,\infty,\gamma-1}+M\bigr)+\|F_0\|_{1,\infty,0}\bigl(\|G_0\|_{1,\infty,\gamma-1}+N\bigr)\\
  \|Q_{0j}^0\|_{0,\infty,\gamma} \lesssim& \bigl(\|G_0\|_{1,\infty,0}+N\bigr)\bigl(\|F_0\|_{1,\infty,\gamma-1}+M\bigr)\\&\quad +\bigl(\|F_0\|_{1,\infty,0}+M\bigr)\bigl(\|G_0\|_{1,\infty,\gamma-1}+N\bigr)\\
    \|Q_{ij}^0\|_{0,\infty,\gamma} \lesssim& \bigl(\|G_0\|_{0,\infty,0}+N\bigr)\bigl(\|F_0\|_{1,\infty,\gamma-1}+M\bigr)\\&\quad +\bigl(\|F_0\|_{0,\infty,0}+M\bigr)\bigl(\|G_0\|_{1,\infty,\gamma-1}+N\bigr)\,.
\end{align*}
This proves the Lemma in the case $N=2$. Higher order estimates follow by commuting the equations  \eqref{eq:asymptoticnullcond1}, \eqref{eq:asymptoticnullcond2}.
\end{proof}

\subsection{Estimates for norms of bilinear forms}
\label{sec:norm:bilinear}

\begin{lemma}\label{lem:thebilinearestimate} Let $B$ be a bilinear form. Then for $0\leq s\leq 3/2$,
\beq
 \big\|\langle t+r\rangle^s B(\pa v,\pa w)(t,\cdot)\big\|_{L^2_x}
\lesssim \langle t\rangle^{s-1} {\sum}_{|I|\leq 2}\|\pa Z^I v(t,\cdot)\big\|_{L^2_x}\|\pa  w(t,\cdot)\|_{L^2_x}.
\eq
\end{lemma}
\begin{proof} By the Klainerman Sobolev inequality we have 
\beq
\langle t+r\rangle^s |\pa v|\lesssim  \frac{\langle t+r\rangle^s}{\langle t+r\rangle \langle t-r\rangle^{1/2}}{\sum}_{|I|\leq 2}\|\pa Z^I v(t,\cdot)\big\|_{L^2_x}
\lesssim \langle t\rangle^{s-1} {\sum}_{|I|\leq 2}\|\pa Z^I v(t,\cdot)\big\|_{L^2_x}.\tag*{\qedhere}
\eq
\end{proof}
There is an improvement for the null forms that we however do not need to use for the asymptotic expansion at this order.
\begin{lemma} Let $Q(\pa v,\pa w)$ be a null form. Then for any $s\geq 1$,
\begin{multline}
 \big\|\langle t+r\rangle^s Q(\pa v,\pa w)(t,\cdot)\big\|_{L^2_x}
\lesssim \langle t\rangle^{s-2} {\sum}_{|I|\leq 3}\|\langle t-r\rangle^{s-1} Z^I v(t,\cdot)\big\|_{L^2_x}\|\pa  w(t,\cdot)\|_{L^2_x}\\
+  \langle t\rangle^{s-2} {\sum}_{|I|\leq 2}\|\pa Z^I v(t,\cdot)\big\|_{L^2_x} {\sum}_{|J|\leq 1}\|\langle t-r\rangle^{s-1} Z^J  w(t,\cdot)\|_{L^2_x}.
\end{multline}
\end{lemma}
\begin{proof} We have
$|Q(\pa v,\pa w)|\lesssim |\,\opa v| |\pa w|+ |\,\pa v| |\opa w|\lesssim
\langle t+r\rangle^{-1} \sum_{|J|= 1}\big(|Z^J v| |\pa w|+ |\pa v| |Z^J w|\big)$.
The statement then follows from Lemma~\ref{lemma:KS:weights}.
\end{proof}

\subsection{Energy estimates for a system satisfying the classical null condition}

Let us return to the system \eqref{eq:nullcondrevisited1}-\eqref{eq:nullcondrevisited2},
which we can recast as a system for $v=\psi-\psi_{01e}$ and $u=\phi-\phi_{01e}$:
\begin{subequations}
\begin{align}
\Box v &=  Q(\pa v ,\pa u)+Q(\pa v,\pa \phi_{01e})
+ Q(\pa \psi_{01e} ,\pa u)+Q(\pa \psi_{01e} ,\pa \phi_{01e})-\Box\psi_{01e},\label{eq:nullcondsystemexpansion1}\\
 \Box u &=  \widetilde{Q}(\pa v ,\pa u)+\widetilde{Q}(\pa v,\pa \phi_{01e})
+ \widetilde{Q}(\pa \psi_{01e} ,\pa u)+\widetilde{Q}(\pa \psi_{01e} ,\pa \phi_{01e})-\Box\phi_{01e}.\label{eq:nullcondsystemexpansion1:u}
\end{align}
\end{subequations}

Indeed for any bilinear form $B$,
we see that
\begin{equation}
  \begin{split}
B(\pa \psi,\pa \phi)=&B(\pa \psi_{01e},\pa \phi_{01e})\\&+B(\pa \psi-\pa \psi_{01e} ,\pa \phi+\pa \phi_{01e})/2
+ B(\pa \psi+\pa \psi_{01e} ,\pa \phi-\pa \phi_{01e})/2\\
=&B(\pa \psi_{01e},\pa \phi_{01e})+B(\pa \psi-\pa \psi_{01e} ,\pa \phi-\pa \phi_{01e})\\
&+ B(\pa \psi_{01e} ,\pa \phi-\pa \phi_{01e})+B(\pa\psi-\pa \psi_{01e} ,\pa \phi_{01e}).
\end{split}
\end{equation}

Consider the right hand sides of \eqref{eq:nullcondsystemexpansion1}, \eqref{eq:nullcondsystemexpansion1:u}.
We will use Proposition~\ref{prop:approximatesolution} to estimate the last two terms, where it is essential that $(\psi_{01e},\phi_{01e})$ are good approximations, and that $Q$ and $\widetilde{Q}$ have null structure.
However, to estimate the remaining terms we do not need the null structure of $Q$ and will appeal to Lemma~\ref{lem:thebilinearestimate} which holds for any bilinear form.

\begin{lemma}\label{lemma:Q:v:psi}
  We have
  \begin{multline}
    \big\|\langle t+r\rangle^s Q(\pa v,\pa \phi_{01e})(t,\cdot)\big\|+\big\|\langle t+r\rangle^s Q(\pa \psi_{01e},\pa u)(t,\cdot)\big\|\lesssim\\
    \lesssim \langle t\rangle^{s-1}\biggl( \Bigl(\sum_{i=0,1} \|G_i\|_{3,-1-i} + N\Bigr)\|\pa  v(t,\cdot)\|+ \Bigl(\sum_{i=0,1} \|F_i\|_{3,-1-i} + M\Bigr)\|\pa  u(t,\cdot)\|\biggr)
  \end{multline}

\end{lemma}
\begin{proof}
  First note that from Lemma~\ref{lemma:psi:01e:prime}, in particular \eqref{eq:psi:01e:higher},  it follows that
  \begin{equation}\label{eq:aprroxsolbound}
    \|\pa Z^I \psi_{01e}(t,\cdot)\| \lesssim \sum_{i=0,1} \|F_i\|_{1+|I|,-1-i} + M
  \end{equation}
  and hence by Lemma~\ref{lem:thebilinearestimate},
  \begin{equation}
    \begin{split}
     \big\|\langle t+r\rangle^s Q(\pa \psi_{01e},\pa u)(t,\cdot)\big\|
     \lesssim& \langle t\rangle^{s-1} {\sum}_{|I|\leq 2}\|\pa Z^I \psi_{01e}(t,\cdot)\big\|\|\pa  u(t,\cdot)\|\\
     \lesssim& \langle t\rangle^{s-1} \Bigl(\sum_{i=0,1} \|F_i\|_{3,-1-i} + M\Bigr)\|\pa  u(t,\cdot)\|
\end{split}
\end{equation}
Similarly for the term $Q(\pa v,\pa \phi_{01e})$.
\end{proof}

\subsection{Proof of Proposition~\ref{prop:nullcond:revisited}}

The above already provides all the necessary estimates for the stated scattering result for systems satisfying the null condition, and we now  put these elements together to give a formal proof of Prop.~\ref{prop:nullcond:revisited}.

Instead of taking  $v_T$, and $u_T$ directly to be solutions to \eqref{eq:nullcondsystemexpansion1}, \eqref{eq:nullcondsystemexpansion1:u}, with trivial data $v_T=u_T=\partial_tv_T=\partial_t u_T=0$  at $t=T$, we introduce an additional cutoff in time to ensure that also $Z^Iu_T$, and $Z^Iv_T$ have trivial initial data at $t=T$:
\begin{subequations}
\begin{align}
\Box v_T &= \chi(\tfrac{t}{T})\bigl( Q(\pa v_T ,\pa u_T)+Q(\pa v_T,\pa \phi_{01e})
+ Q(\pa \psi_{01e} ,\pa u_T)+Q(\pa \psi_{01e} ,\pa \phi_{01e})-\Box\psi_{01e}\bigr),\label{eq:Box:v:nullcondsystem:chi}\\
 \Box u_T &=  \chi(\tfrac{t}{T})\bigl(\widetilde{Q}(\pa v_T ,\pa u_T)+\widetilde{Q}(\pa v_T,\pa \phi_{01e})
+ \widetilde{Q}(\pa \psi_{01e} ,\pa u_T)+\widetilde{Q}(\pa \psi_{01e} ,\pa \phi_{01e})-\Box\phi_{01e}\Bigr).\label{eq:Box:u:nullcondsystem:chi}
\end{align}
\end{subequations}

\subsubsection{Energy estimates}

Similarly to the proof of Prop.~\ref{prop:wave:null} we will assume the following \emph{a priori} estimate, which we first verify for $t=T$, and then establish by a continuity argument in $t\leq T$:

For some $\delta>0$, and $C_0>0$,
\begin{equation}\label{eq:vwapriori}
\sum_{|I|\leq 5}\|\pa Z^I v_T(t,\cdot)\|+\|\pa Z^I u_T(t,\cdot)\|\leq \frac{C_0\varepsilon}{\langle t\rangle^{\delta}}\,.
\end{equation}

At $t=T$ the left hand side vanishes for $|I|=0$, and also for $|I|\geq 1$ by Lemma~\ref{lemma:Z:v}.

First consider the case $|I|=0$.
We use the energy identity,
\begin{equation}
E(t):=\|\pa v_T(t,\cdot)\|+\|\pa u_T(t,\cdot)\|\leq \int_t^T \|\Box v_T\|+\|\Box u_T\| \ud t'
\end{equation}
and observe that the first term in \eqref{eq:nullcondsystemexpansion1}, \eqref{eq:nullcondsystemexpansion1:u}, can be estimated using the \emph{a priori} bound and Lemma~\ref{lem:thebilinearestimate},
\begin{equation}
 \big\|\langle t+r\rangle^s Q(\pa v_T,\pa u_T)(t,\cdot)\big\|
\lesssim  C_0 \varepsilon\, \langle t\rangle^{s-1-\delta} \|\pa  u_T(t,\cdot)\|
\end{equation}
For the second and third term we apply Lemma~\ref{lemma:Q:v:psi}, and for the last two terms Prop.~\ref{prop:approximatesolution};
in the case $s=0$ the above then implies
\begin{equation}\label{eq:nullcond:gronwall}
\|\pa v_T(t,\cdot)\|+\|\pa u_T(t,\cdot)\|\lesssim \int_t^T  \frac{C_0 \varepsilon \langle t\rangle^{-\delta}+\tilde{D}}{\langle t\rangle}\Bigl( \|\pa  v_T(t,\cdot)\|+ \|\pa  u_T(t,\cdot)\|\Bigr)+\frac{\tilde{D}+\tilde{D}^2}{\langle t\rangle^{3/2+\gamma}}\ud t
\end{equation}
where the scattering data $\tilde{D}:=\tilde{D}_0$, is given by
\begin{equation}
 \quad \tilde{D}_k:=\sum_{i=0,1} \Big( \|F_i\|_{3+k,\gamma-1/2-i} +\|F_i\|_{1+k,\infty,\gamma-i}+ M+ \|G_i\|_{3+k,\gamma-1/2-i} +\|G_j\|_{1+k,\infty,\gamma-j}+ N \Big)\,.
\end{equation}
Note that by Lemma~\ref{lem:F1G1byF0G0}, and Lemma~\ref{lemma:F:Sobolev}, $  \tilde{D}\lesssim D_0(1+D_0) $.
In view of the smallness assumption on the data $D_0<\varepsilon$, we then arrive at
\begin{equation}\label{eq:nullcond:energyineq}
E(t)\leq C\int_t^T\frac{(C_0+1)\varepsilon}{ \tau}E(\tau)\ud \tau+\frac{C \varepsilon}{t^{1/2+\gamma}}\,,
\end{equation}
which yields by Gr\"onwall ---  namely by setting $F(t)=\int_t^T(C_0+1)\epsilon\tau^{-1}E(\tau)\ud \tau$, and differentiating $t^\alpha F(t)$, where $\alpha=C(C_0+1)\varepsilon$ --- the estimate
\begin{equation}\label{eq:nullcond:system:L2:v}
  E(t)\leq \Bigl(\tfrac{(C_0+1)C\varepsilon}{1/2+\gamma-\alpha}+1\Bigr)\frac{C\varepsilon}{\langle t\rangle^{1/2+\gamma}}\,,
\end{equation}
which improves \eqref{eq:vwapriori} with  $C_0$ chosen sufficiently large, and $\varepsilon>0$ sufficiently small.

In fact, as in the proof of Prop.~\ref{prop:scattering:homogeneous}, we can apply the conformal energy estimate of Prop.~\ref{prop:frac:morawetz} to obtain:
\begin{equation}
  \|v_T(t,\cdot)\|_{1,s-1}  +\|u_T(t,\cdot)\|_{1,s-1}\lesssim \int_t^T\|\langle t+r\rangle^s\Box v_T(t,\cdot)\|+\|\langle t+r\rangle^s\Box u_T(t,\cdot)\|\ud t\,,
\end{equation}
and as above, but now with $0<s<\gamma+1/2$, we use Lemma~\ref{lem:thebilinearestimate}, Lemma~\ref{lemma:Q:v:psi}, and Prop.~\ref{prop:approximatesolution} to infer that
\begin{multline}\label{eq:nullcond:gronwall:s}
  \int_t^T\|\langle t+r\rangle^s\Box v_T(t',\cdot)\|+\|\langle t+r\rangle^s\Box u_T(t',\cdot)\|\ud t' \lesssim \\
  \lesssim \int_t^T \bigl(C_0\varepsilon\langle t'\rangle^{-\delta}+\tilde{D}\bigr) \langle t'\rangle^{s-1} E(t') + \langle t'\rangle^{s-3/2-\gamma} (\tilde{D}+\tilde{D}^2\bigr)\ud t'
  \lesssim \frac{\varepsilon}{\langle t\rangle^{1/2+\gamma-s}}
\end{multline}
where in the last step we have used \eqref{eq:nullcond:system:L2:v}.

Let us now turn to the higher orders, namely the case $1\leq |I| \leq 5$.
 $Z^Iv_T$ satisfies
 \begin{align}\label{eq:Z:v:T:system}
   &\\
  \Box Z^Iv_T =& \sum_{|J|+|K|\leq |I|}c^{IJ}_K Z^K\chi(\tfrac{t}{T})F_J+\sum_{|J|+|K|\leq |I|}c^{I}_{KJ}Z^K(\chi(\tfrac{t}{T}))Z^J\Bigl(Q(\partial\psi_{01e},\partial\phi_{01e})-\Box\psi_{01e}\Bigr)\\
   \Box Z^Iu_T =& \sum_{|J|+|K|\leq |I|}c^{IJ}_K Z^K\chi(\tfrac{t}{T})\widetilde{F}_J+\sum_{|J|+|K|\leq |I|}c^{I}_{KJ}Z^K(\chi(\tfrac{t}{T}))Z^J\Bigl(\widetilde{Q}(\partial\psi_{01e},\partial\phi_{01e})-\Box\phi_{01e}\Bigr)
\end{align}
where, c.f.~Lemma~\ref{lemma:F:alpha},
\begin{equation}
  F_I=\sum_{|J|+|K|\leq |I|} \Bigl[ Q_{JK}^I(\partial Z^J v_T, \partial Z^K u_T) + Q_{JK}^I(\partial Z^J v_T, \partial Z^K \phi_{01e})
  +Q_{JK}^I(\partial Z^J \psi_{01e}, \partial Z^K u_T) \Bigr]
\end{equation}
and analogously for $\widetilde{F}_I$.;  ($Q_{JK}^I$ are null forms, but for the following estimates their null structure is not necessary).
From Lemma~\ref{lem:thebilinearestimate}, Lemma~\ref{lemma:Q:v:psi} (see also \eqref{eq:aprroxsolbound}), and the a priori bound \eqref{eq:vwapriori} we have
\begin{equation}
\| \langle t+r\rangle^s F_I \|_{L^2_x}+\| \langle t+r\rangle^s \widetilde{F}_I \|_{L^2_x} \lesssim  \langle t\rangle^{s-1}\Bigl(C_0 \varepsilon \langle t\rangle^{-\delta}+\tilde{D}_{3+|I|}\Bigr)\sum_{|J|\leq |I|} \Bigl(\|\partial Z^J v_T\|+\|\partial Z^J u_T\|\Bigr)
\end{equation}
and by Prop.~\ref{prop:approximatesolution} :
\begin{equation}
 \sum_{|J|\leq |I|}\big\|\langle t+r\rangle^s Z^J \big( \Box\psi_{01e}-Q(\pa \psi_{01e},\pa \phi_{01e})\big)(t,\cdot)\big\|_{2}
\lesssim \frac{D_{3+|I|}(1+D_{3+|I|})}{\langle\, t\rangle^{3/2+\gamma-s}}\,.
\end{equation}
Therefore we can apply the conformal energy estimate of Prop.~\ref{prop:frac:morawetz} to \eqref{eq:Z:v:T:system},
which in the case $s=0$ gives the analogue of \eqref{eq:nullcond:gronwall},
hence \eqref{eq:nullcond:energyineq} holds for $E_k(t)$, $k\leq 5$, where
\begin{equation}\label{eq:E:k:system:null}
  E_k(t)= \sum_{|I|\leq k} \|\pa Z^I v_T(t,\cdot)\|_2+\sum_{|I|\leq k} \| \pa Z^I u_T(t,\cdot) \|_2\,.
\end{equation}
This proves \eqref{eq:vwapriori} as shown in  \eqref{eq:nullcond:energyineq} and \eqref{eq:nullcond:system:L2:v}, and hence \eqref{prop:nullcond:system:eq:energybound} follows.
Moreover \eqref{prop:nullcond:system:eq:confenergybound} follows as in \eqref{eq:nullcond:gronwall:s} with the above bounds.

\subsubsection{Existence and uniqueness of the limit}
For the existence of the limit $(v=\lim_{T\to\infty}v_T,u=\lim_{T\to\infty}u_T)$ we proceed as in the proof of Prop.~\ref{prop:wave:null}:
Consider the differences $(w_1=v_{T_1}-v_{T_2},w_2=u_{T_1}-u_{T_2})$ for $T_2>T_1$, which satisfy a system of the form
\begin{align}
  \Box w_1 =& \chi(t/T_1)\bigl(Q(\pa w_1, \pa \phi_{01e}+ \pa u_1)+Q(\pa v_2+\pa \psi_{01e}, \pa w_2)\bigr)-(\chi(t/T_2)-\chi(t/T_1))\Box v_2\\
  \Box w_2 =& \chi(t/T_1)\bigl(\widetilde{Q}(\pa w_1, \pa \phi_{01e}+ \pa u_1)+\widetilde{Q}(\pa v_2+\pa \psi_{01e}, \pa w_2)\bigr)-(\chi(t/T_2)-\chi(t/T_1))\Box u_2\\
\end{align}
Since $(w_1,w_2)=(-v_2,-u_2)$ at $t=T_1$, and the energies \eqref{eq:E:k:system:null} for $(v_2,u_2)$ are known to be bounded by  $\langle T_1\rangle^{-1/2-\gamma}$ at $t=T_1<T_2$, the energy inequality gives
    \begin{equation}
E(t) \lesssim \varepsilon \langle T_1\rangle^{-1/2-\gamma} +\int_t^{T_1} \frac{\varepsilon}{ \tau } E(\tau) \ud\tau \,,\qquad       E(t)= \sum_{\substack{i=1,2\\|I|\leq 5}}\|\partial  Z^Iw_i (t,\cdot)\|_2
    \end{equation}
    which by Gr\"onwall yields $E(T_0)\lesssim \varepsilon \langle T_1 \rangle^{-1/2-\gamma} (T_1/T_0)^\varepsilon \to 0$ as $T_1\to\infty$, if $\varepsilon<1/2+\gamma$.
    
    Therefore the limits $\psi=\lim \psi_T$, and $\phi=\lim \phi_T$ exist, and are unique because if $F_0=G_0=0$, $M=N=0$, then $\|\psi(t,\cdot)\|+\|\phi(t,\cdot)\|\leq \|(\psi-\psi_T)(t,\cdot)\|+\|(\phi-\phi_T)(t,\cdot)\|\to 0$ as $T\to\infty$.

\section{A simple system satisfying the weak null condition}
\label{sec:weak:null:simple}

Let us first consider the following system of wave equations with weak null structure:

\begin{subequations}\label{eq:simple:weak:system}
\begin{align}
  \Box \psi &= Q(\pa \psi,\pa\psi)\label{eq:wave:weak:hom}\\
  \Box \varphi &= \bigl( \partial_t\psi \bigr)^2.
  \label{eq:wave:weak:inhom}
\end{align}
\end{subequations}

We will first construct a scattering solution to the system \eqref{eq:simple:weak:system} in the following manner:
As in Section~\ref{sec:classical:null:revisited} we have a solution to \eqref{eq:wave:weak:hom} with a given radiation field $F_0$.
Inserting the corresponding approximate solution in \eqref{eq:wave:weak:inhom}, we will define an auxillary solution $\Psi$ to the inhomogeneous problem in Section~\ref{sec:weak:forward},
and show that $\varphi-\Psi$ has a radiation field $G_0$.

\subsection{Auxiliary forward solution}
\label{sec:weak:forward}

We prescribe an asymptotic expansion to second order for \eqref{eq:wave:weak:hom} as in Section~\ref{sec:classical:null:revisited},
\beq
\psi\sim\psi_{01e}
\eq
constructed from radiation fields $F_0$, $F_1$, satisfying \eqref{eq:decayassumption}.

Now we define $\Psi_{01e}$ to be the forward solution ($t>t_0)$ of
\begin{align}\label{eq:Psidef}
  \Box\Psi_{01e}&=(\psi_{01e}^{\,\prime})^2\\
  \Psi_{01e}\rvert_{t=t_0}&=\partial_t\Psi_{01e}\rvert_{t=t_0}=0
\end{align}
where $t_0<0$ is simply chosen so that $t=t_0$ does not intersect the support of the approximate solution $\psi_{01e}'$.
Here $\psi_{01e,\,0}^{\,\prime}=-\psi_{01e}^\prime$ as in the previous section, see in particular \eqref{eq:psi:i:prime} and \eqref{eqs:psi:01e:prime:components}.

We will see below that
\beq
\Box(\varphi-\Psi_{01e})=(\pa_t \psi)^2-(\psi_{01e,\,0}^{\,\prime})^2
\eq
is decaying sufficiently fast, and $\varphi-\Psi_{01e}$  approaches a solution to the linear wave equation and has a radiation field $G_0$.
Hence  we can prescribe  an asymptotic expansion for \eqref{eq:wave:weak:inhom} of the form
\beqs
\varphi-\Psi_{01e}\sim\varphi_{01}
\eqs
where, as before,  the radiation fields $G_0$, $G_1$ for $\varphi_{01}$ satisfy the decay assumption \eqref{eq:decayassumption}.

\subsection{Weighted  energy bounds for the weak null system}

\begin{prop}\label{prop:weak:null:system}
  Let $F_0$, $G_0$ be radiation fields, $M>0$, and let $1/2<\gamma<1$.
  Then there exists a solution $(\psi=v+\psi_{01}+\psi_e,\,\varphi=w+\Psi_{01e}+\varphi_{01})$  to the system \eqref{eq:simple:weak:system}, where  $\psi_{01}$  is an approximate solution as in Section~\ref{sec:classical:null:revisited}, and $\varphi_{01}$ is an approximate homogeneous solution \eqref{eq:improvedlinearexpansionpsi01}.
  Moreover $\Psi_{01e}$ is defined in \eqref{eq:Psidef}.

  Then for any  $s<\gamma+1/2$, and  $k\geq 1$ such that the right hand side is finite, 
  \begin{gather}
\|  w(t,\cdot)\|_{k,s-1}\lesssim \frac{\|F_0\|_{k+7,\gamma-1/2}^2+M^2+\|G_0\|_{k+3,\gamma-1/2}}{\langle\, t\rangle^{1/2+\gamma-s}}\,,\\
| Z^I w(t,x)|\lesssim \frac{\|F_0\|_{|I|+9,\gamma-1/2}^2+M^2+\|G_0\|_{|I|+5,\gamma-1/2}}{\langle\, t+r\rangle\langle\, t-r\rangle^{s-1/2}}\,.
\end{gather}
Moreover for $v$ the conclusions of Proposition~\ref{prop:nullcond:revisited} hold true.

\end{prop}

We  only give the proof for $k=1$ since higher $k>1$ follow in the same way.

For any solution $\varphi=w+\Psi_{01e}+\varphi_{01}$ to \eqref{eq:wave:weak:inhom},
and $\psi=v+\psi_{01}+\psi_e$ to \eqref{eq:wave:weak:hom}, we have
\begin{equation}\label{eq:box:w}
  \begin{split}
\Box w=\Box\varphi-\Box\Psi_{01e} -\Box\varphi_{01}=&(\pa_t \psi)^2-(\pa_t \psi_{01e})^2 +\big( (\pa_t \psi_{01e})^2-(\psi_{01e}^{\,\prime})^2\big)-\Box\varphi_{01}\\
=&\big((\pa_t v)^2+2\pa_t\psi_{01e}\pa_t v\big)
+\big( ( \pa_t \psi_{01e})^2-(\psi_{01e}^{\,\prime})^2\big)-\Box\varphi_{01}
\end{split}
\end{equation}

Considering the terms on the right hand side, note first that
 by Lemma~\ref{lem:Blemma},
\begin{equation}\label{eq:box:w:rhs:second}
\big\| \langle t+r\rangle^s \big|(\pa_t \psi_{01e})^2-(\psi_{01e}^{\,\prime})^2\big|(t,\cdot)\big\|_{L^2_x}
\lesssim \frac{1}{\langle\, t\rangle^{3/2+\gamma-s}}\sum_{i,j=0,1}\Bigl(\|F_i\|_{1,\infty,\gamma-i}+M\Bigr)\|F_j\|_{0,\gamma-1/2-j}.
\end{equation}

Moreover by Lemma~\ref{lem:boxapprox}, c.f.~\eqref{eq:L:box:v}, we have the homogeneous estimate for the last term on the r.h.s.~of \eqref{eq:box:w},
\begin{equation}
\|\langle t+r\rangle^s\Box\varphi_{01}(t,\cdot)\|_{L^2_x}
\lesssim \frac{\|G_0\|_{4,\gamma-1/2}}{\langle\, t\rangle^{3/2+\gamma-s}}\,.
\end{equation}
Furthermore for the first term on the r.h.s.~of \eqref{eq:box:w} we can apply Lemma \ref{lem:thebilinearestimate}  to infer that
\begin{equation}
 \big\|\langle t+r\rangle^s \big((\pa_t v)^2\!+|\pa_t\psi_{01}\pa_t v|\big)(t,\cdot)\big\|_{L^2_x}
\!\lesssim \langle t\rangle^{s-1}\!\! \sum_{|I|\leq 2}\!
\big(\|\pa Z^I v(t,\cdot)\big\|_{L^2_x}\!+\|\pa Z^I \psi_{01e}(t,\cdot)\big\|_{L^2_x}\big)\|\pa  v(t,\cdot)\|_{L^2_x}.
\end{equation}

For the solution to \eqref{eq:wave:weak:hom}, with $\psi=v+\psi_{01e}$ we can now appeal directly to the results of Section~\ref{sec:classical:null:revisited}, see Prop.~\ref{prop:nullcond:revisited}, and also Prop.~\ref{prop:normbyenergy}, and Lemma~\ref{lem:approx:flux} to obtain that
\begin{equation}
\sum_{|I|\leq 2}\bigl(\|\pa Z^I v(t,\cdot)\|+\|\pa Z^I \psi_{01e}(t,\cdot)\|\bigr)\lesssim D\qquad D:=\|F_0\|_{8,\gamma-1/2}+M
\end{equation}
and hence
\begin{equation}
  \big\|\langle t+r\rangle^s \big((\pa_t v)^2\!+|\pa_t\psi_{01}\pa_t v|\big)(t,\cdot)\big\|_{L^2_x}\lesssim  \frac{D^2}{\langle\, t\rangle^{3/2+\gamma-s}}\,.
\end{equation}


In conclusion,
\begin{equation}
\int_{t}^{T}\big\|\langle t+r\rangle^s |\Box w(t,\cdot)|\big\|_{L^2_x}  \ud t
\lesssim \frac{\|F_0\|_{8,\gamma-1/2}^2+\|G_0\|_{4,\gamma-1/2}+M^2}{\langle\, t\rangle^{1/2+\gamma-s}},\quad\text{if}\quad
0\leq s<1/2+\gamma .
\end{equation}

Now consider a sequence of solutions $w_T$ to \eqref{eq:box:w} with trivial data at $t=T$. Then it follows by Proposition~\ref{prop:frac:morawetz} that
\beq
\|w_T(t,\cdot)\|_{1,+,s-1}\lesssim \frac{\|G_0\|_{4,\gamma-1/2}+D^2}{\langle\, t\rangle^{1/2+\gamma-s}},\quad\text{if}\quad
1\leq s<\gamma+1/2\,,
\eq
and from the energy identity,
\beq
\|\pa w_T(t,\cdot)\|
\lesssim \frac{\|G_0\|_{4,\gamma-1/2}+D^2}{\langle\, t\rangle^{1/2+\gamma}}\,.
\eq
This implies the existence of the limit $\varphi=\lim\varphi_T$, $\varphi_T=w_T+\Psi_{01e}+\varphi_{01}$ as in Section~\ref{sec:existence:limit}.

Finally as in the linear case we get by Lemma~\ref{lemma:KS:weights} the pointwise decay estimates.

\subsection{Existence of the limit as $T\to \infty$} \label{sec:existence:limit}
We can  reduce the argument to the linear homogeneous case as in Section~\ref{sec:existence:homogeneous}.
First in view of Section~\ref{sec:classical:null:revisited} the limit of $v_T$ and hence $\psi_T$ exist as $T\to \infty$, where $T$ is the time when data for $v_T$ vanish. Then one lets $\psi$ be the fixed limiting solution and with $\psi$ given solve the equation for $\varphi$ with data for $\varphi$ at $T^\prime$ and hence data for $w$ vanishing at $T^\prime$. As in previous section let $T^\prime_2>T^\prime_1$ and let $\varphi_i $ and $w_i$
be the corresponding solutions and let $w=w_2-w_1$. Then again $w=w_2$ when $t=T_1^\prime$ and
$\Box w=0$. The rest of the argument is identical to Section~\ref{sec:existence:homogeneous}.

\subsection{Asymptotics of the forward solution with sources}\label{sec:sources}

In this section we discuss estimates and formulas for the auxiliary solution of \eqref{eq:Psidef}.

For $k=2,3,4$ let
$\Phi^k[n]$ be the solution of 
\beq
-\Box \Phi^k[n](t,r\omega)=n(r-t,\omega)r^{-k}\chi\big(\tfrac{\langle\,r-t\,\rangle}{r}\big)^2,
\eq
with vanishing data at $-\infty$, where we assume that for some $a\geq 0$
\beqs
\|n\|_{N,1,\infty,a}= {\sum}_{|\alpha|+k\leq N}\,\,
\int_{R}\sup_{\omega\in S^2 } |\,(\langle q\rangle\pa_{q})^k\pa_\omega^\alpha n({q},\omega)| \langle q_+\rangle^{a} \ud q\leq C .
 \eqs
 \begin{lemma} For $|I|\leq N$
\beq
|Z^I \Phi^2[n](t,r\omega)|\lesssim \frac{1}{2r}\ln{\Big(\frac{\langle \, t+r\rangle}{\langle\, t-r\rangle}\Big)} \frac{\|n\|_{N,1,\infty,a}}{\langle (r-t)_+\rangle^{a}}
\eq
and for $k=3,4$
\beq
|Z^I \Phi^k[n](t,r\omega)|
\lesssim \frac{1}{\langle \, t+r\rangle \langle\, t-r\rangle^{k-2}}
 \frac{\|n\|_{N,1,\infty,a}}{\langle (r-t)_+\rangle^{a}}.
\eq
\end{lemma}
\begin{proof} The proof follows by first taking the supremum over the angles and then using the formula
for the fundamental solution in the radial case as in the proof of Lemma 9 in \cite{L17}.
\end{proof}

Moreover, we recall from  \cite{L90a,L17} the formula for the solution
 \begin{equation}\label{eq:Phi:2}
\Phi^2[n](t,r\omega)=\int_{r-t}^{\infty} \frac{1}{4\pi}\int_{\bold{S}^2}{\frac{
 n({q},{\sigma})\chi\big(\tfrac{\langle\,{q}\,\rangle}{\rho}\big)  dS({\sigma})d {q}}{t-r+{q}+r\big(1-\langle\, \omega,{\sigma}\rangle\big)}\,
}\,,
\end{equation}
where
\beq
\rho=\frac{1}{2}\, \frac{(t+r+q)(t-r+q)}{t-r+{q}+r\big(1-\langle\, \omega,{\sigma}\rangle\big)}\,.
\eq
Similar expressions can be derived for $\Phi^3[n]$, and $\Phi^4[n]$:
\beq
\Phi^3[n](t,r\omega)=\int_{r-t}^{\infty} \frac{1}{4\pi}\int_{\bold{S}^2}{\frac{
 n({q},{\sigma})\chi\big(\tfrac{\langle\,{q}\,\rangle}{\rho}\big) dS({\sigma})}{(t-r+{q})(t+r+q)}\,
}\, d {q},
\eq
and
\beq
\Phi^4[n](t,r\omega)=\int_{r-t}^{\infty} \frac{1}{4\pi}\int_{\bold{S}^2}n({q},{\sigma}){\frac{
 \chi\big(\tfrac{\langle\,{q}\,\rangle}{\rho}\big)\big(t-r+{q}+r\big(1-\langle\, \omega,{\sigma}\rangle\big) }{(t-r+{q})^2(t+r+q)^2}\,
}\, dS({\sigma}) d {q}.
\eq


\input{ScatteringNullPaper.bbl}

\bigskip
{\itshape Address:} {\scshape\footnotesize Johns Hopkins University, 404 Krieger Hall, 3400 N.~Charles Street, Baltimore, MD 21218, US}\\
{\itshape Email:} {\ttfamily lindblad@math.jhu.edu}

\smallskip
{\itshape Address:} {\scshape\footnotesize University of Melbourne, Parkville, VIC, 3010, Australia}\\
{\itshape Email:} {\ttfamily volker.schlue@unimelb.edu.au}

\end{document}

%% file: ScatteringNullPaper.bbl
\providecommand{\bysame}{\leavevmode\hbox to3em{\hrulefill}\thinspace}
\providecommand{\MR}{\relax\ifhmode\unskip\space\fi MR }
\providecommand{\MRhref}[2]{%
  \href{http://www.ams.org/mathscinet-getitem?mr=#1}{#2}
}
\providecommand{\href}[2]{#2}